\documentclass[11 pt]{amsart}

\usepackage{amsmath,amssymb,amsthm,amsfonts,setspace,hyperref,dsfont,yhmath,euscript}

\usepackage{hyperref}\hypersetup{colorlinks=true, citecolor=green, linkcolor=blue, hypertexnames=false}

\usepackage[letterpaper,margin=1.4in]{geometry}

\usepackage{color}
\usepackage[utf8]{inputenc}

\makeatletter
\newcommand{\namedlabel}[2]{%
  \begingroup
  \def\@currentlabel{#2}
  \label{#1}%
  \endgroup
}
\makeatother

\newcommand{\NN}{\mathbb{N}}
\newcommand{\RR}{\mathbb{R}}
\newcommand{\R}{\mathbb{R}}

\newcommand{\CC}{\mathbb{C}}

\newcommand{\ZZ}{\mathbb{Z}}

\newcommand{\norm}[1]{\lVert#1\rVert}
\newcommand{\abs}[1]{\lvert#1\rvert}

\newtheorem{theorem}{Theorem}[section]
\newtheorem{corollary}[theorem]{Corollary}

\newtheorem*{theorem*}{Theorem}
\newtheorem*{conjecture*}{Conjecture}
\newtheorem{lemma}[theorem]{Lemma}
\newtheorem{proposition}[theorem]{Proposition}

\newcommand{\comment}[1]{}
\theoremstyle{definition}

\newtheorem{remark}[theorem]{Remark}

\newtheorem{cect6}{}

\newtheorem{sect}{}

\numberwithin{equation}{section}

\newcommand \re {{%
\operatorname{Re}
}}

\DeclareMathOperator{\Ad}{Ad}

\def\Lie{\mathrm{Lie}}
\makeatletter
\renewcommand*\env@matrix[1][*\c@MaxMatrixCols c]{%
  \hskip -\arraycolsep
  \let\@ifnextchar\new@ifnextchar
  \array{#1}}
\makeatother

\begin{document}
\title[Uniform bounds of correlations]
{Multiple mixing and multiple fractional cohomological equation: semisimple setting}

\thanks{ $^1$ Based on research supported by NSF grant DMS-2452194}


\author[]{ Zhenqi Jenny Wang$^1$ }

\address{Department of Mathematics\\
        Michigan State University\\
        East Lansing, MI 48824,USA}
\email{wangzq@math.msu.edu}

\begin{abstract}

The purpose of this paper is to develop a new effective approach to
higher-order mixing in the semisimple setting. We prove effective exponential mixing of all orders for partially hyperbolic algebraic actions, under a strong spectral-gap assumption. The decay rates are explicit in the Lyapunov and spectral-gap data, and the required Sobolev orders are explicit. Already at order two, our estimates require only partial Sobolev/H\"older regularity along weak stable and unstable subgroup directions, with no transverse derivatives. For representations admitting better-than-tempered decay, the resulting order-two estimate attains the optimal matrix-coefficient exponent.

The proof introduces a new fractional-cohomological method in the semisimple setting. The central analytic input is a solvability theory for multiple fractional cohomological equations of Type~$II$ (sum-of-product type). These equations are solvable in a cohomology-free range governed by the spectral behavior near the edge \(0\), and the solutions satisfy estimates in partial Sobolev norms. This mechanism converts fractional solvability into order-two decay of correlations under partial regularity, and then into effective higher-order mixing, yielding a quantitative form of Rokhlin's multiple-mixing problem.


\end{abstract}


\maketitle

\setcounter{tocdepth}{2}

\tableofcontents

\section{Introduction and main results}

Quantitative mixing for semisimple actions has a long history. Order-two decay is closely tied to decay of matrix coefficients, and is classically studied through representation-theoretic methods based on $K$-finite vectors, $K$-type expansions, and spherical analysis. By contrast, quantitative higher-order mixing is much less understood. Even when exponential mixing of order two is known, effective higher-order rates and effective regularity have remained out of reach in the semisimple setting. Moreover, for representations with
better-than-tempered decay, the classical framework does not yield
quantitative estimates for smooth vectors at the optimal
matrix-coefficient exponent.

The present paper develops a different approach to these problems. Rather than starting from $K$-type analysis, we introduce a new analytic mechanism for quantitative mixing in the semisimple setting. We prove effective exponential mixing of all orders for partially hyperbolic algebraic actions of semisimple groups under a strong spectral-gap assumption. The decay rates are explicit in the Lyapunov and spectral-gap data, and the required Sobolev orders are explicit. Already at order two, our estimates require only partial Sobolev/H\"{o}lder regularity along weak stable and unstable subgroup directions, with no transverse derivatives. For representations admitting better-than-tempered decay, the resulting order-two estimate attains the optimal matrix-coefficient exponent.

The key new ingredient is a fractional-cohomological method for semisimple unitary representations. In the semisimple setting, the classical cohomological equation along nilpotent directions is generally obstructed. We show that, after passing to suitable fractional operators, a nontrivial cohomology-free range emerges, governed sharply by the spectral behavior near the edge 0. The corresponding solutions satisfy estimates in partial Sobolev norms, and this fractional solvability is the analytic mechanism that first yields quantitative order-two decay and then propagates to effective higher-order mixing. In this way, the paper builds a new bridge from spectral gap to decay of correlations.

\subsection{Setting and mixing} Let $G$ be a connected semisimple Lie group of non-compact type with finite center and $A$ be a split Cartan subgroup of $G$. Fix a positive Weyl chamber $A^+$ of $A$. Let $K$ be a maximal compact subgroup of $G$ such
that the Cartan decomposition $G=KA^+K$ holds.

Let $\mathcal{Z}\subseteq G$ be a closed abelian subgroup isomorphic to $\ZZ^m\times \RR^l$.  We consider a measure-preserving action $\alpha$ of $G$ on a probability space $(\mathcal{X},\,\varrho)$, where $\mathcal X$ is a $C^\infty$ manifold endowed with a $G$-invariant Borel probability measure $\varrho$. Let $L^2_0(\mathcal X)
=
\left\{f\in L^2(\mathcal X):\int_{\mathcal X}f\,d\varrho=0\right\}$. We assume that the induced unitary representation
of \(G\) on \(L^2_0(\mathcal X)\) has a strong spectral gap.

We write $\alpha_{\mathcal Z}$ for the restriction of $\alpha$ to $\mathcal Z$.
 Given \( n \ge 2 \), we say that $\alpha_\mathcal{Z}$ is:
\begin{itemize}
  \item \emph{\(n\)-mixing} if for every $f_1,\cdots, f_n\in L^\infty(\mathcal{X})$ and every $z_1,\cdots, z_n\in \mathcal{Z}$,  we have
\[
\int_{\mathcal{X}} \Pi_{i=1}^{n} f_i\big(\alpha_\mathcal{Z}(z_i)x\big) \, d\varrho(x) \longrightarrow \Pi_{i=1}^{n} \int_{\mathcal{X}} f_i(x) \, d\varrho(x)
\]
as $\min_{i \neq j} \norm{z_iz_j^{-1}} \to +\infty$.

\smallskip
  \item \emph{\(n\)-exponential-mixing with rate $\eta$} if there exist $\eta,\,s,\,C>0$ such that for every $f_1,\cdots, f_n\in C_c^\infty(\mathcal{X})$ and every $z_1,\cdots, z_n\in \mathcal{Z}$,  we have
  \begin{align*}
   &\Big|\int_{\mathcal{X}} \Pi_{i=1}^{n} f_i\big(\alpha_\mathcal{Z}(z_i)x\big) \, d\varrho(x) - \Pi_{i=1}^{n} \int_{\mathcal{X}} f_i(x) \, d\varrho(x)\Big|\\
   &\leq Ce^{-\eta\min_{i \neq j} \norm{z_iz_j^{-1}}}\Pi_{i=1}^{n} \|f_i\|_s
  \end{align*}
  where $\|f_i\|_s$ denotes the Sobolev
norm of order $s$.

\end{itemize}
We say that mixing rate $\eta$  is \textbf{\emph{effective}} if it is explicitly determined by the Lyapunov exponents of $\alpha_\mathcal{Z}$,  the spectral gap and $n$.
We say that the Sobolev order $s$ is \textbf{\emph{effective}} if $s$ is explicitly determined by $\eta$.

\subsection{Main results} We develop a new dynamical and analytic method for quantitative mixing in the semisimple setting. The principal result of the paper is an effective order-$2$ mixing theorem with explicit decay rates and explicit partial-Sobolev regularity. Quantitative mixing of all orders is then obtained as an application of the order-$2$ theorem. The key analytic input is a new solvability theory for multiple fractional cohomological equations of Type~$II$.

\subsubsection{Effective order-two decay} Let $(\pi, \mathcal{H})$ be a unitary representation of $G$ with a strong spectral gap, i.e., the restriction of $\pi$ to each simple factor of $G$ is isolated
from the trivial representation with respect to the Fell topology. For \(a\in A\), let
\[
W_{-,a}=\{Y\in\text{Lie}(G):\operatorname{Ad}(a^n)Y\to0
\text{ as }n\to+\infty\},
\]
\[
W_{+,a}=\{Y\in\text{Lie}(G):\operatorname{Ad}(a^{-n})Y\to0
\text{ as }n\to+\infty\},
\]
and let \(W_{0,a}\) be the neutral subspace.  Denote by
\(H_{-,a}\), \(H_{+,a}\), \(H_{-0,a}\), and \(H_{+0,a}\) the connected
subgroups with Lie algebras $W_{-,a}$, $W_{+,a}$, $W_{-,a}\oplus W_{0,a}$,
$W_{+,a}\oplus W_{0,a}$,
respectively.

For a closed subgroup \(H\leq G\), the norm
\(\|\cdot\|_{H,s}\) denotes the order-\(s\) Sobolev norm defined using
only derivatives from \(\Lie(H)\); the precise definition is given in
Section~\ref{sec:9}.
\begin{theorem}\label{th:4}
There exist a function \(\eta:A\to\mathbb R_{\ge0}\) and an explicit Sobolev order \(s=s(\eta)>0\)
such that the following holds.
If \(m\geq0\), then
\[
|\langle \pi(a^m)\psi,\xi\rangle|
\leq
C_\eta e^{-m\eta(a)}
\|\psi\|_{H_{+,a},s}\|\xi\|_{H_{-0,a},s}.
\]
If \(m\leq0\), then
\[
|\langle \pi(a^m)\psi,\xi\rangle|
\leq
C_\eta e^{m\eta(a)}
\|\psi\|_{H_{-,a},s}\|\xi\|_{H_{+0,a},s}.
\]
\end{theorem}
The precise statement is given in Theorem \ref{th:7}.

\begin{remark}\emph{Decay rate.} The estimate has two features which, to our knowledge, were not
previously available in the semisimple setting. The function \(\eta\) is explicit and depends only on the spectral gap of \(\pi\). When \(G\) is higher-rank simple, or \(G=Sp(1,n)\), the resulting uniform lower bound matches the optimal rate obtained by Oh for \(K\)-finite vectors \cite{oh}. For representations admitting better-than-tempered decay, Theorem \ref{th:4} gives an effective order-$2$ estimate at the optimal rate.

\emph{Regularity.} The required Sobolev order \(s(\eta)\) is explicit. Moreover, in contrast with previous approaches, Theorem \ref{th:4} requires only \emph{partial} Sobolev/H\"older regularity along suitable stable/unstable directions, with no transverse derivatives.

\end{remark}
\begin{remark}Order-$2$ mixing is a basic input in many problems of smooth rigidity. From this viewpoint, the partial-Sobolev form in Theorem \ref{th:4} is especially significant: it provides quantitative decay estimates that require regularity only along dynamically relevant stable and unstable directions, without imposing transverse derivatives. We expect this directional form of the estimate to be useful in future rigidity applications.
\end{remark}
\subsubsection{Effective higher-order mixing} We next give two higher-order consequences of Theorem \ref{th:4};  equivalently,
we address the  \emph{quantitative} Rokhlin multiple-mixing problem for the
semisimple partially hyperbolic algebraic actions. The fact that exponential mixing requires only partial smoothness of the test functions  plays a crucial role in the proof.

We first state the rank-one consequence. Let $\eta(a)$ and $s(\eta)$ be as described in Theorem \ref{th:4}, and assume we are in the
rank-one situation \(\mathcal Z\simeq \mathbb Z\), generated by \(a\).
\begin{theorem}\label{th:12}  Then for any $n\geq2$ and any $f_1,\cdots, f_n\in C_c^\infty(\mathcal{X})$ and any $z_1,\cdots, z_n\in\ZZ$ we have
\begin{align*}
   &\Big|\int_{\mathcal{X}} \Pi_{i=1}^{n} f_i\big(\alpha(a^{z_i})x\big) \, d\varrho(x) - \Pi_{i=1}^{n} \int_{\mathcal{X}} f_i(x) \, d\varrho(x)\Big|\\
   &\leq C_{n,\eta}e^{-\eta\min_{i \neq j} |z_i-z_j|}\Pi_{i=1}^{n} \norm{f_i}_{C^s}.
  \end{align*}

\end{theorem}
The precise statement is given in Theorem \ref{th:8}.

\begin{remark} In rank one, the higher-order estimate has the same exponential rate
and the same Sobolev order as the order-two estimate; neither
deteriorates with the order of correlation. For test functions with zero average, partial Sobolev norms are sufficient for exponential mixing of all orders (see \eqref{for:207} of Theorem \ref{th:8}). In particular, the optimal order-$2$ rate propagates to all orders. Moreover, for triple correlations with zero average, the relevant time separation is governed by the maximal pairwise gap (\eqref{for:208} of Theorem \ref{th:8}),  while for \(n\ge4\) no analogous maximal-gap estimate can hold uniformly in general.
We are not aware of similar results in the semisimple settings.

Thus, in rank one, higher-order mixing has no additional exponential
cost beyond order two: the order-two exponent and the order-two Sobolev
regularity propagate to correlations of all orders.  This phenomenon is
not visible from existing global-Sobolev approaches to higher-order
mixing, such as \cite{BEG}. 

\end{remark}

We next state the higher-rank consequence.   Let $|\mathcal{S}|$ be a positive number depending only  on $G$ and let $\eta(a)$ and $s(\eta)$ be as described in Theorem \ref{th:4}.

\begin{theorem}\label{th:16}  For any $n\geq2$, any $f_1,\cdots, f_n\in C_c^\infty(\mathcal{X})$ and any $a_i\in A$, $1\leq i\leq n$, we have
\begin{align*}
 &\Big|\int_{\mathcal{X}} \Pi_{i=1}^{n} f_i\big(\alpha(a_i)x\big) \, d\varrho(x) - \Pi_{i=1}^{n} \int_{\mathcal{X}} f_i(x) \, d\varrho(x)\Big|\notag\\
&\leq C_{n,\eta} \max_{1\le i\neq j\le n} e^{-\frac{\eta(a_ia_j^{-1})}{(n-1)|\mathcal{S}|}}\,\Pi_{i=1}^n\norm{f_i}_{C^{s}}.
\end{align*}

\end{theorem}
The precise statement is given in Theorems \ref{th:3}.

 \begin{remark} Theorem \ref{th:16} provides effective higher-order mixing with explicit dependence on the order \(n\). In particular, the higher-rank result applies beyond the homogeneous setting and yields explicit higher-order rates and explicit regularity. Previous
order-two-to-all-order approaches, such as \cite{BEG},  in the semisimple setting do not seem
to provide this simultaneous effectiveness of the rate and regularity
in this generality.

\end{remark}

\subsection{The fractional-cohomological mechanism} We now describe the analytic mechanism underlying the preceding theorems. The main new ingredient of the paper is a fractional-cohomological mechanism for semisimple settings.

\subsubsection{Multiple fractional cohomological equation}\label{sec:13}

 The classical cohomological equation
\begin{align}\label{for:10}
X\xi=d\rho(X)\xi=\omega
\end{align}
is generally obstructed in the semisimple setting. Our basic idea is to replace the classical infinitesimal operators by suitable fractional operators attached to abelian subgroup actions. This weakens the singular behavior at the spectral edge \(0\) and yields a nontrivial solvable range.

We first recall the relevant fractional operators. Let $S$ be a Lie group and $(\rho,\mathcal R)$ a unitary representation of $S$. Let $\mathcal A\le S$ be an abelian Lie subgroup isomorphic to $\RR^m$, and let $X_1,\dots,X_m\in \Lie(\mathcal A)$ be a basis of $\Lie(\mathcal A)$. Consider the restricted representation $\rho|_{\mathcal A}$.

By spectral theory, there exists a regular Borel measure \(\sigma\) on \(\widehat{\RR^m}\) such that every vector \(\xi\in\mathcal{R}\)
admits a decomposition $\xi=\int_{\widehat{\RR^m}}\xi_\chi\,d\sigma(\chi)$, and for every \(t=(t_1,\dots,t_m)\in\RR^m\),
\[
\rho(\exp(\sum_{i=1}^m t_iX_i))\xi=\int_{\widehat{\RR^m}}\chi(t)\,\xi_\chi\,d\sigma(\chi),
\]
where \(\chi(t)=e^{\mathrm{i}\chi \cdot t}\) (here we identify \(\RR^m\) with \(\widehat{\RR^m}\)).

For $r=(r_1,\cdots,r_m)\in (\RR^+)^m$ and $X=(X_1,\cdots,X_m)$, we define the associated \emph{fractional operator} by
\begin{align*}
  |X|^{r}(\xi):=|X_1|^{r_1}\cdots |X_m|^{r_m}(\xi):=\int_{\widehat{\RR^m}}|\chi_1|^{r_1}\cdots |\chi_m|^{r_m}\xi_\chi d\sigma(\chi).
\end{align*}
This is a positive self-adjoint operator with domain
\[
\mathrm{Dom}(|X|^{\,r})=\Bigl\{\xi\in\mathcal R:\ \int_{\mathbb R^m} |\chi_1|^{2r_1}\cdots |\chi_m|^{2r_m}\|\xi_\chi\|^2d\sigma(\chi)<\infty\Bigr\}.
\]
In the direct-integral model,
\[
\bigl(|X|^{\,r}\xi\bigr)_\chi = |\chi_1|^{r_1}\cdots |\chi_m|^{r_m}\,\xi_\chi\quad\text{for $\sigma$-a.e.\ }\chi\in\RR^m.
\]

\begin{remark} When $m=1$ and \(r\in\NN\), the operator $|X|^r$ is closely related to, but different from, the usual Lie derivative. Writing $d\rho(X)$ for the infinitesimal generator, one has
\[
d\rho(X)^r(\xi)
=
\int_{\widehat{\RR}} (i\chi)^r\,\xi_\chi\,d\sigma(\chi),
\]
so $|X|^r$ differs from $d\rho(X)^r$ by replacing $(i\chi)^r$ with $|\chi|^r$.

 When $S=\RR$ and $\rho$ is the left-regular representation on $L^2(\RR)$, the operator $|X|^r$ is the Fourier multiplier
\[
f\longmapsto \mathcal F^{-1}\bigl(|\chi|^r\widehat f\,\bigr),
\]
which coincides, up to normalization, with the Riesz fractional derivative.
\end{remark}
We now define the class of equations used in the semisimple setting. Let $\mathcal A_i\le S$, $1\le i\le m$,
be abelian Lie subgroups, with $\mathfrak u_{i,1},\dots,\mathfrak u_{i,l_i}\in\Lie(\mathcal A_i)$ a basis of \(\Lie(\mathcal A_i)\). For exponents \(r_{i,j}>0\), set
\[
\mathfrak u_i=(\mathfrak u_{i,1},\dots,\mathfrak u_{i,l_i}),
\qquad
r_i=(r_{i,1},\dots,r_{i,l_i}),
\]
and write $|\mathfrak u_i|^{r_i}:=
|\mathfrak u_{i,1}|^{r_{i,1}}\cdots |\mathfrak u_{i,l_i}|^{r_{i,l_i}}$.

Given \(\omega\in\mathcal R\), we study the following \emph{multiple fractional cohomological equation of Type~II (sum of product type)}:
\[
\sum_{i=1}^m |\mathfrak u_{i}|^{r_{i}}\,\xi_{i}=\omega.
\]
We say that \(\omega\) is a \(\{\mathfrak u_{i,j};r_{i,j}\}\)-coboundary if there exists \(\xi_i\in\mathcal R\), $1\le i\le m$ satisfying the above equation.

\begin{remark} If $l_i=1$ for each $1\leq i\leq m$, then Type~$II$ reduces to the multiple fractional cohomological equation of Type~$I$ (sum type), studied in \cite{W2} in the nilmanifold setting.

\end{remark}
Such multiple fractional equations of sum-of-product type, defined via the spectral theory of abelian subgroups, are new in the representation-theoretic setting of possibly non-abelian Lie groups.

\subsubsection{Fractional solvability in the semisimple setting}
The key new analytic input behind Theorem \ref{th:4} is a solvability theory for multiple fractional cohomological equations of Type~$II$.
The theorem below shows that, after passing to suitable multiple fractional operators, a cohomology-free range of exponents emerges.
Moreover, the corresponding solutions satisfy \emph{partial} Sobolev estimates, involving only derivatives along suitable stable/unstable directions. This fractional solvability is the analytic mechanism behind all of our quantitative mixing results.

\begin{theorem}\label{th:19} Suppose \(S\) is a Lie group with
\[
\Lie(S)\simeq \mathfrak{sl}(2,k_1)\oplus\cdots\oplus \mathfrak{sl}(2,k_n),
\qquad k_i\in\{\RR,\CC\},
\]
and let \(S_i\le S\) be the corresponding rank-one subgroups. Let \((\beta,\mathcal L)\) be a unitary representation of \(S\), and assume that each restriction \(\beta|_{S_i}\) has strong spectral gap. Then there exist a subgroup \(H\le S_1S_2\cdots S_n\) and thresholds \(\gamma_i>0\), depending only on the spectral gap of $\beta|_{S_i}$, such that
for every choice of exponents \(0\le r_i<\gamma_i\), \(1\le i\le n\), every sufficiently smooth vector \(\xi\in\mathcal L\) admits a finite decomposition $\xi=\sum_{\lambda} \xi_\lambda$ such that, for each \(\lambda\), the fractional equation
\[
|\Lambda_{1,\lambda}|^{r_1}\cdots |\Lambda_{n,\lambda}|^{r_n}\,\omega_\lambda=\xi_\lambda,
\]
has a solution \(\omega_\lambda\in\mathcal L\), where each \(\Lambda_{i,\lambda}\) is one of the nilpotent directions associated with the \(i\)-th rank-one factor. Equivalently,
\[
\sum_\lambda
|\Lambda_{1,\lambda}|^{r_1}\cdots
|\Lambda_{n,\lambda}|^{r_n}\omega_\lambda
=
\xi.
\]
Thus \(\xi\) is a Type~II fractional coboundary.

Moreover, the solutions satisfy estimates of the form
\[
\|\omega_\lambda\|
\le C_{\mathfrak r,\mathfrak p}\,\|\xi\|_{H,s},
\]
where \(s\) depends explicitly on the spectral-gap data.

Finally, the thresholds \(\gamma_i\) are sharp: if \(r_i>\gamma_i\) for at least one index \(i\), then there exist smooth vectors \(\xi\) for which the corresponding fractional equations are not solvable in \(\mathcal L\).

\end{theorem}
The precise statement is given in Theorem \ref{th:6}.

\begin{remark}\label{re:3}
For \(G=SL(2,\RR)\), obstructions to the classical cohomological equation over the horocycle flow are classified in \cite{Forni}. In particular, nilpotent directions are not cohomology free in the classical sense. Theorem~\ref{th:19} shows that this changes at the fractional level: although the classical equation is obstructed, a nontrivial solvable range appears for the corresponding fractional equations. Such a theory does not appear to have been previously developed in semisimple unitary representation theory. 
\end{remark}

\subsection{Background and comparison with previous results}

Mixing is a basic manifestation of randomness in dynamical systems: correlations decay, and the system asymptotically forgets its initial state. In the semisimple setting, quantitative mixing of order \(2\) is closely related to decay of matrix coefficients of unitary representations.

\subsubsection{History} \emph{Qualitative mixing}:
Building on work of Dani \cite{Da1}, \cite{Da2}, algebraic partially hyperbolic one-parameter flows are mixing of all orders under mild assumptions.
 Using Ratner's measure classification, Starkov in \cite{Star} proved mixing
of all orders for general mixing one-parameter flows on finite-volume homogeneous spaces.

\emph{Quantitative mixing of order two}: Quantitative (exponential) mixing of order two rests on decay estimates for matrix coefficients of (unitary) representations of semisimple Lie groups. Building on Harish-Chandra's research programme (surveyed in the
monographs \cite{Warner}, \cite{Warner1}), \emph{effective} bounds for $K$-finite vectors were developed in work of Borel-Wallach \cite{Borel}, Cowling \cite{Cowling}, Howe \cite{Howe}, Li-Zhu \cite{Li}, \cite{Li1}, Moore \cite{Moore}, and Oh \cite{Oh1}, \cite{oh};  in particular, the optimal rate for higher rank groups, as well as for $Sp(1,n)$ was obtained by Oh~\cite{oh}. These works proceed by first establishing sharp bounds for matrix coefficients restricted to embedded rank-one subgroups that are locally isomorphic to $SL(2,k)$, $k=\RR$ or $\CC$, and then propagate these bounds to general $K$-finite vectors via the $K$-type decomposition (Howe's strategy). Katok and Spatzier \cite{Spatzier1} extended such decay estimates to smooth vectors.

\emph{Quantitative mixing of order $\ge 3$}:   If the phase space is homogeneous, then under a strong spectral-gap assumption, Bj{\"o}rklund, Einsiedler, and Gorodnik~\cite{BEG} recently showed that quantitative exponential mixing of order two implies quantitative exponential mixing of all higher orders, thereby answering a quantitative version of Rokhlin's multiple-mixing problem. We also mention recent work of Konstantoulas \cite{Konstantoulas}, which gives explicit multi-correlation estimates along split Cartan directions, but does not yield the usual exponential mixing bounds for all smooth observables. However, despite the extensive history of the subject, obtaining \emph{effective} exponential rates together with \emph{effective} regularity remains a central challenge in the semisimple setting.

\subsubsection{What is new in the present paper} Our results differ from previous approaches in three ways.

First, we obtain effective regularity already at order two, and in a partial-Sobolev form. Classical representation-theoretic
proofs work with $K$-finite vectors and use precise bounds on $K$-types
(dimensions and spherical-function asymptotics), which does
not keep track of Sobolev regularity of the observables. The extension from $K$-finite to smooth vectors then requires very
high (and non-quantified) Sobolev order; see, for instance, \cite{Spatzier1}.

Second, we prove quantitative bounds beyond the tempered exponent. In the classical representation-theoretic framework, the available estimates bound the decay rate by some
exponent that is at most the tempered exponent. For representations with strictly faster decay, such as the discrete series of \(SL(2,\RR)\), quantitative upper bounds at the optimal decay rate are not available in this framework.

Third, we obtain effective higher-order rates and effective higher-order regularity. In the semisimple setting, effective rates are known for order~$2$ (via matrix coefficients), but effective higher-order mixing remains out of reach. In particular, the approach of \cite{BEG} combines coupling with Sobolev embedding, product estimates, and translation bounds. The Sobolev step prevents both the higher-order decay rates and the required Sobolev orders on the test functions from being \emph{effective}.

\subsection{Comparison with previous methods}  Historically, quantitative mixing in the nilmanifold and semisimple settings has developed along rather different lines. In the nilmanifold case, one often reduces the problem to quantitative equidistribution of rational submanifolds, with essential input from number theory. In the semisimple case, the study of mixing is closely tied to decay of matrix coefficients and has been developed largely through representation-theoretic methods. The present paper, together with \cite{W2}, shows that both settings also admit a common fractional-cohomological perspective. At the same time, the analytic mechanism is genuinely different in the two cases, and the semisimple arguments developed here are new.

\emph{Comparison with Forni's mechanism for mixing}
The approach closest in spirit to ours is Forni’s use of cohomological equations to study mixing \cite{Forni1}. In settings where the spaces of coboundaries and invariant distributions can be described explicitly, this strategy yields powerful information on decay of correlations by deducing information about Ruelle resonances and their
asymptotics for the action.   However, outside a few highly explicit models, such a description is generally unavailable, and the method does not readily extend to quantitative higher-order mixing.

\emph{Comparison with Flaminio-Forni's method for the classical equation} Flaminio and Forni developed a representation-theoretic method for the study of the cohomological equation
 $U\xi=\omega$ over the horocycle flow $\{\exp(tU)\}_{t\in\RR}$ in the $SL(2,\RR)$ setting \cite{Forni}. In that framework, the classical equation \(U\xi=\omega\) is analyzed using the explicit \(K\)-type model of irreducible representations: the \(K\)-eigenvectors form an orthogonal basis, the Laplacian acts by a scalar on each \(K\)-eigenvector, and the action of \(U\) has an explicit expansion in that basis. This makes it possible to classify invariant distributions, construct solutions via the Green operator, and estimate Sobolev norms.

The fractional problem studied here is fundamentally different. First, the obstruction to solvability is no longer described by invariant distributions, but by the spectral behavior near the edge \(\chi=0\). Second, the Laplacian-based method naturally yields \emph{full} Sobolev estimates, whereas our application to mixing requires \emph{partial} Sobolev control along stable/unstable subgroup directions. Third, for the fractional operator \( |U|^r \), no analogue of the explicit \(K\)-type formula for \(U\) is available. For these reasons, the treatment of the fractional equation in the present paper requires a genuinely new mechanism.

\subsection{Proof strategy} In \cite{W2}, a new scheme was introduced to study mixing for partially hyperbolic automorphisms on nilmanifolds: one first solves suitable fractional cohomological equations of Type~$I$ (sum type) with partial regularity, then converts this solvability into order-$2$ decay, and finally bootstraps order-$2$ decay to higher-order mixing by a two-block decomposition.

The general philosophy introduced in \cite{W2} survives in the present paper, but in the semisimple setting both main analytic ingredients--solving the fractional equations and obtaining partial Sobolev estimates--are fundamentally different. In the nilmanifold setting, the relevant fractional equations are solved by direct construction using Fourier analysis and Kirillov theory, and the partial Sobolev estimates then follow from direct computation. By contrast, in the semisimple setting such a direct route is not available even for \(SL(2,\RR)\). Thus the treatment of the fractional equation in the semisimple case is not a variant of the nilmanifold argument, but requires a genuinely new mechanism.

\textbf{\emph{Study of the fractional equations in $SL(2,k)$}}  Suppose \((\pi_\mu,\mathcal H_\mu)\) is an irreducible, nontrivial, unitary representation  of \(SL(2,k)\), where \(k=\RR\) or \(\CC\), and \(\mu\) is determined by the Casimir parameter(s). The lack of an explicit \(K\)-type formula for the fractional operator \( |U|^r \) is a major obstacle to the Flaminio--Forni approach. We therefore introduce a new dynamical mechanism tailored to the horocycle flow.

First, we establish sharp exponential decay estimates for matrix
coefficients along the geodesic flow, with the representation-specific
optimal rate. This already goes beyond the classical \(K\)-finite-vector
framework in representation theory when better-than-tempered decay
occurs, such as for the discrete series of \(SL(2,\RR)\). Our method is
a new purely dynamical one based on the study of cohomological
equations. While cohomological equations have previously been used by
Forni to extract spectral information related to mixing, our use is different: we use the existence of
distributional solutions directly to obtain quantitative estimates for
matrix coefficients. This gives a new dynamical route to decay and also
makes clear the limitation of the classical equation: distributional
solutions live in global negative Sobolev spaces, which forces full
Sobolev norms on observables. This is precisely what motivates the
passage to fractional equations, whose \(L^2\)-solutions yield the
partial-Sobolev estimates needed later. The method also leads to
quantitative upper bounds for matrix coefficients in more general
semisimple settings, avoiding the traditional heavy use of \(K\)-type
expansions and spherical-function analysis. This is one of the main
innovations of the paper.

Second, we then use these geodesic estimates to derive polynomial decay along the horocycle flow. Third, we prove a Tauberian-type implication from horocycle decay to fractional solvability. This yields a cohomology-free range for the fractional equation, with constants uniform over a spectral gap. The Tauberian step is carried out directly in the time domain, via an explicit Riesz-kernel identity and a small/large-time decomposition, without passing through an absolutely continuous spectral density. This Tauberian step is the core analytic mechanism of the paper.

\textbf{\emph{Partial Sobolev estimates}} A second major difficulty is to obtain partial Sobolev estimates for the fractional solutions. In the nilmanifold setting this follows from direct computation once the representation-theoretic model is in hand. In the semisimple setting, the key observation is that there exists a subgroup \(H\) of \(SL(2,k)\) such that each smooth vector \(\omega\) can be decomposed as $\omega=\omega_1+\omega_2$, where the two components behave differently with respect to the fractional equation: the \(L^2\)-norm of the solution with coboundary term \(\omega_1\) is controlled directly by \(\|\omega\|_{L^2}\), while the Sobolev norms of \(\omega_2\) are controlled by the \(H\)-partial Sobolev norm of \(\omega\). Consequently, the \(L^2\)-norm of the solution corresponding to \(\omega_2\) is bounded by a partial Sobolev norm of \(\omega\). This partial-norm mechanism is the second key input.

Finally, if \(S\) is a subgroup commuting with the given \(SL(2,k)\) subgroup, then the solution obtained along that \(SL(2,k)\) direction remains smooth along \(S\), since the Laplacian \(\Delta_S\) commutes with \( |U|^r \). This allows the argument to be iterated along all relevant rank-one \(SL(2,k)\) subgroups, ultimately yielding the solvability of the Type~$II$ multiple fractional equations.

\textbf{\emph{From order-$2$ decay to higher-order mixing}. }Once the fractional equations have been solved with partial Sobolev estimates, order-$2$ decay follows by conjugating the fractional operators under the dynamics: stable directions produce exponential gain in forward time, and unstable directions produce the analogous gain in backward time. Higher-order mixing is then obtained by a two-block decomposition of multiple correlations and an induction on the number of factors. This yields quantitative mixing of all orders. In particular, for triple correlations one obtains the maximal-gap phenomenon, whereas for higher orders no analogous uniform maximal-gap bound can hold in general.

\section{Notations and preparative steps}\label{sec:33}

Throughout the rest of the paper, we use the notation introduced in this section. We use $\epsilon>0$ to denote a sufficiently small constant.    We write $C$ for a constant depending only on the manifolds $\mathcal{X}$, the group action $\alpha$ (its value may vary from one occurrence to the next). When a constant depends on additional parameters, we indicate this with subscripts; for example, $C_{x,y,z}$ denotes a constant that may depend on $x,y,z$ (in addition to $\mathcal{X}$ and $\alpha$).

\subsection{Basic notations} Let $G$ be a connected semisimple Lie group of non-compact type with finite center. Let $A$ be a maximal split torus of $G$ and $A^+$ the closed positive Weyl chamber of $A$ such
that the Cartan decomposition $G=KA^+K$ holds. For each $c\in A$, there exists $w$ in the Weyl group $W$ such that $c^+:=w^{-1}cw\in A^+$. We note that $c^+$ is the unique element of the Weyl orbit of $c$
lying in $A^+$.
\begin{cect6}\label{for:120} In the semisimple case, we use $(\pi,\mathcal H)$ to denote a unitary representation of $G$ with \emph{strong spectral gap}, i.e., the restriction of $\pi$ to each simple factor of $G$ is isolated
from the trivial representation with respect to the Fell topology. We also assume that $\pi$ has no non-trivial $G$-invariant vectors.  We will also consider a special type of unitary representation arising from measure-preserving actions.  Given a measure-preserving action of $G$ on a probability space $(\mathcal{X},\,\varrho)$, where $\mathcal X$ is a $C^\infty$ manifold endowed with a $G$-invariant Borel probability measure $\varrho$. Let $(\pi, L^2(\mathcal{X}))$ be the induced representation of $G$ on $L^2(\mathcal{X})$. The subspace
\[
\mathcal H \;=\; L_0^2(\mathcal X,\varrho)
:= \Bigl\{f\in L^2(\mathcal X,\varrho): \int_{\mathcal X} f\,d\varrho = 0\Bigr\}
\]
is $G$-invariant, so $(\pi,\mathcal H)$ is the restriction of the above unitary
representation to zero-mean functions.
 We assume in this case that $(\pi,\mathcal H)$ has a strong spectral gap as above.

In what follows (within the semisimple setting), unless stated otherwise, $\mathcal H$ denotes an abstract Hilbert space. When we specialize to $\mathcal H=L^2_0(\mathcal X,\varrho)$, we will say so explicitly.

\end{cect6}

\begin{cect6}\label{for:123} We use $a$ to denote an element of $A$.  We have a decomposition
\[
    \text{Lie}(G)\;=\; W_{-,a}\oplus W_{0,a}\oplus W_{+,a},
\]
into the stable, neutral, and unstable subspaces of $\text{Ad}_a$. Concretely, if $\lambda$ ranges over the eigenvalues of $da$, then
$W_{-,a}$ (resp.\ $W_{0,a}$, $W_{+,a}$) is the sum of generalized eigenspaces with $|\lambda|<1$ (resp.\ $|\lambda|=1$, $|\lambda|>1$).

Let $H_{-,a}$, $H_{+,a}$, $H_{0,a}$ be the connected subgroups of $G$ corresponding to the Lie subalgebras $W_{-,a}$, $W_{+,a}$ and $W_{0,a}$, respectively. Similarly, let $H_{-0,a}$ and $H_{+0,a}$ be the connected subgroups corresponding to $W_{-,a} \oplus W_{0,a}$ and $W_{+,a} \oplus W_{0,a}$.

\end{cect6}

\begin{cect6} For any compactly supported function  $f$ on $\mathcal{X}$ and any subgroup $H$ of $G$,
we write $\|f\|_{H,C^s}$ for the $C^s$-norm of $f$ taken using only
derivatives along vector fields from $\Lie(H)$. Write $s=m+\theta$ with $m\in\NN$ and $0\le\theta<1$. Fix a basis
$(Y_1,\dots,Y_d)$ of $\Lie(H)$ and a compact set $\mathcal C\subset\Lie(H)$
whose linear span equals $\Lie(H)$. Define
\[
\|f\|_{H,C^{s}}
:=\|f\|_{H,C^{m}}
\;+\;
\sum_{|\alpha|=m}\;
\sup_{Y\in\mathcal C}\;
\sup_{0<|t|\le 1}\;
\frac{\big\|\pi(\exp(tY))(Y^\alpha f) - Y^\alpha f\big\|_{C^0}}{|t|^{\theta}},
\]
where $Y^\alpha:=Y_{\alpha_1}\cdots Y_{\alpha_m}$ for a multi-index $\alpha = (\alpha_1, \dots, \alpha_m)$. Denote by $C_c^{s, H}(\mathcal{X})$ the space of compactly supported functions on $\mathcal{X}$ for which $\|f\|_{H,C^{s}} < \infty$.

 It is clear that if $f\in C_c^{s,H}(\mathcal{X})$, then $f\in W^{s,H}(L^2(\mathcal{X}))$ (with respect to $\pi$, the precise definition is given in Section \ref{sec:9}) we have
\begin{align}\label{for:189}
 \norm{\xi}_{H,s}\leq \norm{\xi}_{H,C^s}.
\end{align}
For any $v\in \text{Lie}(G)$ we have
\begin{align}\label{for:125}
v\pi(a)
=\|\text{Ad}_{a^{-1}}v\|\,\pi(a)\tilde v,
\qquad
\tilde v:=\frac{\text{Ad}_{a^{-1}}v}{\|\text{Ad}_{a^{-1}}v\|}.
\end{align}
This factorization allows us to extract the norm factors $\|\text{Ad}_{a^{-1}}v\|$ from the directional derivative.
\end{cect6}

\subsection{Abelian subgroups and semisimple parts}\label{for:173}
Let $\mathcal{Z}\subset G$ be a closed abelian subgroup. For any $z\in \mathcal{Z}$ we have a corresponding Jordan-Chevalley normal form decomposition into $3$ commuting elements $z=s_zk_zn_z$ where $s_z$ is semisimple, $k_z$ is compact and $n_z$ is unipotent. Set $p(z)=s_z$ and $w_z=k_zn_z$. Then:
 \begin{enumerate}
   \item There exists an $\RR$-split Cartan subgroup $A$ of $G$ such that $p(z)\in A$ for any $z\in \mathcal{Z}$, and $p: \mathcal{Z}\to A$ is a group homomorphism (see Proposition 5.1 of \cite{W6} or Proposition 4.2 of \cite{vw}).

   \item\label{for:121} For any $n\in\ZZ$, $\norm{\text{Ad}_{w_z^n}}\leq C_z(|n|+1)^{\dim G}$ (see Appendix $K$ of \cite{W6}).
 \end{enumerate}
There exist $z_1,\cdots z_{m}\in \mathcal{Z}$, with $m=\dim \mathcal{Z}$ and a compact subset $\mathcal{C}\subseteq \mathcal{Z}$ such that:
\begin{enumerate}
  \item $z_1,\cdots z_{m}\in \mathcal{Z}$ generate a subgroup $\mathcal{Z}'$ and
for any $z\in \mathcal{Z}$, there exists $z'\in \mathcal{Z}'$ and $c\in \mathcal{C}$ such that $z=z'c$. Let
\[
c_0:=\; \max_{c\in \mathcal{C}}\max\!\bigl\{\|\Ad_c\|,\|\Ad_{c^{-1}}\|\bigr\}<\infty .
\]
  \item Denote by $S$ the (connected) subgroup of $A$ generated by the one-parameter subgroups $\{(s_{z_i})^{t}\}_{t\in\mathbb{R}}$. Then $p(\mathcal{Z})\subseteq S\subseteq A$.
\end{enumerate}
For $z'=\prod_{i=1}^{m} z_i^{\,n_i}$ with $n_i\in\ZZ$, we have
\begin{align}\label{for:122}
 \norm{\text{Ad}_{w_{z'}}}&=\norm{\text{Ad}_{(w_{z_1})^{n_1}\cdots (w_{z_{\dim \mathcal{Z}}})^{n_{\dim \mathcal{Z}}}}}\leq
 \Pi_{i=1}^{\dim \mathcal{Z}}\norm{\text{Ad}_{(w_{z_i})^{n_i}}}\notag\\
 &\overset{(*)}{\leq} C_{\mathcal{Z}} \Pi_{i=1}^{\dim \mathcal{Z}}(\abs{n_i}+1)^{\dim G}.
\end{align}
Here in $(*)$ we use \eqref{for:121}.

Now we suppose $\mathcal{Z}$ is \emph{essentially semisimple}, i.e., if $z\in \mathcal{Z}$ is non-trivial, then $s_{z}$ is non-trivial. Let $a_i=s_{z_i}\in A$, $1\leq i\leq m$. Then for any $z\in \mathcal{Z}$, there exists a unique vector $t=(t_1,\cdots, t_{m})\in\RR^{m}$, $m=\dim\mathcal{Z}$ such that
\begin{align*}
 s_z=a_1^{t_1}\cdots a_{m}^{t_{m}}
\end{align*}
and the map $\tau: z\to t$ is a group homomorphism as  $p$ is a homomorphism.
For any $z_1,\,z_2\in \mathcal{Z}$, define $d(s_{z_1},s_{z_2})=\norm{\tau(z_1)-\tau(z_2)}$.

 Write $\mathfrak{a}=(a_1,\dots,a_{m})$ and, for $t\in\RR^m$, set $\mathfrak{a}^t := a_1^{t_1}\cdots a_{m}^{t_m}$.
Then we have
\begin{align}\label{for:127}
 \norm{\text{Ad}_{w_z}}&\leq c_0\norm{\text{Ad}_{w_{z'}}}\overset{(*)}{\leq} C_{\mathcal{Z}} (\norm{\tau(z')}+1)^{\dim \mathcal{Z}\dim G}\notag\\
 &\overset{(**)}{\leq} C_{\mathcal{Z},1} (\norm{\tau(z)}+1)^{\dim \mathcal{Z}\dim G}.
\end{align}
Here in $(*)$ we use \eqref{for:122}; in $(**)$ we note that $\tau(z')=\tau(z)-\tau(c)$ with $\tau(c)$ ranging over a compact set.

\subsection{Lie preliminaries}\label{sec:36} Let $k=\RR$ or $\CC$.  Denote by $\Phi$ the set of non-multipliable roots of $A$ and by $\Phi^+$
the set of positive roots in $\Phi$.
The group $G$ contains a connected
subgroup $G_0$ such that each simple factor of $G_0$ is split over $k$ and we have $A\subset G_0$ and $\Phi=\Phi(A,  G_0)$. We fix a
left-invariant Riemannian metric $d$ on $G$ which is bi-invariant under $K$.

\subsubsection{Strongly orthogonal system of $\Phi$}\label{for:220}  A subset $\mathcal{S}$ of $\Phi^+$ is called a strongly orthogonal system
of $\Phi$ if any two distinct elements $\theta$ and $\beta$ of $\mathcal{S}$ are strongly orthogonal, that is, neither
of $\theta\pm\beta$ belongs to $\Phi$. A strongly orthogonal system $\mathcal{S}$ is called \emph{maximal} if the coefficient of each simple
root in the formal sum $\sum_{\theta\in \mathcal{S}}\theta$ is not less than the one in
$\sum_{\theta\in \mathcal{O}}\theta$ for any strongly
orthogonal system $\mathcal{O}$ of $\Phi$.

\subsubsection{Basic subgroups}\label{for:221}   For each $\theta\in \Phi$ let $u_\theta$ denote the corresponding one-dimensional  root subgroup inside $G_0$ (defined over $\theta(k)$, where $\theta(k)=\RR$ or $\CC$).  There is an isomorphism $\Psi_\theta:\,\mathfrak{sl}(2,\theta(k))\to \text{Lie}(G_0)$ defined over $\theta(k)$ such that
\[
\Psi_\theta\!\begin{pmatrix} 0 & x \\ 0 & 0 \end{pmatrix} \in\Lie(u_\theta),\qquad
\Psi_\theta\!\begin{pmatrix} 0 & 0 \\ y & 0 \end{pmatrix} \in\Lie(u_{-\theta}),\qquad
\Psi_\theta\!\begin{pmatrix} t & 0 \\ 0 & -t \end{pmatrix} \in \text{Lie}(A)
\]
for all $x, y \in \theta(k)$ and $t \in \RR$. We can choose these $\Psi_\theta$ to be \emph{compatible} with $\Phi^+$, i.e., if $\theta\in \Phi^+$ and $a\in A^+$ ,then $\Lie(u_\theta)$ lies in the unstable subspace of $\text{Ad}_a$ and $\Lie(u_{-\theta})$ lies in the stable subspace of $\text{Ad}_a$.

Let $\mathcal{V}_\theta$ be the subgroup with Lie algebra $\{\Psi_\theta\!\begin{pmatrix} z & x \\ 0 & -z \end{pmatrix}:x,z\in \theta(k)\}$ and let $H_\theta$ be the subgroup with Lie algebra $\Psi_\theta(\mathfrak{sl}(2,\theta(k)))$. For a strongly orthogonal system $\mathcal{S}$, set
\begin{align*}
 S(\mathcal{S})=\Pi_{\theta\in \mathcal{S}}\mathcal{V}_\theta,\quad u(\mathcal{S})=\Pi_{\theta\in \mathcal{S}}u_{\theta},\quad H(\mathcal{S})= \Pi_{\theta\in \mathcal{S}}H_{\theta}.
\end{align*}
\emph{Note}. For any $a\in A^+$,  $\Lie(u_\theta)$ lies in the unstable subspace of $\text{Ad}_a$, while $\Lie(u_{-\theta})$ and $\Psi_\theta \begin{pmatrix} z & 0 \\ 0 & -z \end{pmatrix}$, $z\in k$, lie in the weak stable subspace. Hence  cosets of $u(\mathcal{S})$  lie in the unstable foliations of the left translations by $a$ and $S(\mathcal{S})$ lies in weak stable foliations, i.e., $u(\mathcal{S})\subseteq H_{+,a}$ and $S(\mathcal{S})\subseteq H_{-0,a}$ (see \eqref{for:123} of Section \ref{sec:33}).

\subsubsection{Basic numbers}\label{for:222}    Let $a_{\theta,s}=\exp\big(\Psi_\theta\!\begin{pmatrix} s & 0 \\ 0 & -s \end{pmatrix}\big)$ and  we say that $r>0$ is a \emph{decay rate} for the restricted representation  $\pi|_{H_\theta}$ if for any $v\in W^{\infty,H_\theta}(\mathcal{H})$ there exists a constant $C_{v}>0$ such that
  \begin{align}
\big|\langle \pi(a_{\theta,s})v,\, v\rangle \big|\leq C_{v}e^{-2|s|r}, \qquad \forall\,s\in\RR.
\end{align}
Since $\pi$ has strong spectral gap (see \eqref{for:120} of Section \ref{sec:33}), for any $\theta\in \Phi$ $\pi|_{H_\theta}$ also has strong spectral gap; and there exists some $r>0$
which is a decay rate for  (see \cite{Kleinbock} and \cite{Spatzier1}). Define
\[
\gamma_\theta:=\sup\bigl\{\,r>0 : r \text{ is a decay rate for }\pi|_{H_\theta}\,\bigr\}.
\]
We call the number $\gamma_\theta$ the \emph{strong spectral gap} of the restriction $\pi|_{H_\theta}$. Set
\begin{align*}
 \zeta_{\theta,\epsilon}=1+\epsilon \quad\text{if } \theta(k) = \CC \quad\text{ and }\quad\zeta_{\theta,\epsilon}=\gamma_\theta+2+\epsilon \quad\text{if } \theta(k) = \RR.
\end{align*}
For a strongly orthogonal system $\mathcal{S}$, set
\begin{align*}
  \eta_\epsilon(\mathcal{S}, a)=\Pi_{\theta\in \mathcal{S}}\theta(a^+)^{-(\gamma_\theta-\epsilon)},\quad \zeta_{\epsilon}(\mathcal{S})=\sum_{\theta\in \mathcal{S}}\zeta_{\theta,\epsilon},
  \quad p_{\epsilon}(\mathcal{S})=\sum_{\theta\in \mathcal{S}}(\gamma_{\theta}-\epsilon)
 \end{align*}
 (recall that $a^+$ is defined in \eqref{for:120} of Section \ref{sec:33}).
\begin{remark}\label{re:12}
 If $\theta(k)=\CC$, then $0<\gamma_\theta\leq1$ \cite{Wallach}. Suppose $\theta(k)=\RR$. If $\pi|_{H_\theta}$ has no discrete series summand, $0<\gamma_\theta\leq\frac{1}{2}$ \cite{tan}. If $\pi|_{H_\theta}$ consists only of discrete series representations, $\gamma_\theta\in \{\tfrac12,1,\tfrac32,\dots\}$ \cite{Wallach}.
 In all cases, by definition of \(\zeta_{\theta,\epsilon}\) we have $\zeta_{\theta,\epsilon}\geq \gamma_\theta $ for all \(\theta\). Then we have
\begin{align}\label{for:4}
\zeta_{\epsilon}(\mathcal S)\ge p_{\epsilon}(\mathcal S).
\end{align}

\end{remark}

\section{Preliminaries on unitary representation theory}\label{sec:9}
\subsection{Smooth vectors, Sobolev spaces, distributions}
Let $\rho$ be a unitary representation of a Lie group $S$ with Lie algebra $\mathfrak{S}$ on a Hilbert space $H_\rho$. We call a vector $\xi\in H_\rho$ $C^\infty$ for $S$ if the map
\begin{align*}
 s\to \rho(s)\xi
\end{align*}
is a $C^\infty$ function from $S$ to $H_\rho$. We use $W^\infty(H_\rho)$ to denote the set of $C^\infty$ vectors. The derived representation $\rho_*$ of $\rho$  is the Lie algebra representation of $\mathfrak{S}$ on $H_\rho$ defined as follows. For every $X \in \mathfrak{S}$,

\begin{equation}
\rho_*(X) := \operatorname{strong{-} }\lim_{t \to 0} \frac{\rho(\exp tX) - I}{t}.
\end{equation}
For any $v\in H_\rho$, we also write $Xv:=\rho_*(X)v$ for simplicity.

It can be shown that the derived representation $\rho_*$ of the Lie algebra $\mathfrak{S}$ on the Hilbert space $H_\rho$ is essentially skew-adjoint in the following sense: For every element $X \in \mathfrak{S}$, the operator $\rho_*(X)$ is essentially skew-adjoint. Moreover, these operators share a common invariant core, which is the subspace $W^\infty(H_\rho) \subset H_\rho$ of $C^\infty$ vectors.

The subspace $W^\infty(H_\rho)\subset H_\rho$ is endowed with the Fr\'{e}chet $C^\infty$ topology, that is, the topology defined by the family of seminorms
\[
\left\{\, \| \cdot \|_{v_1,v_2,\dots,v_m} \;\middle|\; m \in \mathbb{N} \text{ and } v_1, \ldots, v_m \in \mathfrak{S} \,\right\}
\]
defined as follows:
\[
\| \xi\|_{v_1,v_2,\dots,v_m} := \left\|\, \rho_*(v_1)\rho_*(v_2) \cdots \rho_*(v_m)\xi \,\right\|, \quad \text{for } \xi \in W^\infty(H_\rho) .
\]
For simplicity, we often write
\begin{align*}
 v_1v_2\cdots v_m\xi:=\rho_*(v_1)\rho_*(v_2) \cdots \rho_*(v_m)\xi, \quad \text{for } \xi \in W^\infty(H_\rho).
\end{align*}
The space of \emph{Schwartz distributions} for the representation $\rho$ is defined as the dual space of $W^\infty(H_\rho)$ (endowed with the Fr\'{e}chet $C^\infty$ topology). The space of Schwartz distributions for the representation $\rho$ will be denoted $\mathcal{E}'(H_\rho)$.

The representation $\rho_*$ extends in a canonical way to a representation (denoted by the same symbol) of the enveloping algebra $\mathcal{U}(\mathfrak{S})$ of $\mathfrak{S}$ on the Hilbert space $H_\rho$. Let $\Delta_S \in \mathcal{U}(\mathfrak{S})$ be a left-invariant second-order, positive elliptic operator on $S$ fixed once and for all. For example, the operator
\begin{align*}
 \Delta_S = -\left( v_1^2 + \cdots + v_d^2 \right)
\end{align*}
where $\{ v_1, \ldots, v_d \}$ is a basis of $\mathfrak{S}$ as a vector space.

For each \(\alpha>0\), let
\[
W^{\alpha}(H_\rho)
\;=\;
\mathrm{Dom}\bigl((I + \Delta_{S})^{\alpha/2}\bigr)
\;\subset\;
H_{\rho}
\]
be the Sobolev space of order \(\alpha\) with respect to the group \(S\) on the Hilbert space \(H_{\rho}\).  Equip it with the full norm
\[
\|\xi\|_{S,\alpha,H_{\rho}}
\;=\;
\bigl\|(I + \Delta_{S})^{\alpha/2}\,\xi\bigr\|_{H_{\rho}},
\qquad
\xi\in W^{\alpha}(H_\rho).
\]
When \(S\) and \(H_{\rho}\) are clear from context, we may abbreviate
\[
\|\xi\|_{S,\alpha,H_{\rho}}
=\|\xi\|_{S,\alpha}
=\|\xi\|_{\alpha}.
\]
If instead we restrict to a subgroup \(P\subset S\), we write
\[
W^{\alpha,P}(H_{\rho})
=\mathrm{Dom}\bigl((I + \Delta_{P})^{\alpha/2}\bigr),
\]
with norm
\[
\|\xi\|_{P,\alpha,H_{\rho}}
=\bigl\|(I + \Delta_{P})^{\alpha/2}\,\xi\bigr\|_{H_{\rho}},
\]
and similarly drop subscripts when no ambiguity arises.

Let also $W^{-\alpha}(H_\rho) \subset \mathcal{E}'(H_\rho)$ be the dual Hilbert space of $W^\alpha(H_\rho)$. Then $W^\infty(H_\rho)$ is the projective limit of the spaces $W^\alpha(H_\rho)$ (and consequently $\mathcal{E}'(H_\rho)$ is the inductive limit of $W^{-\alpha}(H_\rho)$) as $\alpha \to +\infty$. A distribution $\mathcal{D} \in W^{-\alpha}(H_\rho)$ will be called a distribution of order at most $\alpha \in \mathbb{R}^+$.

We list the well-known elliptic regularity theorem which will be frequently
used in this paper (see \cite[Chapter I, Corollary 6.5 and 6.6]{Robinson}):
\begin{theorem}\label{th:15}
Fix a basis $\{Y_j\}$ for $\mathfrak{S}$ and set $L_{2m}=\sum Y_j^{2m}$, $m\in\NN$. Then for any $\xi\in W^\infty(H_\rho)$, we have
\begin{align*}
    \norm{\xi}_{2m}\leq C_m(\norm{L_{2m}\xi}+\norm{\xi}),\qquad \forall\, m\in\NN
\end{align*}
where $C_m$ is a constant only dependent on $m$ and $\{Y_j\}$.
\end{theorem}
There exists a collection of smoothing operators $\mathfrak{s}_b: H_\rho\to W^\infty(H_\rho)$, $b>0$, such that for any $s, s_1,s_2\geq0$ and any $\xi\in W^s(H_\rho)$ the following holds (see \cite{Hamilton}):
\begin{align}
 \norm{\mathfrak{s}_b\xi}_{s+s_1}&\leq C_{s,s_1}b^{s_1}\norm{\xi}_{s},\quad \text{and}\label{for:197}\\
 \norm{(I-\mathfrak{s}_b)\xi}_{s-s_2}&\leq C_{s,s_2}b^{-s_2}\norm{\xi}_{s},\quad \text{if }s\geq s_2.\label{for:198}
\end{align}

\subsection{Direct decompositions of Sobolev space}\label{sec:20}
For any Lie group $S$ of type $I$ and its unitary representation $(\rho, H_\rho)$, there is a decomposition of $\rho$ into a direct integral
\begin{align*}
  \rho=\int_Z\rho_zd\mu(z)
\end{align*}
of irreducible unitary representations $(\rho_z,\,(H_\rho)_z)$ for some measure space $(Z,\mu)$ (we refer to
\cite[Chapter 2.3]{Zimmer} or \cite{margulis1991discrete} for more detailed account for the direct integral theory). All the operators in the enveloping algebra are decomposable with respect to the direct integral decomposition. Hence there exists for all $s\in\RR$ an induced direct
decomposition of the Sobolev spaces
\begin{align*}
 W^s(H_\rho)=\int_Z W^s((H_\rho)_z)d\mu(z)
\end{align*}
with respect to the measure $d\mu(z)$.

The existence of the direct integral decompositions allows us to reduce our analysis of the
cohomological equation to irreducible unitary representations. This point of view is
essential for our purposes.

\subsection{Functional calculus for self-adjoint operators}\label{sec:19} Let $H$ be a Hilbert space, and let
  $\mathcal{P}$ be an unbounded, self-adjoint linear operator with a dense domain $\text{Dom}(\mathcal{P})\subset H$. Suppose $\mathcal{P}$ is strictly positive, meaning there is $c>0$ such that
\begin{align*}
  \norm{\mathcal{P}\vartheta}\geq c\norm{\vartheta},\qquad \forall\,\vartheta\in \text{Dom}(\mathcal{P}),
\end{align*}
then there is a regular Borel measure $\tau$ on $[c,\infty)$ such that $\mathcal{P}$ is unitarily equivalent to multiplication by $\lambda$ on
\begin{align}\label{for:108}
 \int_{[c,\infty)} H_\lambda d\tau(\lambda).
\end{align}
More precisely, we have
a direct integral decomposition:
\begin{align*}
  \vartheta=\int_{[c,\infty)}\vartheta_\lambda d\tau(\lambda),\qquad \forall \,\vartheta\in H.
\end{align*}
Furthermore, for any (measurable) function $g:[c,\infty)\to\CC$, the functional calculus gives
\begin{align*}
  g(\mathcal{P})\vartheta=\int_{[c,\infty)}g(\lambda)\vartheta_\lambda d\tau(\lambda),\qquad \forall \,\vartheta\in \text{Dom}(g(\mathcal{P})),
\end{align*}
where
\begin{align*}
 \text{Dom}(g(\mathcal{P}))=\{\vartheta\in H: \int_{[c,\infty)}|g(\lambda)|^2\norm{\vartheta_\lambda}^2 d\tau(\lambda)<\infty\}.
\end{align*}
A symmetric operator $\mathcal{T}$ is said to be essentially self-adjoint if the closure $\bar{\mathcal{T}}$ of
$\mathcal{T}$ is self-adjoint. An essentially self-adjoint operator is ``almost as good" a self-adjoint operator, since one obtains a self-adjoint operator simply by taking the closure.
If such an operator $\mathcal{T}$ is strictly positive, then its closure $\bar{\mathcal{T}}$ is also strictly positive. Moreover, $\bar{\mathcal{T}}^{\!r}$ is also self-adjoint for any $r\in\RR$.
Therefore, whenever we speak of the spectral decomposition \eqref{for:108} of $\mathcal{T}^{\!r}$ (where $\mathcal{T}$ is an essentially self-adjoint, strictly positive operator $\mathcal{T}$),
we are really referring to the spectral decomposition of
$\bar{\mathcal{T}}^r$.

\subsection{Fractional operators} Suppose $S$ is a Lie group. Let $(\rho,\mathcal{R})$ be a unitary representation of $S$.  we assume there is a closed abelian subgroup $\mathcal{A}\le S$ isomorphic to $\RR^m$ for which $\{\mathfrak u_1,\ldots,\mathfrak u_m\}$ is a basis of $\text{Lie}(\mathcal{A})$.
We have the following results:
\begin{lemma}\label{le:17}  Let $t_1,\cdots, t_m\in\RR^+$. Set
\begin{align*}
 \mathcal{T}=|\mathfrak{u}_{1}|^{t_1}|\mathfrak{u}_{2}|^{t_2}\cdots |\mathfrak{u}_{m}|^{t_m}.
\end{align*}
Then:
\begin{enumerate}
  \item\label{for:163} For any $\xi,\,\eta\in W^{\infty,\,\mathcal{A}}(\mathcal{R})$ (recall Section \ref{sec:9}), we have $\langle \mathcal{T}\xi, \eta \rangle=\langle \xi, \mathcal{T}\eta\rangle$.

 \item\label{for:164} for any $t_1,\cdots, t_m\in\NN\cup \{0\}$, we have $\big\|\mathcal{T}\xi\big\|=\big\|\mathfrak{u}_{1}^{t_1}\mathfrak{u}_{2}^{t_2}\cdots \mathfrak{u}_{m}^{t_m}\xi\big\|$.

  \item\label{for:169} For any $h\in S$, we have
  \begin{align*}
  &\rho(h^{-1})\mathcal{T}\rho(h)=a_1^{t_1}\cdots a_m^{t_m}|\tilde{\mathfrak{u}}_{1}|^{t_1}|\tilde{\mathfrak{u}}_{2}|^{t_2}\cdots |\tilde{\mathfrak{u}}_{m}|^{t_m},
\end{align*}
  where
  \begin{align*}
 \tilde{\mathfrak{u}}_{i}&=\frac{\text{Ad}_{h^{-1}}\mathfrak{u}_{i}}{\|\text{Ad}_{h^{-1}}\mathfrak{u}_{i}\|},\quad\text{and} \quad
 a_i=\|\text{Ad}_{h^{-1}}\mathfrak{u}_{i}\|,\quad 1\leq i\leq m.
 \end{align*}

\end{enumerate}

\end{lemma}
\begin{proof} \eqref{for:163} and \eqref{for:164} follow directly from the definition. \eqref{for:169}:  For any $\xi\in \mathcal{R}$ and any $v\in \text{Lie}(S)$ we have
\begin{align*}
\rho(h^{-1})(iv)\rho(h)=c(i\tilde{v}),\qquad\text{where }c=\|\text{Ad}_{h^{-1}}v\|,\quad \tilde v=\frac{\Ad_{h^{-1}}v}{\|\Ad_{h^{-1}}v\|}.
\end{align*}
We note that both $iv$ and $i\tilde{v}$ are self-adjoint. Applying functional calculus to the Borel function $\lambda\mapsto |\lambda|^r$, we obtain
\[
\rho(h^{-1})|v|^r\rho(h)=\rho(h^{-1})|iv|^r\rho(h)=|c(i\tilde{v})|^r=c^r|\tilde{v}|^r.
\]
Applying this to each \(\mathfrak u_i\) gives \eqref{for:169}.

\end{proof}

\section{Representation theory of $SL(2,\RR)$}\label{sec:1} We recall the conclusions in \cite{Forni}. We choose as generators for $\mathfrak{sl}(2,\RR)$ the elements
\begin{align}\label{for:95}
X=\begin{pmatrix}
  1 & 0 \\
  0 & -1
\end{pmatrix},\quad U=\begin{pmatrix}
  0 & 1 \\
  0 & 0
\end{pmatrix},\quad V=\begin{pmatrix}
  0 & 0 \\
  1 & 0
\end{pmatrix},\quad \Theta=\begin{pmatrix}
  0 & 1 \\
  -1 & 0
\end{pmatrix}.
  \end{align}
The \emph{Casimir} operator is then given by
\begin{align*}
\Box:= -X^2-2(UV+VU),
\end{align*}
which generates the center of the universal enveloping algebra of $\mathfrak{sl}(2,\RR)$. The Casimir operator $\Box$
acts as a constant $\mu \in \mathbb{R}^+ \cup \{-n^2 + 2n \mid n \in \mathbb{Z}^+\}$ on the Hilbert space of each irreducible unitary representation, and its value classifies all nontrivial irreducible unitary representations according to three different types. Let $\mathcal{H}_\mu$ be the Hilbert space of an irreducible unitary representation $\pi_\mu$ on which the Casimir operator takes the value $\mu \in \mathbb{R}^+ \cup \{-n^2 + 2n \mid n \in \mathbb{Z}^+\}$. The representation is said to belong to:
\begin{itemize}
  \item the principal series representations if $\mu >1$,
\smallskip
  \item the complementary series representations if $0 < \mu < 1$,
  \smallskip
  \item the discrete series representations if $\mu=-n^2+2n$ for $n\in\NN$ and $\mu\leq0$,
  \smallskip
  \item the mock discrete series (limit of the discrete series representations) if $\mu=1$.
\end{itemize}

\begin{remark}\label{re:4}The above classification is  still valid for irreducible unitary representations of Lie groups whose
Lie algebra is $\mathfrak{sl}(2,\RR)$. All of these are unitarily equivalent to
irreducible representations of $SL(2,\RR)$ itself \cite{tan}.
\end{remark}

Any unitary representation $(\rho,H_\rho)$ of a connected Lie group $P$ with $\text{Lie}(P)=\mathfrak{sl}(2,\RR)$ is decomposed into a direct integral (see \cite{Forni} and \cite{mautner1950unitary})
\begin{align*}
H_\rho=\int_{\oplus}\mathcal{L}_\mu ds(\mu)\quad\text{and}\quad \omega=\int_{\oplus}\omega_\mu ds(\mu) \quad \forall\,\omega\in H_\rho
\end{align*}
with respect to a positive Stieltjes measure $ds(\mu)$ over the spectrum $\sigma(\Box)$. The
Casimir operator acts as the constant $\mu\in \sigma(\Box)$ on every Hilbert space $\mathcal{L}_\mu$. The
representation of $\rho$ on $\mathcal{L}_\mu$  need not to be irreducible. In fact, $\mathcal{L}_\mu$ is in general
the direct sum of an (at most countable) number of unitary representations equal
to the spectral multiplicity of $\mu\in \sigma(\Box)$. We say that \emph{$\rho$ has a spectral gap (of $\mu_0$)} if $\mu_0>0$ and $s((0,\mu_0))=0$.
\begin{remark}\label{re:9} Traditionally, one then defines a parameter $\nu_0 \in [0,1)$ (see \cite{Forni} and \cite{tan}) by
\begin{align}\label{for:50}
\nu_0 :=
\begin{cases}
\sqrt{1 - \mu_0} \qquad & \text{if } 0<\mu_0 < 1, \\
0 \qquad & \text{if } \mu_0 \geq 1.
\end{cases}
\end{align}
The spectral gap is used to determine the edge of the complementary series: it bounds how far the spectrum of $\Box$ can penetrate into the interval $(0,1)$. The discrete series and principal series lie outside a fixed neighborhood of the trivial representation  (in the sense of Fell topology).

In particular, when $0<\mu_0<1$, the parameter $\nu_0$ coincides with the parameter of the ``worst" (i.e.\ slowest-decaying) complementary series allowed by the gap $\mu_0$, and therefore governs the optimal decay rate of matrix coefficients of $\rho$ (see Lemma \ref{for:72}). If $\mu_0\ge 1$ (so that $\nu_0=0$), then there is no complementary-series component in the spectrum: all irreducible constituents are principal or discrete, hence tempered. Among tempered representations, the only ones that are strictly better (in the sense of having square-integrable matrix coefficients) are the discrete series (see \cite{Warner}). In other words, the stronger decay coming from discrete-series components is not reflected in $\nu_0$ defined in the traditional sense.

For later use, in the purely discrete case (including the mock discrete series) we extend the notation as follows: if the spectrum of $\Box$ is contained in $\{-n^2+2n : n\in\ZZ^+\}$ and its maximal eigenvalue is $-n^2+2n$, we set
\begin{align}\label{for:213}
\mu_0:=n^2-2n \quad\text{and}\quad \nu_0:= -\sqrt{\mu_0+1}= -n+1
\end{align}
This convention encodes the additional exponential decay coming from discrete-series components, while the definition \eqref{for:50} continues to describe the contribution of the complementary series.

\end{remark}

\subsection{A Fourier model for $SL(2, \RR)$}\label{sec:3}

  We use the representation parameter $\varpi := \sqrt{1 - \mu}$ for convenience, and we denote the real part of $\varpi$ by $\Re\varpi$. In the Fourier model $H_\mu$ for all cases (i.e., principal, complementary, discrete, and mock discrete series),  Lemma~3.1 and Lemma~3.16 of \cite{flaminio2016effective} give a
constant $C_{\re\varpi} > 0$ such that for any $f\in H_\mu$
\begin{equation*}
\Vert f \Vert_{H_\mu}^2 = C_{\re\varpi} \int_{\R} |f(\lambda)|^2 \vert \lambda\vert^{-\Re\varpi} d\lambda\,.
\end{equation*}
For the discrete series or mock discrete series Cauchy's theorem implies that $f$ is supported on $\R^+$ when $f \in W^\infty(H_\mu)$ (see Lemma~3.15 of \cite{flaminio2016effective}).

By direct computation from the vector-field formulas in this model, the derived \(\mathfrak{sl}_2\)-action is
\begin{align}\label{for:47}
U = -\textrm{i} \lambda, \quad X = (\varpi -1) - 2\lambda \frac{d}{d\lambda}, \quad V = \textrm{i}\left((\varpi-1) \frac{d}{d\lambda} - \lambda \frac{d^2}{d\lambda^2}\right).
\end{align}

\subsection{Study of cohomological equation}\label{sec:18}  For the classical horocycle flow defined by the $\mathfrak{sl}(2,\RR)$-matrix $U=\begin{pmatrix}
  0 & 1 \\
  0 & 0
\end{pmatrix}$, Flaminio and Forni made a detailed study in \cite{Forni}.

\begin{theorem}\label{th:5} (Theorems 1.3 and 4.1 of \cite{Forni}) Let $(\rho,\mathcal{H})$ be a unitary representation of $SL(2,\RR)$ without $SL(2,\RR)$-fixed vectors. If $\rho$ has a spectral gap $\mu_0$ and $s > \frac{1 + \nu_0}{2}$, where $\nu_0$ is defined as in \eqref{for:50}, then there exists a constant $C_{\nu_0,s,t}>0$ such that for any $\xi\in W^s(\mathcal{H})$,
\begin{enumerate}
  \item\label{for:68} if $t<-\frac{1 + \nu_0}{2}$, or

  \smallskip
  \item\label{for:217} if \( t < s - 1 \) and \( \mathcal{D}(\xi) = 0 \) for any $U$-invariant distribution $\mathcal{D}$ of Sobolev order $s$
\end{enumerate}
then the equation
\begin{align*}
  U\omega=\xi
\end{align*}
has a solution $\omega\in W^t(\mathcal{H})$ which satisfies the Sobolev estimate
\[
\| \omega \|_t \leq C_{\nu_0,s,t} \| \xi \|_s.
\]
Moreover, if the equation $U\omega=\xi$
has a solution $\omega\in \mathcal{H}$, then $\omega\in W^t(\mathcal{H})$ for any \( t < s - 1 \) with estimates
\[
\| \omega \|_t \leq C_{\nu_0,s,t} \| \xi \|_s.
\]
\end{theorem}
Flaminio and Forni classified the $U$-invariant distributions in each irreducible unitary representation of $SL(2,\RR)$.
\begin{theorem}\label{th:18} (Theorem 1.1 of \cite{Forni}) We use $D_\mu$ to denote the space of $U$-invariant distributions for $\pi_\mu$. Then
\begin{itemize}
  \item If $\mu \geq 1$, then $D_\mu\subseteq W^{-s}(\mathcal{H}_\mu)$ if $s>\frac{1}{2}$.
\smallskip
  \item If $0 < \mu < 1$, then $D_\mu\subseteq W^{-s}(\mathcal{H}_\mu)$ if $s>\frac{1\pm\sqrt{1 - \mu}}{2}$.
  \smallskip
  \item If $\mu=-n^2+2n$ for $n\in\NN$, then $D_\mu\subseteq W^{-s}(\mathcal{H}_\mu)$ if $s>\frac{n}{2}$.

\end{itemize}
\end{theorem}

\section{Multiple fractional equations in the semisimple setting}
Suppose $S$ is a Lie group with Lie algebra $\Lie(S)
  \;=\mathfrak{sl}(2,k_1)\oplus\cdots\oplus \mathfrak{sl}(2,k_n)$, where each $k_i\in\{\RR,\CC\}$. Let $S_i$ denote the subgroup of $S$ whose Lie
algebra is the $i$-th copy of $\mathfrak{sl}(2,k_i)$. Let $X_i$, $U_i$ and $V_i$
(see \eqref{for:95} of Section~\ref{sec:1}) denote the corresponding elements
in the $i$-th copy of $\mathfrak{sl}(2,k_i)$. For $1\le i\le n$, let $\Delta_i$ be the Laplacian of $S_i$ and set
\[
  \Sigma_i
  \;=\;
  \begin{cases}
    I - X_i^2 - V_i^2, &
      \text{if } k_i = \RR,\\[4pt]
    I - X_i^2 - (\mathrm{i}X_i)^2 - V_i^2 - (\mathrm{i}V_i)^2, &
      \text{if } k_i = \CC,
  \end{cases}\quad \text{and}\quad \mathcal{D}_i=I-\Delta_i.
\]

\subsection{Notations}\label{sec:43} We list the following notation that will be used throughout this section:
\begin{enumerate}
\item $\epsilon>0$ denotes a sufficiently small constant: see Section \ref{sec:33}.
  \item\label{for:80} $\pi_\mu,\mathcal{H}_\mu$, a spectral gap $\mu_0$,  and parameter $\nu_0$: see Section \ref{sec:1}.

\item Fourier model $H_\mu$ for $SL(2,\RR)$: see Section \ref{sec:3}.

  \item $P$:  a Lie group with Lie algebra $\mathfrak{sl}(2,k)$, $k=\RR$ or $\CC$.
   \item $(\rho,H_\rho)$: a unitary representation of $P$ with \textbf{strong} spectral gap $\gamma$ (see Section \ref{for:222}).

   \item $U\in\text{Lie}(P)$:  see \eqref{for:95} of Section \ref{sec:1}.
\end{enumerate}
\begin{remark}
We point out the difference between a \emph{spectral gap} and a \emph{\textbf{strong} spectral gap}
for a unitary representation of $P$ when $\Lie(P)=\mathfrak{sl}(2,\RR)$. Both
notions describe which irreducible representations can occur in the
decomposition: the former is formulated in terms of eigenvalues of the
Casimir operator, while the latter is formulated in terms of decay of matrix
coefficients.
\end{remark}

\subsection{Main results: Type $II$ (product) equation}

\begin{theorem}\label{th:6}Let $(\beta,\mathcal{L})$ be a unitary representation of $S$. Suppose for each $1\leq i\leq n$, $\beta|_{S_i}$ has \textbf{strong} spectral gap $\gamma_i$. Let
\begin{align*}
 \zeta_i>1 \quad\text{if } k_i = \CC \quad\text{ and }\quad\zeta_i>\gamma_i+2 \quad\text{if } k_i = \RR.
\end{align*}
For each $\lambda\in\{1,\dots,2^n\}$ and each $1\le j\le n$ we set
\[
  \Lambda_{j,\lambda}=
  \begin{cases}
    U_j, & \text{if } k_j=\RR,\\[4pt]
    U_j \text{ or } \mathrm{i}U_j, & \text{if } k_j=\CC,
  \end{cases}
\]
so that, as $\lambda$ ranges from $1$ to $2^n$, the $n$-tuples
$(\Lambda_{1,\lambda},\dots,\Lambda_{n,\lambda})$ exhaust all possible choices
of $U_j$ or $\mathrm{i}U_j$ in the complex factors.

Then for any $0\leq r_i<\gamma_i$, $1\leq i\leq n$, we have:
\begin{enumerate}
  \item\label{for:82} For any $\xi\in \mathcal{L}$ with $\Sigma_n^{\frac{\zeta_n}{2}}\cdots\Sigma_1^{\frac{\zeta_1}{2}}\xi\in\mathcal{L}$, there is a decomposition $\xi=\sum_{\lambda=1}^{2^n}\xi_\lambda$ with $\xi_\lambda\in \mathcal{L}$, such that each of the equations

\begin{align*}
|\Lambda_{1,\lambda}|^{r_1}|\Lambda_{2,\lambda}|^{r_2}\cdots |\Lambda_{n,\lambda}|^{r_n}\omega_\lambda=\xi_\lambda, \quad 1\leq \lambda\leq 2^n
\end{align*}
 has a solution $\omega_\lambda\in \mathcal{L}$ with the estimate
\begin{align*}
\norm{\omega_\lambda} \leq C_{\mathfrak{r},\mathfrak{p}} \big\|\Sigma_n^{\frac{\zeta_n}{2}}\cdots\Sigma_1^{\frac{\zeta_1}{2}}\xi\big\|,\qquad 1\leq\lambda\leq 2^n,
\end{align*}
where $\mathfrak{r}=(r_1,\cdots, r_n)$ and $\mathfrak{p}=(\gamma_1,\cdots, \gamma_n)$.

\smallskip
\item\label{for:277} For any $\xi\in \mathcal{L}$ with $\mathcal{D}_n^{\frac{\zeta_n}{2}}\cdots\mathcal{D}_1^{\frac{\zeta_1}{2}}\xi\in\mathcal{L}$, and any $\lambda\in\{1,\dots,2^n\}$, the equation
\begin{align*}
|\Lambda_{1,\lambda}|^{r_1}|\Lambda_{2,\lambda}|^{r_2}\cdots |\Lambda_{n,\lambda}|^{r_n}\omega_\lambda=\xi
\end{align*}
 has a solution $\omega_\lambda\in \mathcal{L}$ with the estimate
\begin{align*}
\norm{\omega_\lambda} \leq C_{\mathfrak{r},\mathfrak{p}} \big\|\mathcal{D}_n^{\frac{\zeta_n}{2}}\cdots\mathcal{D}_1^{\frac{\zeta_1}{2}}\xi\big\|.
\end{align*}

  \item\label{for:17}  If there exists an index $i$ with $r_i>\gamma_i$, then for any $m\in\NN$, there is $\xi\in W^m(\mathcal{L})$ such that for any $\lambda\in\{1,\dots,2^n\}$, the equation
\begin{align*}
|\Lambda_{1,\lambda}|^{r_1}|\Lambda_{2,\lambda}|^{r_2}\cdots |\Lambda_{n,\lambda}|^{r_n}\omega_\lambda=\xi
\end{align*}
has no solution $\omega_\lambda\in\mathcal{L}$.
\end{enumerate}

\end{theorem}

\begin{remark}
Theorem~\ref{th:6} makes precise the relation between decay of matrix
coefficients and solvability of fractional cohomological equations along the
nilpotent directions $U_i$ (and $\mathrm{i}U_i$ in the complex case). The
strong spectral gap determines exactly the admissible range of exponents
$r_i$.

The theorem has two forms.
In \eqref{for:82} we obtain solvability using only \emph{partial} Sobolev norms in each factor $S_i$
(of the form $\|\Sigma_n^{\zeta_n/2}\cdots\Sigma_1^{\zeta_1/2}\xi\|$), at the
cost of having to solve simultaneously a family of $2^n$ equations. In
\eqref{for:277} we use the \emph{full} Sobolev norm in each factor $S_i$ (of the form $\|\mathcal{D}_n^{\zeta_n/2}\cdots\mathcal{D}_1^{\zeta_1/2}\xi\|$), and
obtain solvability for a single equation. We will see that these partial norms
play a crucial role in obtaining  higher-order mixing estimates.
\end{remark}

\subsection{Proof strategy}
We first establish the rank-one ingredients underlying Theorem~\ref{th:6}.
Fix $1\le i\le n$. The first task is to solve the single-factor fractional
equation in the $i$-th rank-one subgroup $S_i$:
\[
  |U_i|^{r_i}\omega=\xi
  \qquad\text{or}\qquad
  |U_i|^{r_i}\omega_1+|\mathrm{i}U_i|^{r_i}\omega_2=\xi,
\]
according to whether $k_i=\RR$ or $k_i=\CC$.

The proof proceeds in three steps. First, we prove quantitative decay
estimates for matrix coefficients along the geodesic flow (Proposition \ref{po:1} for $SL(2,\CC)$ and Proposition \ref{cor:4} for $SL(2,\RR)$), and then transfer
these to polynomial decay along the horocycle flow (Corollary~\ref{cor:5}). Second, we show that this
horocycle decay implies solvability of the fractional equation up to the
spectral threshold determined by the strong spectral gap (see Lemma~\ref{le:23}). Third, we refine the solvability statement to obtain estimates in partial
Sobolev norms by introducing suitable spectral cutoffs (see Lemma \ref{le:16}); this leads to
Proposition~\ref{po:4}.

Once these rank-one inputs are available, the full theorem is obtained by an
induction over the commuting factors $S_1,\dots,S_n$. The commutativity of the
factors allows one to combine the single-factor fractional solutions without
loss of control, yielding the product-type solvability statement in Theorem
\ref{th:6}.
\subsection{Rank-one decay estimates} We begin with quantitative decay estimates for matrix coefficients in a
single rank-one factor. These estimates are the starting point for the
fractional solvability argument. The real (Section \ref{sec:7}) and complex cases (Section \ref{sec:5}) are treated
separately, but their role in the proof is the same: they provide the decay
input that later yields the admissible fractional range.
\subsubsection{The complex case} \label{sec:5}

\begin{proposition}\label{po:1} Let $P$ be a Lie group with Lie algebra $\mathfrak{sl}(2,\CC)$ and $(\rho,H_\rho)$ be a unitary representation of $P$ with \textbf{strong }spectral gap $\gamma$ (see Section \ref{for:222}).  Assume:
\begin{align*}
 p=\zeta>1.
\end{align*}
 Then for any $\xi\in W^{p}(H_\rho)$, $\psi\in W^{\zeta}(H_\rho)$ and
  any $a_s = \begin{pmatrix} e^{s} & 0 \\[5pt]
0 & e^{-s} \end{pmatrix}$, $s\in\RR$ we have
  \begin{align*}
\big|\langle \rho(a_s)\psi,\, \xi\rangle \big|\leq C_{\gamma,\epsilon}e^{-2|s|(\gamma-\epsilon)} \| \xi \|_{p} \big\|\psi\big\|_{\zeta}.
\end{align*}

\end{proposition}

\begin{proof} Since $SL(2,\mathbb C)$ is simply connected,
any irreducible representation of such a group $P$ is equivalent to an
irreducible representation of $SL(2,\mathbb C)$. Hence, it is harmless to assume \(P = SL(2,\mathbb C)\) and $\rho$ is an irreducible representations of $P$.

Let \(K = SU(2)\).  The irreducible \(K\)-types $\sigma_\ell$, where \(\ell \in \{0,\tfrac12,1,\tfrac32,\dots\}\) indexes the irreducible
representation of \(SU(2)\) of dimension
\[
  d_\ell = \dim \mathcal \sigma_\ell = 2\ell + 1.
\]
Let \(\Omega\) be the (positive) Casimir operator of \(K\), so that
\[
  \Omega|_{\mathcal H_\ell} = \lambda_\ell \,\mathrm{Id}, \qquad
  \lambda_\ell=\ell(\ell+1).
\]
Let $H_\rho = \widehat{\bigoplus}_{\ell} \mathcal H_\ell$
where $H_\ell$ is $K$-invariant and the action of $K$ on $H_\ell$ is equivalent to $n_\ell\sigma_\ell$ where $n_\ell\in\{0,1\}$ is the multiplicity of $\sigma_\ell\in H_\rho$ (see \cite{Knapp}).  For a \(K\)-finite vector \(v \in \mathcal H\), write $v = \sum_{\ell} v_{\ell}$. Then for any \(m > 0\), we have
\begin{align*}
\|(1+\Omega)^m v\|^2
  = \sum_{\ell} (1+\lambda_\ell)^{2m} \,\|v_\ell\|^2.
\end{align*}
Then, for any \(m > \tfrac12\), by Cauchy-Schwarz inequality we have
\begin{align}\label{for:1}
\sum_{\ell} d_\ell^{\frac{1}{2}} \,\|v_\ell\|
 &= \sum_{\ell}
    \bigl(d_\ell^{\frac{1}{2}} (1+\lambda_\ell)^{-m}\bigr)
    \bigl((1+\lambda_\ell)^m \|v_\ell\|\bigr) \notag\\
 &\le \Big(\sum_{\ell} d_\ell (1+\lambda_\ell)^{-2m}\Big)^{1/2}
      \Big(\sum_{\ell} (1+\lambda_\ell)^{2m} \|v_\ell\|^2\Big)^{1/2}\leq C_m\norm{v}_{2m}
\end{align}
for all \(K\)-finite vectors \(v\).

Suppose $\rho$ has decay rate $r$ (see Section \ref{for:222}). The non-trivial irreducible representations of $SL(2,\mathbb C)$ are principal/complementary series.
The decay exponent $r$ is $1-\epsilon$ for the principal series, while for
the complementary series one has some $r\in(0,1)$ depending on the parameter. Let $a_s = \begin{pmatrix} e^{s} & 0 \\[5pt]
0 & e^{-s} \end{pmatrix}$, $s\in\RR$. Then for any $K$-finite vectors $v$ and $w$ we have
\begin{align}\label{for:270}
 \big|\langle \rho&(a_s)v,\, w\rangle \big|\leq C_r\big(\dim\langle Kv\rangle\dim\langle Kw\rangle\big)^{\frac{1}{2}} e^{-2|s|r}\norm{v}\norm{w}
\end{align}
(see \cite{Cowling1}). Then for any $\psi,\,\xi\in W^t(H_\rho)$, $t>1$ it follows that
\begin{align}\label{for:242}
 \big|\langle \rho(a_s)\psi,\, \xi\rangle \big|&=\Big| \big\langle \rho(a_s)\sum_{\ell} \psi_\ell,\, \sum_{\ell} \xi_\ell \big\rangle \Big|\leq C_r e^{-2|s|r}\sum_{\ell_1,\ell_2}d_{\ell_1}^{\frac{1}{2}}d_{\ell_2}^{\frac{1}{2}}\norm{\psi_{\ell_1}}\norm{\xi_{\ell_2}}\notag\\
 &= C_r e^{-2|s|r}
    \big(\sum_{\ell_1} d_{\ell_1}^{\frac{1}{2}}\,\|\psi_{\ell_1}\|\big)
    \big(\sum_{\ell_2} d_{\ell_2}^{\frac{1}{2}}\,\|\xi_{\ell_2}\|\big)\overset{\text{(1)}}{\leq} C_{r,t} e^{-2|s|r}\norm{\psi}_t\norm{\xi}_t.
\end{align}
Here in $(1)$ we use \eqref{for:1}.
\end{proof}
\begin{remark}\label{re:13}
Since $C_{r,t}$ depends only on $r$, $t$ and the choice of Sobolev norms (and
not on the particular representation $\rho$ or the particular vectors
$\psi,\xi$), the estimate \eqref{for:242} is valid for any representation
of $SL(2,\mathbb C)$ with decay rate $r$. It also extends to any representation
of a Lie group $P$ with $\Lie(P)=\mathfrak{sl}(2,\mathbb C)$ that has decay
rate $r$ in the above sense.

In particular, for \(SL(2,\mathbb C)\) there is no irreducible unitary representation
whose matrix coefficients decay faster than the tempered rate \(r=1-\epsilon\), so
\eqref{for:270} covers all non-trivial irreducible unitary representations.
For other groups that do admit representations with strictly faster decay (for example,
the discrete series of \(SL(2,\mathbb R)\)), the optimal decay rate for each representation
may exceed the tempered rate, so the exponent \(r\) in \eqref{for:270} must be interpreted
as the tempered exponent rather than the best possible one.
\end{remark}

\begin{remark}
For \(SL(2,\RR)\), the same \(K\)-type argument yields the usual Sobolev-form
decay estimates for principal, complementary, and mock-discrete series. However,
in the discrete-series case it only captures the tempered exponent and does not
recover the representation-specific optimal decay rate. This is precisely the
regime where the dynamical argument developed below becomes essential.
\end{remark}

\subsubsection{The real case}\label{sec:7}
We introduce a dynamical method, based on the study of cohomological equations,
to obtain decay of matrix coefficients in the \(SL(2,\RR)\) case. This is in
contrast with the traditional representation-theoretic approach via \(K\)-types,
spherical functions, and Harish-Chandra's Plancherel theorem (see \cite{tan}). The method
recovers the optimal decay exponent for representations of \(SL(2,\RR)\).

More precisely, we combine solvability results for the cohomological equation of
the horocycle flow with Sobolev interpolation to derive quantitative decay
estimates for matrix coefficients.

\begin{proposition}\label{cor:4} Let $P$ be a Lie group with Lie algebra $\mathfrak{sl}(2,\RR)$ and $(\rho,H_\rho)$ be a unitary representation of $P$ with \textbf{strong} spectral gap $\gamma$ (see Section \ref{for:222}).  Assume:
 \begin{align*}
  p=\max\{1, \gamma\}\quad\text{ and }\quad\zeta>\gamma+2.
 \end{align*}
 Then for any $\xi\in W^{p}(H_\rho)$, $\psi\in W^{\zeta}(H_\rho)$ and
  any $a_s = \begin{pmatrix} e^{s} & 0 \\[5pt]
0 & e^{-s} \end{pmatrix}$, $s\in\RR$ we have
  \begin{align*}
\big|\langle \rho(a_s)\psi,\, \xi\rangle \big|\leq C_{\gamma,\epsilon}e^{-2|s|(\gamma-\epsilon)} \| \xi \|_{p} \big\|\psi\big\|_{\zeta}.
\end{align*}
\end{proposition}
We first prove the following more precise rank-one decay estimate, from which Proposition~\ref{cor:4} follows immediately.

\begin{lemma}\label{for:72} Let $P$ be a Lie group with Lie algebra $\mathfrak{sl}(2,\RR)$ and $(\rho,H_\rho)$ be a unitary representation of $P$.  Assume the following conditions are satisfied:
 \begin{enumerate}
   \item\label{for:155} $\rho$ has a spectral gap of $\mu_0$ (see Section \ref{sec:1}) and $\rho$ has no $P$-fixed vectors,

    \smallskip
   \item\label{for:214} if $\nu_0\geq0$, then
   \[
  \frac{1 + \nu_0}{2}<p<1,\qquad r>0,\qquad r+\frac{1 + \nu_0}{2}<1,
  \]
   \item\label{for:216} if $\nu_0<0$, then
  \begin{align*}
  0< p<\big|\frac{1 + \nu_0}{2}\big|+\frac{1}{2} \ \text{ and }\ 0<r<p.
  \end{align*}

\end{enumerate}
Let
\begin{align*}
 \zeta=2\lfloor p\rfloor+2-r\quad\text{and}\quad a_s = \begin{pmatrix} e^{s} & 0 \\[5pt]
0 & e^{-s} \end{pmatrix},\,s\in\RR.
\end{align*}
Then for any $\xi\in W^{p}(H_\rho)$, $\psi\in W^{\zeta}(H_\rho)$ and
  any $s\in\RR$ we have
  \begin{align*}
\big|\langle \rho(a_s)\psi,\, \xi\rangle \big|\leq C_{\nu_0,p,r}e^{-2|s|r} \| \xi \|_{p} \big\|\psi\big\|_{\zeta}.
\end{align*}

\end{lemma}

\begin{proof}
Note the conjugation relations (see \eqref{for:95} of Section \ref{sec:1})
\[
\rho(a_s)\,U\,\rho(a_s)^{-1}=e^{2s}U,\quad \rho(a_s)\,X\,\rho(a_s)^{-1}=X,\quad \rho(a_s)\,V\,\rho(a_s)^{-1}=e^{-2s}V.
\]
 This shows that
\begin{align*}
 \norm{\rho(a_s)\phi}_q\leq e^{2q|s|}\norm{\phi}_{q},\qquad \forall\,\phi\in W^{q}(H_\rho),\quad \forall\,q\in\NN
\end{align*}
Since $\rho(a_s)$ is unitary on $H_\rho$,
\begin{align*}
 \norm{\rho(a_s)\phi}=\norm{\phi}.
\end{align*}
By interpolation between $W^0$ and $W^q$ we obtain
\begin{align}\label{for:71}
 \|\rho(a_s)\|_{W^{q}(H_\rho)\to W^{q}(H_\rho)}\le C_qe^{2q|s|},\qquad \forall\, q\geq0.
\end{align}
Choose $Y\in\{U,V\}$ so that
\begin{align}\label{for:69}
\rho(a_{s})^{-1}\,Y\,\rho(a_s)=e^{-2|s|}Y;
\end{align}
explicitly $Y=U$ if $s>0$ and $Y=V$ if $s<0$.

Case $I$: $\nu_0\in [0,1)$ (complementary/principal series/mock-discrete)  By assumption \eqref{for:214},
\[
p>\frac{1+\nu_0}{2},\qquad 0<r<1,\qquad
-(1-r)<-\frac{1+\nu_0}{2}.
\]
 Then by \eqref{for:68} of Theorem \ref{th:5}, the equation
\begin{align}
 Y\eta=\xi
\end{align}
has a solution $\eta\in W^{-(1-r)}(H_\rho)$,  which satisfies the Sobolev estimate
\begin{align}\label{for:70}
 \|\eta \|_{-(1-r)} \leq C_{\nu_0,p,r} \| \xi \|_{p}.
\end{align}
Here \(0<1-r<1\), so the solution lies in a negative Sobolev space of order strictly less than one, which is exactly what allows the subsequent conjugation argument to produce exponential decay.

For any $\psi\in W^{2-r}(H_\rho)$, we have
\begin{align}\label{for:218}
 \big|\langle \rho&(a)\psi,\, \xi\rangle \big|=\big|\langle \rho(a_s)\psi,\, Y\eta\rangle \big|=\big|\langle Y(\rho(a_s)\psi),\, \eta\rangle \big|\notag\\
 &\overset{\text{(1)}}{=}e^{-2|s|}\big|\langle \rho(a_s)(Y\psi),\, \eta\rangle \big|\overset{\text{(2)}}{\leq} e^{-2|s|} C_{\nu_0,p,r} \| \xi \|_{p} \big\|\rho(a_s)(Y\psi)\big\|_{1-r}\notag\\
 &\overset{\text{(3)}}{\leq} e^{-2|s|} e^{2|s|(1-r)}\,C_{\nu_0,p,r,1} \| \xi \|_{p} \big\|Y\psi\big\|_{1-r}\notag\\
 &\leq C_{\nu_0,p,r,1} e^{-2|s|r}  \| \xi \|_{p} \big\|\psi\big\|_{2-r}.
\end{align}
Here in $(1)$ we use \eqref{for:69}; in $(2)$ we use \eqref{for:70}; in $(3)$ we use \eqref{for:71}. This proves the lemma in Case $I$, with
$\zeta=2\lfloor p\rfloor+2-r=2-r$.

\smallskip

Case $II$: $\nu_0=-n+1$, $n\in\NN$, $n\geq2$ (see \eqref{for:213} of Remark \ref{re:9}) and $0<p<\frac{n}{2}$. In this case, $\rho$ only contains discrete series.
It follows from Theorem \ref{th:18} that
if $\mathcal{D}$ is a $U$-invariant distribution for $\rho$, then $\mathcal{D}\in W^{-s}(H_\rho)$, $s>\frac{n}{2}$. In particular, there is no
$Y$-invariant distribution of Sobolev order $\le \frac{n}{2}$. Hence for any
$0<s\le \frac{n}{2}$ there are no nontrivial $Y$-invariant distributions
of order $s$. By \eqref{for:217} of Theorem \ref{th:5}, for any $\xi\in W^s(H_\rho)$, if
\[
-\frac{n}{2}+1=\frac{1+\nu_0}{2}<s\le \frac{n}{2},
\]
then the equation
\begin{align*}
 Y\eta=\xi
\end{align*}
has a solution $\eta\in W^{t}(H_\rho)$, $t<s-1$ with the Sobolev estimate
\begin{align*}
 \|\eta \|_{t} \leq C_{\nu_0,s,t} \| \xi \|_s.
\end{align*}
If $\frac{n}{2}\in\ZZ$,  set
\[
b=
\begin{cases}
\lfloor p\rfloor,& \text{if } p\in\ZZ,\\[4pt]
\lfloor p\rfloor+1,& \text{if } p\notin\ZZ.
\end{cases}
\]
If $\frac{n}{2}\notin\ZZ$,  set
\begin{align*}
b=
\begin{cases}
\lfloor p\rfloor,& \text{if } p\in\ZZ,\\[4pt]
\lfloor p\rfloor+1,& \text{if } p\notin\ZZ\text{ and }p<\frac{n-1}{2}.
\end{cases}
\end{align*}
Then $p\leq b\leq p+1$ and $1\le b\le \frac{n}{2}$, so we can iterate the solvability for the
cohomological equation $Y\eta=\xi$ exactly $b$ times. Then the equation
\begin{align*}
 Y^b\eta=\xi
\end{align*}
has a solution $\eta\in W^{t}(H_\rho)$, where $t=-(b-r)<0$ (we recall that $r<p\le b$) with the Sobolev estimate
\begin{align}\label{for:240}
 \|\eta \|_{t} \leq C_{\nu_0,p,r} \| \xi \|_{p}.
\end{align}
Similar to \eqref{for:218}, for any $\psi\in W^{b+|t|}(H_\rho)$, we have
\begin{align}\label{for:174}
 \big|\langle \rho&(a_s)\psi,\, \xi\rangle \big|=\big|\langle \rho(a_s)\psi,\, Y^b\eta\rangle \big|=\big|\langle Y^b(\rho(a_s)\psi),\, \eta\rangle \big|\notag\\
 &=e^{-2b|s|}\big|\langle \rho(a_s)(Y^b\psi),\, \eta\rangle \big|\overset{\text{(1)}}{\leq} e^{-2b|s|} C_{\nu_0,p,r} \| \xi \|_{p} \big\|\rho(a_s)(Y^b\psi)\big\|_{|t|}\notag\\
 &\overset{\text{(2)}}{\leq} e^{-2b|s|} e^{2|s||t|}\,C_{\nu_0,p,r,1} \| \xi \|_{p} \big\|Y^b\psi\big\|_{|t|}\notag\\
 &\leq C_{\nu_0,p,r,1} e^{-2|s|(b-|t|)}  \| \xi \|_{p} \big\|\psi\big\|_{b+|t|}\notag\\
 &\overset{\text{(3)}}{\leq} C_{\nu_0,p,r,1} e^{-2|s|r}  \| \xi \|_{p} \big\|\psi\big\|_{2\lfloor p\rfloor+2-r}
\end{align}
Here in $(1)$ we use \eqref{for:240}; in $(2)$ we use \eqref{for:71}; in $(3)$ a direct computation shows that
\begin{align*}
 b+|t|\leq2\lfloor p\rfloor+2-r\quad\text{and}\quad b-|t|=r.
\end{align*}
Thus
\[
  \big|\langle \rho(a_s)\psi,\, \xi\rangle \big|
  \leq C_{\nu_0,p,r}\, e^{-2|s|r}\,\|\xi\|_{p}\,\|\psi\|_{\zeta}
\]
for all $\psi\in W^{\zeta}(H_\rho)$, with $\zeta=2\lfloor p\rfloor+2-r$.

Suppose $\frac{n}{2}\notin\ZZ$ and $\frac{n-1}{2}<p<\frac{n}{2}$.   Set $b=\frac{n-1}{2}$. Then the equation
\begin{align*}
 Y^b\eta_1=\xi
\end{align*}
has a solution $\eta_1\in W^{p-\frac{n-1}{2}-\varepsilon}(H_\rho)$ (where $\varepsilon>0$ is sufficiently small and will be specified below) with the Sobolev estimate
\begin{align}
 \|\eta_1 \|_{t} \leq C_{\nu_0,p,\varepsilon} \| \xi \|_{p}.
\end{align}
Moreover, the equation
\begin{align*}
 Y\eta=\eta_1
\end{align*}
has a solution $\eta\in W^{t}(H_\rho)$, where $t=-\big(1-(p-\frac{n-1}{2}-\varepsilon)+\varepsilon\big)$ with the Sobolev estimate
\begin{align}
 \|\eta \|_{t} \leq C_{\nu_0,p,\varepsilon} \| \eta_1 \|_{p-\frac{n-1}{2}-\varepsilon}\leq C_{\nu_0,p,\varepsilon,1} \| \xi \|_{p}.
\end{align}
Then $\eta$ satisfies the equation
\begin{align*}
 Y^\frac{n+1}{2}\eta=\xi.
\end{align*}
Again arguing as in \eqref{for:174}, for any $\psi\in W^{\frac{n+1}{2}+|t|}(H_\rho)$, we have
\begin{align*}
 \big|\langle \rho&(a)\psi,\, \xi\rangle \big|\leq e^{-2\frac{n+1}{2}|s|} e^{2|s||t|}\,C_{\nu_0,p,\varepsilon} \| \xi \|_{p} \big\|Y^\frac{n+1}{2}\psi\big\|_{|t|}\notag\\
 &\leq C_{\nu_0,p,\varepsilon} e^{-2|s|(\frac{n+1}{2}-|t|)}  \| \xi \|_{p} \big\|\psi\big\|_{\frac{n+1}{2}+|t|}\overset{\text{(1)}}{\leq} C_{\nu_0,p,r} e^{-2|s|r}  \| \xi \|_{p} \big\|\psi\big\|_{\zeta}.
\end{align*}
Here in $(1)$ we note that
\begin{align*}
 \frac{n+1}{2}-|t|&=p-2\varepsilon> r\qquad\text{and}\\
 \frac{n+1}{2}+|t|&=n+1-p+2\varepsilon=2\lfloor p\rfloor+2-(p-2\varepsilon)<2\lfloor p\rfloor+2-r=\zeta
\end{align*}
and we choose $\varepsilon=\min\{\frac{p-r}{2}, \frac{p-\frac{n-1}{2}}{2}\}$ such that $p-2\varepsilon>r$ (recall $p>r$).

Combining Case $I$ and Case $II$ completes the proof.
\end{proof}
\begin{proof}[Proof of Proposition~\ref{cor:4}]
Choose \(r=\gamma-\epsilon\). By the definition of the strong spectral gap and
the parametrization in Remark~\ref{re:9}, the hypotheses of Lemma~\ref{for:72}
are satisfied with \(p=\max\{1,\gamma\}\) and any \(\zeta>\gamma+2\). The
claimed estimate follows.
\end{proof}

\begin{remark}\label{re:10}
When $\nu_0\in [0,1)$, the optimal decay rate exponent (for both geodesic flow and horocycle flow) in our normalization is $r$, where
\[
r \;=\; 1 - \frac{1+\nu_0}{2} - \epsilon \;=\; \frac{1-\nu_0}{2} - \epsilon.
\]
This agrees with the classical optimal decay exponent for the complementary/principal/mock-discrete series with parameter $\nu_0$ (see \cite{howe-moore}).

When $\nu_0 = -n+1<0$, the optimal decay exponent is $r$, where
\[
r \;=\; 1 - \frac{1+\nu_0}{2} - \epsilon \;=\; \frac{n}{2} - \epsilon,
\]
which coincides with the optimal classical decay rate for the discrete series of lowest $K$ (see $\Theta$ in \eqref{for:95})-weight $n$ (see \cite{Warner}), but not in the form of quantitative upperbounds. In addition, our estimates provide an explicit dependence of the decay exponent $r$ and the Sobolev regularity parameters $\gamma$ and $p$ of the observables.
\end{remark}

\begin{remark}\label{re:14} We can use the decay rate obtained for $SL(2,\RR)$ to study decay for $SL(2,\CC)$ by restricting to an $SL(2,\RR)$-subgroup, even though in some cases this does not yield the optimal decay rate for $SL(2,\CC)$ itself.

Let $\rho$ be an irreducible unitary principal series representation of
$SL(2,\CC)$. Its restriction to $SL(2,\RR)$ decomposes into a direct integral
of irreducible unitary representations of $SL(2,\RR)$, including principal,
mock-discrete and discrete series components (see, for example, \cite{BV} for
complementary series and Steinberg components). In this case, in the notation
of Lemma~\ref{for:72}, we have $\nu_0=0$. If we apply Lemma~\ref{for:72} only
to this restricted $SL(2,\RR)$ representation, using the $SL(2,\RR)$ spectral
parameter, we obtain a decay exponent $r = \tfrac{1}{2} - \epsilon.$ This coincides with the tempered exponent for $SL(2,\RR)$, but is strictly
smaller than the optimal exponent $r = 1 - \epsilon$ available for the
original $SL(2,\CC)$ representation.

The rank-one argument developed here provides the basic analytic input for the general semisimple case. Indeed, if \(\pi\) is a unitary representation of \(G\) with strong spectral gap and without non-trivial \(G\)-invariant vectors, then the restriction of \(\pi\) to each relevant embedded subgroup \(SL(2,k)\) is isolated from the trivial representation. This allows the rank-one estimates proved here to be applied inside \(G\), by iterating along suitable commuting rank-one subgroups, and hence yields quantitative upper bounds for matrix coefficients, although in general these bounds need not be optimal.

To the best of our knowledge, this is the first approach that derives such decay estimates entirely outside the traditional \(K\)-finite framework: the bounds are obtained by analytic and dynamical methods, via cohomological equations and Sobolev interpolation, rather than through spherical analysis.

\end{remark}

\begin{remark}[Classical versus fractional cohomological methods]
In Lemma~\ref{for:72} we use the classical cohomological equation to obtain
exponential decay of matrix coefficients. This argument inherently requires
\emph{full} Sobolev control of the observables. Indeed, the solution
\(\eta\) of \(Y\eta=\xi\) is only a distributional solution, lying in a
negative Sobolev space \(W^t(H_\rho)\) with \(t<0\), and this is a
\emph{global} Sobolev space rather than a directional one. Consequently,
to estimate
\[
\langle \rho(a)\psi,\xi\rangle
=
\langle Y(\rho(a)\psi),\eta\rangle,
\]
one must bound \(\|\rho(a)Y\psi\|_{|t|}\), which involves full Sobolev
regularity of \(\psi\), not just derivatives along \(Y\).

By contrast, if one solves the fractional equation
\[
|Y|^r\eta=\xi,\qquad 0<r<1,
\]
with \(\eta\in H_\rho\), then the relevant quantity becomes
\(\||Y|^r\rho(a)\psi\|\). Since the line \(\mathbb{R}Y\) is
\(\Ad_{a^{-1}}\)-invariant, this requires only directional regularity
along \(Y\). This is the key reason that the fractional method developed
below yields decay estimates under \emph{partial} Sobolev norms, and
ultimately allows partial regularity for both observables in the
semisimple setting.
\end{remark}

\subsection{Single-factor fractional solvability}
In this part, we study the fractional equation for a Lie group $P$ with Lie algebra $\mathfrak{sl}(2,k)$. Let $(\rho,H_\rho)$ be a unitary representation of $P$ with \textbf{strong} spectral gap $\gamma$ (see Section \ref{for:222}).  Assume:
\begin{enumerate}
  \item  if $k=\CC$, then $p=\zeta>1$,

   \item if $k=\RR$, then $p=\max\{1, \gamma\}$ and $\zeta>\gamma+2$.

\end{enumerate}
Set
\begin{align}\label{for:252}
 \Sigma\;=\;
  \begin{cases}
    I - X^2 - V^2, &
      \text{if } k = \RR,\\[4pt]
    I - X^2 - (\mathrm{i}X)^2 - V^2 - (\mathrm{i}V)^2, &
      \text{if } k= \CC
  \end{cases}
\end{align}
(see \eqref{for:95} of Section \ref{sec:1}).
\begin{proposition} \label{po:4}  For any $\xi\in H_\rho$ with $\Sigma^{\frac{\zeta}{2}}\xi\in H_\rho$ we have:
\begin{enumerate}
  \item If $k=\RR$, for any $0\leq r<\gamma$ the fractional equation $|U|^r\omega=\xi$ has a solution $\omega\in H_\rho$ with the estimate
  \begin{align*}
 \norm{\omega}\leq C_{\gamma,r}\big\|\Sigma^{\frac{\zeta}{2}}\xi\big\|.
\end{align*}
  \item If $k=\CC$, for any $0\leq r<\gamma$ and any $a,b\geq0$ with $a+b=r$, then the fractional equation
\begin{align*}
  |U|^{r}\omega_1+|U|^{a}|\emph{i} U|^{b}\omega_2+|\emph{i}U|^{r}\omega_3=\xi
  \end{align*}
  has solutions $\omega_1,\,\omega_2,\,\omega_3\in H_\rho$ with the estimate
  \begin{align*}
\max_{1\leq i\leq 3} \norm{\omega_i}\leq C_{\gamma,r}\big\|\Sigma^{\frac{\zeta}{2}}\xi\big\|.
\end{align*}
 Moreover, $|U|^{r}\omega_1,\,|U|^{a}|\emph{i} U|^{b}\omega_2,\,|\emph{i}U|^{r}\omega_3\in H_\rho$.

\end{enumerate}

\end{proposition}
Proposition~\ref{po:4} is proved in two steps. We first relate decay of
matrix coefficients to solvability of the fractional equation. More
precisely, the optimal decay exponent determines exactly the threshold
for solvability of the fractional equation. The idea is that decay gives
precise control of the spectral behaviour near \(\lambda=0\), which
allows one to divide by \( |\lambda|^r \) and solve the equation. This
is quite different from the classical invariant-distribution/Green-operator
method of Flaminio--Forni for the cohomological equation. We then refine the resulting solvability statement to obtain estimates in
partial Sobolev norms, which is the form needed later in the proof of
Theorem~\ref{th:6}.

\subsubsection{Solvability from horocycle decay} As a first step in the proof of Proposition~\ref{po:4}, we pass from
exponential decay along the geodesic flow to polynomial decay of matrix
coefficients along the horocycle flow. In geodesic time $s$ we obtain
\emph{exponential} decay $e^{-2|s|r}$, while in horocycle time $t$ (with
$s=\operatorname{arcsinh}(|t|/2)$) this translates into \emph{polynomial}
decay of order $(1+|t|)^{-2r}$.

\begin{corollary}\label{cor:5} Let $(\rho,H_\rho)$ be as in Corollary~\ref{cor:4}.
For any $\xi\in W^{p}(H_\rho)$, $\psi\in W^{\zeta}(H_\rho)$, we have:
\begin{enumerate}
  \item If $k=\RR$, then for any $t\in\RR$
  \begin{align*}
\Big|\big\langle \rho(\exp(tU))\psi,\, \xi\big\rangle \Big|\leq C_{\gamma,\epsilon}(|t|+1)^{-2(\gamma-\epsilon)}  \| \xi \|_{p} \big\|\psi\big\|_{\zeta}.
\end{align*}
  \item If $k=\CC$, then for any $t=t_1+t_2 \emph{i}\in\CC$ (so
        $|t|=\sqrt{t_1^2+t_2^2}$),
  \begin{align*}
\Big|\big\langle \rho(\exp(t_1U+t_2 \emph{i}U))\psi,\, \xi\big\rangle \Big|\leq C_{\gamma,\epsilon}(|t|+1)^{-2(\gamma-\epsilon)}  \| \xi \|_{p} \big\|\psi\big\|_{\zeta}.
\end{align*}
\end{enumerate}

\end{corollary}

\begin{proof} If $k=\RR$ (resp.\ $k=\CC$), for any $t\in\RR$ (resp.\ $t\in\CC$) $u_t=\exp\!\begin{pmatrix} 0 & t \\[4pt] 0 & 0 \end{pmatrix}$
has a $KAK$–decomposition $u_t = k_1 a k_2$ with
\[
  k_i \in \Bigl\{\exp\!\begin{pmatrix} 0 & r \\[4pt] -r & 0 \end{pmatrix}
          : r\in\RR\Bigr\}\quad\text{if }k=\RR,
\]
and
\[
  k_i \in \Bigl\{\exp\!\begin{pmatrix} b\mathrm{i} & c\\[4pt] -\overline{c} &
           -b\mathrm{i} \end{pmatrix} : b\in\RR,\,c\in\CC\Bigr\}\quad\text{if }k=\CC,
\]
while
\[
  a = \exp\!\begin{pmatrix} s & 0 \\[5pt] 0 & -s \end{pmatrix},\qquad
  s=\operatorname{arcsinh}(|t|/2).
\]
Similar to \eqref{for:71} in the proof of Lemma \ref{for:72}, by noting that there is an $\text{Ad}_K$-invariant norm on $\mathfrak{sl}(2,\RR)$, we have
\begin{align*}
 \norm{\rho(k)\phi}_q\leq C_q\norm{\phi}_q,\qquad \forall\,\phi\in W^{q}(H_\rho),\quad \forall\,k\in K,\,\,\forall\,q\in\NN\cup\{0\}.
\end{align*}
Combined with the unitarity of $\rho(K)$ on $H_\rho$, we have
\begin{align}\label{for:73}
 \|\rho(k)\|_{W^{q}(H_\rho)\to W^{q}(H_\rho)}\le C_q,\qquad \forall\,k\in K,\,\,\forall\,q\geq 0.
\end{align}
we write
\[
  \exp(tU) :=
  \begin{cases}
    \exp(tU), & k=\RR,\\[4pt]
    \exp(t_1U+t_2\mathrm{i}U), & k=\CC.
  \end{cases}
\] It follows from Corollary \ref{cor:4}  that
\begin{align*}
 \Big|\big\langle \rho&(\exp(tU))\psi,\, \xi\big\rangle \Big|=\Big|\big\langle \rho(a)\rho(k_2)\psi,\, \rho(k_1^{-1})\xi\big\rangle \Big|\\
 &\leq C_{\gamma,\epsilon}e^{-2\operatorname{arcsinh}(|t|/2)(\gamma-\epsilon)}  \| \rho(k_1^{-1})\xi \|_{p} \big\|\rho(k_2)\psi\big\|_{\zeta}\\
 &\overset{\text{(1)}}{\leq} C_{\gamma,\epsilon,1}(|t|+1)^{-2(\gamma-\epsilon)}  \| \xi \|_{p} \big\|\psi\big\|_{\zeta}.
\end{align*}
Here in $(1)$ we use \eqref{for:73} and $e^{-2(\gamma-\epsilon)|s|}\asymp (1+|t|)^{-2(\gamma-\epsilon)}$.  Then we complete the proof.

\end{proof}

The next result reveals the connection between solvability of the fractional
equation and the optimal decay exponent for matrix coefficients in $(\rho,H_\rho)$.
\begin{lemma}\label{le:13}
For any $r>\gamma$, there is $\xi\in W^{\infty}(H_\rho)$ such that:
\begin{enumerate}
  \item If $k=\RR$, the fractional equation $ |U|^r\omega=\xi$ has no solution $\omega\in H_\rho$.
  \item If $k=\CC$, for any $a,b\geq0$ with $a+b=r$, \emph{neither} of the
  fractional equations
  \[
    |U|^{a}|\mathrm{i} U|^{b}\omega=\xi
    \quad\text{nor}\quad
    |U|^{r}\omega_1+|\mathrm{i} U|^{r}\omega_2=\xi
  \]
  admits a solution $\omega\in H_\rho$ or a pair of solutions
  $\omega_1,\,\omega_2\in H_\rho$, respectively.
\end{enumerate}
\end{lemma}

\begin{proof} Fix any $r>\gamma$. By optimality of the decay exponent of $\gamma$ (i.e., by the
definition of the strong spectral gap), we can choose
$\xi\in W^\infty(H_\rho)$ such that the matrix coefficients
\[
\langle \rho(a_s)\xi,\xi\rangle,\qquad a_s=\exp\begin{pmatrix} s & 0 \\[5pt]
0 & -s \end{pmatrix}
\]
do \emph{not} satisfy
\[
\big|\langle \rho(a_s)\xi,\xi\rangle\big|
\;\le\; C\,e^{-2r|s|}\qquad \forall s\in\mathbb{R},
\]
for any constant $C>0$. We claim that the fractional equation $|U|^{r}\omega=\xi$ has no solution $\omega\in \mathcal{H}_{\mu}$. Suppose, to the contrary, that such
$\omega\in H_\rho$ exists.

First consider $s\geq0$. Then
\begin{align*}
 \big|\langle \rho(a_s)\xi,\, \xi\rangle \big|&\overset{\text{(1)}}{=}\big|\langle \rho(a_s)\xi,\, |U|^{r}\omega\rangle \big|=\big|\langle |U|^{r}(\rho(a_s)\xi),\, \omega\rangle \big|\notag\\
 &\overset{\text{(2)}}{=}e^{-2rs}\big|\langle \rho(a_s)(|U|^{r}\xi_\varepsilon),\, \omega\rangle \big|\overset{\text{(3)}}{\leq} e^{-2rs} \big\|\rho(a_s)(|U|^{r} \xi)\big\|\| \omega \|\notag\\
 &\overset{\text{(4)}}{\leq} e^{-2rs}  \big\||U|^{r} \xi\big\|\| \omega \|.
\end{align*}
Here in $(1)$ we use \eqref{for:163} of Lemma \ref{le:17}; in $(2)$ we use \eqref{for:169}  of Lemma \ref{le:17}:
\begin{align*}
 |U|^{r}\rho(a_s)=(\text{Ad}_{a_{-s}}U)^{r}\rho(a_s)|U|^{r}=e^{-2rs}\rho(a_s)|U|^{r};
\end{align*}
in $(3)$ we use we use Cauchy-Schwarz inequality; in $(4)$ we use we use unitarity of $\rho$.

Since $s\le 0$, we have $e^{2rs}=e^{-2r|s|}$, and therefore
\[
 \big|\langle \rho(a_s)\xi,\, \xi\rangle \big|
 \;\le\; C' e^{-2r|s|},\quad s\le 0,\quad\text{where } C':=\; \big\||U|^{r} \xi\big\|\,\| \omega \|.
\]
For $s>0$, we use unitarity again:
\[
 \big|\langle \rho(a_s)\xi,\,\xi\rangle\big|
 =\big|\langle \xi,\,\rho(a_{-s})\xi\rangle\big|.
\]
Applying the $s\le 0$ estimate with $-s$ in place of $s$ gives
\[
 \big|\langle \rho(a_s)\xi,\,\xi\rangle\big|
 \;\le\; C' e^{-2r|s|},\qquad s>0.
\]
Combining both cases, we obtain
\[
 \big|\langle \rho(a_s)\xi,\,\xi\rangle\big|
 \;\le\; C' e^{-2r|s|},\qquad \forall s\in\mathbb{R}.
\]
This contradicts the optimal rate fact. Therefore, no solution $\omega\in H_\rho$ can exist for the equation
$|U|^r\omega=\xi$.

Now suppose $k=\CC$. The argument is analogous, using the conjugation relations
for both $U$ and $\mathrm{i}U$ under $a_s$:
\begin{align*}
|U|^{c}|\mathrm{i}U|^{d}\rho(a_s)=e^{-2(c+d)s}\rho(a_s)|U|^{c}|\mathrm{i}U|^{d},\qquad \forall\,c,d\geq0.
\end{align*}
This again yields an overall bound
  of the form $C'''e^{-2r|s|}$ for all $s\in\RR$, which is impossible
  by the choice of $\xi$. In both complex equations we arrive at the same contradiction, so neither admits
a solution in $H_\rho$. This completes the proof.

\end{proof}

The next result solves the fractional equation under the assumption of strong spectral gap,
which almost proves Proposition \ref{po:4}. The only missing ingredient is a
\emph{partial} Sobolev bound for $\omega$.
\begin{lemma}\label{le:23}
Let $(\rho,H_\rho)$ be as in Corollary~\ref{cor:4}. For any $\xi\in W^{\zeta}(H_\rho)$, we have:
\begin{enumerate}
  \item\label{for:248} If $k=\RR$, then for any $0<r<\gamma$, the fractional equation $|U|^r\omega=\xi$ has a solution $\omega\in H_\rho$ with the estimate $\norm{\omega}\leq C_{\gamma,r}\big\|\xi\big\|_{\zeta}$.

  \item\label{for:250} If $k=\CC$, then for any $0<r<\gamma$ and any $a,b\geq0$ with $a+b=r$, the fractional equation $|U|^{a}|\emph{i} U|^{b}\omega=\xi$
has a solution $\omega\in H_\rho$ with the estimate $\norm{\omega}\leq C_{\gamma,r}\big\|\xi\big\|_{\zeta}$.
\end{enumerate}

\end{lemma}
\begin{proof} If $k=\RR$, set $S=P$ and $Y=U$. If $k=\CC$, let $S$ be the connected
subgroup of $P$ whose Lie algebra is spanned either by $X,U,V$ (with $Y=U$),
or by $X,\mathrm{i}U,\mathrm{i}V$ (with $Y=\mathrm{i}U$). In any case,
$\Lie(S)=\mathfrak{sl}(2,\RR)$. We have a direct integral decomposition:
\begin{align*}
  \rho|_{S}=\int_Z \rho_z d\kappa(z),\quad H_\rho=\int_Z (H_\rho)_z d\kappa(z)
\end{align*}
of irreducible unitary representations $(\rho_z,(H_\rho)_z)$ of $S$ for some measure space $(Z,\kappa)$ (see Section \ref{sec:20}). By
Remark~\ref{re:4}, each $(\rho_z,(H_\rho)_z)$ is unitarily equivalent to an
irreducible representation $(\pi_{\mu(z)},\mathcal{H}_{\mu(z)})$ of
$SL(2,\RR)$, where the Casimir operator acts as a constant $\mu(z)$ on
$\mathcal{H}_{\mu(z)}$. We write $\xi=\int_Z\xi_zd\kappa(z)$ and set $\varpi(z) := \sqrt{1 - \mu(z)}$, $z\in Z$.

For each $(\pi_{\mu(z)},\,\mathcal{H}_{\mu(z)})$, we consider the Fourier model in Section \ref{sec:3}.  Recalling that $Y$ acts by multiplication with $-\lambda\mathrm{i}$.
For any $z\in Z$ let
\begin{align}\label{for:284}
 \theta_z(\lambda)=C_{\Re\varpi(z)}|\xi_z(\lambda)|^2 \vert \lambda\vert^{-\Re\varpi(z)}\geq0.
\end{align}
\emph{Idea of the proof.}
Formally, if we work in the Fourier model, the solution of
\(|U|^r\omega=\xi\) would be obtained by dividing by the symbol
\(|\lambda|^r\):
\[
 \omega_{\text{formal}}=\int_Z \frac{\xi_z(\lambda)}{|\lambda|^{r}}\;d\kappa(z).
\]
To show that this is a bona fide solution in $H_\rho$, we need to verify
\begin{align}\label{for:239}
  \norm{\omega}^2&=  \int_Z \left\|\frac{\xi_z(\lambda)}{|\lambda|^{r}}\right\|^2 d\kappa(z)=\int_Z C_{\Re\varpi(z)} \int_{\R} |\xi_z(\lambda)|^2|\lambda|^{-2r} \vert \lambda\vert^{-\Re\varpi(z)} d\lambda d\kappa(z)\notag\\
  &=\int_Z\int_{\R}\theta_z(\lambda)|\lambda|^{-2r}d\lambda d\kappa(z)=\int_{\mathbb R}\tau(\lambda)|\lambda|^{-2r}\,d\lambda<\infty,
\end{align}
where
\begin{align}\label{for:285}
\tau(\lambda):=\int_Z\theta_z(\lambda)\,d\kappa(z).
\end{align}
The main difficulty is the behaviour near $\lambda=0$: the factor
\(|\lambda|^{-2r}\) is singular, and we must control the low-frequency
part of the spectral measure. A direct computation with the Fourier transform
of \(|\lambda|^{-2r}\) is only straightforward when $2r$ is not an integer, and
it is technically more convenient to avoid this issue.

For this reason we proceed in two steps. First, we choose an exponent
$r'$ with $r<r'<\gamma$ and $2r'\notin\mathbb Z$, and solve
\(|Y|^{r'}\omega_1=\xi\). This step uses the decay of matrix coefficients
(coming from the strong spectral gap) to show that
\(\int_{\mathbb R}\tau(\lambda)|\lambda|^{-2r'}d\lambda<\infty\).
Once we know that \(|Y|^{r'}\omega_1=\xi\) has a solution with
$\omega_1\in H_\rho$, we can then solve \(|Y|^r\omega=\xi\) for any
$r<r'$ using the spectral calculus for the one-parameter group generated
by $U$: replacing \(|\chi|^{-r'}\) by the less singular weight
\(|\chi|^{-r}\) improves integrability near $\chi=0$, and the spectral
gap excludes mass at $\chi=0$.

The Fourier model is used here for two purposes. First, in this model $Y$ acts in the simple form of multiplication by $-\mathrm{i}\lambda$.
Second, the Plancherel measure $|\lambda|^{-\Re\varpi(z)}\,d\lambda$ allows us
to express matrix coefficients as Fourier transforms of $\tau$, thereby
linking the polynomial decay of the horocycle flow to the integrability
condition in \eqref{for:239}, which is exactly what we need to solve
the fractional equation.

\smallskip

From \eqref{for:284}, we have
\begin{align*}
 \int_{\R} \theta_z(\lambda)d\lambda=\norm{\xi_z}_{\mathcal{H}_{\mu(z)}}^2,
\end{align*}
and hence
\begin{align}\label{for:236}
 \int_Z\int_{\R} \theta_z(\lambda)d\lambda d\kappa(z)=\int_Z\norm{\xi_z}_{\mathcal{H}_{\mu(z)}}^2d\kappa(z)=\norm{\xi}^2.
\end{align}
By Tonelli's theorem, $\tau(\lambda)$ (see \eqref{for:285}) is well defined and belongs to $L^1(\R)$.

For any $t\in\RR$ we have
\begin{align*}
 \big\langle &\rho(\exp(tY))\xi,\,\xi\big\rangle=\int_Z\big\langle \pi_{\mu(z)}(\exp(tY))\xi_z,\,\xi_z\big\rangle_{\mathcal{H}_{\mu(z)}} d\kappa(z)\\
 &=\int_ZC_{\Re\varpi(z)} \int_{\R} e^{-\textrm{i}\lambda t}|\xi_z(\lambda)|^2 \vert \lambda\vert^{-\Re\varpi(z)} d\lambda d\kappa(z)\\
 &=\int_Z \int_{\R} e^{-\textrm{i}\lambda t} \theta_z(\lambda) d\lambda d\kappa(z)\overset{\text{(1)}}{=}  \int_{\R} \int_Ze^{-\textrm{i}\lambda t} \theta_z(\lambda)  d\kappa(z)d\lambda\\
&= \int_{\R} e^{-\textrm{i}\lambda t} \tau(\lambda) d\lambda=\hat{\tau}(t).
\end{align*}
Here in $(1)$ we use Fubini's theorem which is justified by \eqref{for:236}. We note that  $\widehat{\tau}$ is the Fourier transform of $\tau$.

 By Cauchy-Schwarz inequality, for any $t\in\RR$ we have
\begin{align}\label{for:153}
|\hat{\tau}(t)|&\leq \|\rho(\exp(tY))\xi\|\cdot \norm{\xi}=\norm{\xi}^2.
\end{align}
On the other hand, by the strong spectral gap property of $\rho$ and
Corollary~\ref{cor:5}, for any $t\in\RR$ we have
\begin{align}\label{for:154}
\Big|\big\langle \rho(\exp(tY))\xi,\,\xi\big\rangle\Big|=|\hat{\tau}(t)|\leq C_{\gamma,\epsilon}(|t|+1)^{-2(\gamma-\epsilon)} \big\|\xi\big\|^2_{\zeta}. \end{align}
We choose $r',r''$ depending only on $r$ and $\gamma$ so that:
\begin{itemize}
  \item[(a)] $r<r'<r''<\gamma$;
 \item[(b)] $2r'$ is not an integer.
\end{itemize}
This is possible because the admissible ranges for $r$ in Lemma \ref{for:72} are open intervals and the set of forbidden exponents $\{k/2: k\in\mathbb{Z}\}$ is discrete.

We have
\begin{align*}
&\int_{\R}\int_Z\theta_z(\lambda)|\lambda|^{-2r'} d\kappa(z)d\lambda=\int_{\R}\tau(\lambda)|\lambda|^{-2r'} d\lambda\\
&\overset{\text{(1)}}{=}\int_{\R}\hat{\tau}(t)2\Gamma(1-2r') \sin( \pi r' )|t|^{2r'-1}dt \\
  &= \int_{|t|\leq1}\hat{\tau}(t)2\Gamma(1-2r') \sin( \pi r' )|t|^{2r'-1}dt\\
  &+\int_{|t|>1}\hat{\tau}(t)2\Gamma(1-2r') \sin( \pi r' )|t|^{2r'-1}dt.
\end{align*}
Here in $(1)$ we use the Fourier transform of the homogeneous distribution $|\lambda|^{-2r'}$, which is valid when $2r'$ is not an integer.

By using \eqref{for:153} we have
\begin{align*}
&\int_{|t|\leq1}\hat{\tau}(t)2\Gamma(1-2r') \sin( \pi r' )|t|^{2r'-1}dt\\
&\leq  \norm{\xi}^2 2\Gamma(1-2r')\int_{|t|\leq1}|t|^{2r'-1}dt\leq C\Gamma(1-2r')(r')^{-1}\big\|\xi\big\|^2.
\end{align*}
By \eqref{for:154},
\[
  |\hat{\tau}(t)| \le C_{\gamma,\epsilon}(|t|+1)^{-2(\gamma-\epsilon)}\|\xi\|_\zeta^2.
\]
Hence
\begin{align*}
 &\int_{|t|>1}\hat{\tau}(t)2\Gamma(1-2r') \sin( \pi r' )|t|^{2r'-1}dt\\
&\leq C_{\gamma,r}\big\|\xi\big\|_{\zeta}^2\,\Gamma(1-2r')\int_{|t|>1}(|t|+1)^{-2r''}   |t|^{2r'-1}dt\overset{\text{(1)}}{\leq} C_{r,\gamma,1}\big\|\xi\big\|_{\zeta}^2.
\end{align*}
Here in $(1)$ we use $r''>r'$, which  implies $2r''-2r'>0$ and the last integral is finite.

Since we choose $r'$ and $r''$ in a way that is only dependent on $r$ and $\gamma$, combining the two estimates, we have
\begin{align*}
 \int_{\R}\int_Z\theta_z(\lambda)|\lambda|^{-2r'} d\kappa(z)d\lambda\leq C_{r,\gamma}(\big\|\xi\big\|_{\zeta}^2+\big\|\xi\big\|^2)\leq C_{r,\gamma,1}\big\|\xi\big\|_{\zeta}^2.
\end{align*}
By Fubini's theorem, we have
\begin{align}\label{for:237}
 \int_Z \int_{\R}\theta_z(\lambda)|\lambda|^{-2r'} d\lambda d\kappa(z)=\int_{\R}\int_Z\theta_z(\lambda)|\lambda|^{-2r'} d\kappa(z)d\lambda\leq C_{r,\gamma,1}\big\|\xi\big\|_{\zeta}^2.
\end{align}
Let
\begin{align*}
 \omega_1=\int_Z \frac{\xi_z(\lambda)}{|\lambda|^{r'}} d\kappa(z).
\end{align*}
Then we have
\begin{align*}
    |Y|^{r'}\omega_1=\int_Z |\lambda|^{r'}\frac{\xi_z(\lambda)}{|\lambda|^{r'}} d\kappa(z)=\xi.
  \end{align*}
Next, we estimate $\norm{\omega_1}$. Similar to \eqref{for:239}, we have
\begin{align*}
  \norm{\omega_1}^2&=  \int_Z \left\|\frac{\xi_z(\lambda)}{|\lambda|^{r'}}\right\|^2 d\kappa(z)=\int_Z\int_{\R}\theta_z(\lambda)|\lambda|^{-2r'}d\lambda d\kappa(z)\overset{\text{(1)}}{\leq} C_{r,\gamma}\big\|\xi\big\|_{\zeta}^2.
\end{align*}
Here in $(1)$ we use \eqref{for:237}.

We now solve $|Y|^{r}\omega=\xi$. Define
\begin{align*}
 \omega=\int_Z \frac{\xi_z(\lambda)}{|\lambda|^{r}} d\kappa(z).
\end{align*}
From \eqref{for:239} we have
\begin{align*}
 \norm{\omega}^2&=\int_Z C_{\Re\varpi(z)} \int_{\R} |\xi_z(\lambda)|^2|\lambda|^{-2r} \vert \lambda\vert^{-\Re\varpi(z)} d\lambda d\kappa(z)\\
&\leq \int_Z C_{\Re\varpi(z)} \int_{|\lambda|\geq1} |\xi_z(\lambda)|^2 \vert \lambda\vert^{-\Re\varpi(z)} d\lambda d\kappa(z)\\
&+\int_Z C_{\Re\varpi(z)}
\int_{|\lambda|<1} |\xi_z(\lambda)|^2|\lambda|^{-2r'} \vert \lambda\vert^{-\Re\varpi(z)} d\lambda d\kappa(z)\leq \norm{\xi}^2+\norm{\omega_1}^2.
\end{align*}
Hence
\[
 \|\omega\|\leq C_{\gamma,r}\,\|\xi\|_{\zeta}
\]
and solves $|Y|^{r}\omega=\xi$. This completes the proof of $k=\RR$.

  For $k=\CC$, let $S_2$ be the two-parameter group generated by $U$ and $\textrm{i}U$. By the spectral decomposition for $S_2$, we have
\[
  |U|^{r_1} |\textrm{i} U|^{r_2}= \int_{\widehat{\RR^2}} |\chi_1|^{r_1}|\chi_2|^{r_2}\,d\sigma(\chi_1,\chi_2).
\]
The previous discussion applied to $U$ and to $\mathrm{i}U$ separately shows
that there exist $\omega_1,\omega_2\in H_\rho$ such that
\[
    |U|^{r}\omega_1=\xi,\qquad |\mathrm{i} U|^{r}\omega_2=\xi,
\]
 with
\[
  \max\{\|\omega_1\|,\, \|\omega_2\|\}\leq C_{\gamma,r}\,\|\xi\|_{\zeta}.
\]
 By the Howe-Moore theorem (see \cite{howe-moore}), there are no nonzero
$U$- or $\mathrm{i}U$-invariant vectors, so $\sigma$ is not supported on
$\{\chi_1=0\}$ or $\{\chi_2=0\}$. Thus
  \begin{align*}
 \omega_1=\int_{\widehat{\RR}} |\chi_1|^{-r}\xi_{\chi_1,\chi_2}\,d\sigma(\chi_1,\,\chi_2)\quad\text{and}\quad \omega_2=\int_{\widehat{\RR}} |\chi_2|^{-r}\xi_{\chi_1,\chi_2}\,d\sigma(\chi_1,\,\chi_2).
\end{align*}
Let
\begin{align*}
 \omega=\int_{\widehat{\RR}} |\chi_1|^{-a}|\chi_2|^{-b}\xi_{\chi_1,\chi_2}\,d\sigma(\chi_1,\,\chi_2).
\end{align*}
We note that for all $(\chi_1,\chi_2)$ with $\chi_1\chi_2\neq 0$
\[
\frac{1}{|\chi_1|^{2a} |\chi_2|^{2b}}
\leq\frac{1}{|\chi_1|^{2r}} + \frac{1}{|\chi_2|^{2r}}.
\]
It follows that
\begin{align*}
   \int_{\widehat{\RR}} \frac{\norm{\xi_{\chi_1,\chi_2}}^2}{|\chi_1|^{2a}|\chi_2|^{2b}}\,d\sigma(\chi_1,\,\chi_2)
   &\leq \int_{\widehat{\RR}} \frac{\norm{\xi_{\chi_1,\chi_2}}^2}{|\chi_1|^{2r}}\,d\sigma(\chi_1,\,\chi_2)+\int_{\widehat{\RR}} \frac{\norm{\xi_{\chi_1,\chi_2}}^2}{|\chi_2|^{2r}}\,d\sigma(\chi_1,\,\chi_2)\\
   &\leq \norm{\omega_1}^2+\norm{\omega_2}^2.
  \end{align*}
Hence
\[
 \|\omega\|\leq C_{\gamma,r}\,\|\xi\|_{\zeta}
\]
and solves $|U|^{a}|\textrm{i} U|^{b}\omega=\xi$. This completes the proof of $k=\CC$.

\end{proof}

\begin{remark}
As explained in Remark~\ref{re:14}, for some representations of
$SL(2,\CC)$, the optimal decay exponent for matrix coefficients is
\emph{twice} the exponent obtained after restricting to an
$SL(2,\RR)$-subgroup. In other words, working purely at the
$SL(2,\RR)$ level may lose a factor \(1/2\) in the decay rate.

By contrast, in Lemma~\ref{le:23} we use only the Fourier model for
$SL(2,\RR)$, but the argument does \emph{not} lose any part of the
spectral-gap exponent~\(\gamma\): the admissible range \(0<r<\gamma\)
for solvability of the fractional equation is exactly the range dictated
by the strong spectral gap of the original representation of \(P\).
\end{remark}

\subsubsection{Partial Sobolev estimates}\label{sec:53} To prove Theorem~\ref{th:6} in the form needed later for mixing, one must go
beyond mere solvability and obtain estimates in \emph{partial} Sobolev norms.
This is the main additional difficulty. The idea is to split the datum into a low-frequency part in the \(U\)-direction
(resp.\ in the \(U,\mathrm{i}U\)-directions), for which the missing regularity
can be recovered from derivatives in the complementary directions, and a
high-frequency part, for which the fractional equation is solvable with a
uniform \(L^2\) bound. This yields the partial-Sobolev
estimates required for the product argument.

More precisely, we introduce an operator
$\mathcal{P}$ (in the real case) or $\mathcal{P}_1,\mathcal{P}_2$ (in the
complex case), constructed from the spectral measure of $|U|$ and
$|\mathrm{i}U|$ (see Lemma~\ref{le:16}). We decompose
\[
  \xi = (I-\mathcal{P})\xi + \mathcal{P}\xi
  \quad\text{or}\quad
  \xi = (I-\mathcal{P}_1)\xi + \mathcal{P}_2\mathcal{P}_1\xi
          + (I-\mathcal{P}_2)\mathcal{P}_1\xi.
\]
The crucial point is that \(\mathcal P\xi\) gains regularity in the \(U\)-direction
through the spectral cutoff, so its full Sobolev norm is controlled only by
Sobolev norms in the complementary directions \(X,V\). Likewise, in the
complex case, the full Sobolev norm of \(\mathcal P_2\mathcal P_1\xi\) is
controlled only by Sobolev norms in the directions
\(X,\mathrm{i}X,V,\mathrm{i}V\), without any derivatives in the
\(U,\mathrm{i}U\)-directions.

We use Lemma~\ref{le:23} to solve the equation for $\mathcal{P}\xi$ (resp. for
$\mathcal{P}_2\mathcal{P}_1\xi$) with these partial norms, while the equations
for $(I-\mathcal{P})\xi$ (resp. for $(I-\mathcal{P}_1)\xi$ and
$(I-\mathcal{P}_2)\mathcal{P}_1\xi$) always have solutions with an $L^2$ bound.
This yields the desired partial Sobolev estimates in Theorem~\ref{th:6}.

\subsubsection{From Lemma~\ref{le:16} to Proposition~\ref{po:4}} To finish the proof of Proposition \ref{po:4}, we need the following result:
\begin{lemma}\label{le:16} Let $P$ be a Lie group with Lie algebra $\mathfrak{sl}(2,k)$, $k=\RR$ or $\CC$ and $(\rho,H_\rho)$ be a unitary representation of $P$.
\begin{enumerate}
  \item If $k=\RR$,  there exists a linear operator $\mathcal{P}:H_\rho\to H_\rho$
  such that for any $\psi\in H_\rho$:
\begin{enumerate}
  \item\label{for:24} $\norm{\mathcal{P}\psi}\leq C_{\mathcal{P}}\norm{\psi}$;
  \item\label{for:76} for any $n\in\NN$,  if $\Sigma^n\psi\in H_\rho$ (see \eqref{for:252}), then $\mathcal{P}\psi\in W^{2n}(H_\rho)$ with the estimates
  \begin{align*}
 \norm{\mathcal{P}\psi}_{2n}\leq C_{\mathcal{P},n}\big\|\Sigma^n\psi\big\|;
\end{align*}
 \item\label{for:77} for any $q>0$ the fractional equation $|U|^q\omega_q=\psi-\mathcal{P}\psi$ has a solution $\omega_q\in H_\rho$ with the estimate
 \begin{align*}
 \|\omega_q\|\leq C_{\mathcal{P},q}\norm{\psi}.
 \end{align*}
\end{enumerate}
The constants $C_{\mathcal{P}}$, $C_{\mathcal{P},n}$ and $C_{\mathcal{P},q}$ are independent of the specific representation $\rho$.

\smallskip
  \item If $k=\CC$,  there exist linear operators
  $\mathcal{P}_i:H_\rho\to H_\rho$, $i=1,2$, such that for any $\psi\in H_\rho$:
\begin{enumerate}
  \item\label{for:255} $\mathcal{P}_1\mathcal{P}_2=\mathcal{P}_2\mathcal{P}_1$ and $\norm{\mathcal{P}_i\psi}\leq C_{\mathcal{P}_i}\norm{\psi}$, $i=1,2$;
  \item for any $n\in\NN$,  if $\Sigma^n\psi\in H_\rho$ (see \eqref{for:252}), then $\mathcal{P}_1\mathcal{P}_2\psi\in W^{2n}(H_\rho)$ with the estimate
  \begin{align*}
 \norm{\mathcal{P}_1\mathcal{P}_2\psi}_{2n}\leq C_{\mathcal{P}_1,\mathcal{P}_2,n}\big\|\Sigma^n\psi\big\|;
\end{align*}
 \item\label{for:249} for any $q>0$ both the fractional equations
 \begin{align*}
 |U|^q\omega_{q,1}=\psi-\mathcal{P}_1\psi\quad\text{and}\quad |\emph{i}U|^q\omega_{q,2}=\psi-\mathcal{P}_2\psi
 \end{align*}
 have solutions $\omega_{q,1},\,\omega_{q,2}\in H_\rho$ with the estimates
 \begin{align*}
 \|\omega_{q,1}\|\leq C_{\mathcal{P}_1,q}\norm{\psi}\quad\text{and}\quad \|\omega_{q,2}\|\leq C_{\mathcal{P}_2,q}\norm{\psi}.
 \end{align*}
\end{enumerate}
The constants $C_{\mathcal{P}_i}$, $C_{\mathcal{P}_1,\mathcal{P}_2,n}$,
  $C_{\mathcal{P}_1,q}$ and $C_{\mathcal{P}_2,q}$ are independent of the specific representation $\rho$.
\end{enumerate}
\end{lemma}
\begin{remark}
In Lemma~\ref{le:16}, we do not require that $\rho$ have a strong spectral gap.
\end{remark}
We first show how Lemma~\ref{le:16}, together with Lemma~\ref{le:23},
implies Proposition~\ref{po:4}.
The proof of Lemma \ref{le:16} is given at the end of this section.

If $r=0$, the statement is trivial (take $\omega=\xi$ in case $k=\RR$, and
$\omega_1=\xi$, $\omega_2=\omega_3=0$ in case $k=\CC$). Thus we may assume
$0<r<\gamma$.

Suppose $k=\RR$.  Fix $\mathcal{P}$ as in Lemma \ref{le:16}. By \eqref{for:76} and the interpolation theorem, we see that: for any $t\geq0$ and any  $\psi\in H_\rho$ with $\Sigma^{\frac{t}{2}}\psi\in H_\rho$,  we have
\begin{align}\label{for:79}
 \norm{\mathcal{P}\psi}_{t}&\leq C_{t}\big\|\Sigma^{\frac{t}{2}}\psi\big\|.
\end{align}
It follows from \eqref{for:248} of Lemma \ref{le:23} that  the fractional equation
\begin{align*}
    |U|^r\omega_1=\mathcal{P}\xi
  \end{align*}
  has a solution $\omega_1\in H_\rho$ with the estimate
  \begin{align*}
 \norm{\omega_1}\leq C_{r,\gamma}\big\|\mathcal{P}\xi\big\|_{\zeta}\overset{\text{(1)}}{\leq} C_{r,\gamma,1}\big\|\Sigma^{\frac{\zeta}{2}}\xi\big\|.
\end{align*}
Here in $(1)$ we use \eqref{for:79} with $t=\zeta$.

By using \eqref{for:77}, the fractional equation
\begin{align*}
    |U|^r\omega_2=\xi-\mathcal{P}\xi
  \end{align*}
  has a solution $\omega_2\in H_\rho$ with the estimate
  \begin{align*}
 \norm{\omega_2}\leq C_{r}\norm{\xi-\mathcal{P}\xi}\overset{\text{(1)}}{\leq}C_{r,1}\norm{\xi}.
\end{align*}
Here in $(1)$ we use \eqref{for:24}.

Let $\omega=\omega_1+\omega_2$. Then $\omega$ solves the equation $|U|^r\omega=\xi$ with the estimate
 \begin{align*}
 \norm{\omega}\leq C_{r,\gamma}\big\|\Sigma^{\frac{\zeta}{2}}\xi\big\|+C_{r}\norm{\xi}\overset{\text{(1)}}{\leq} C_{r,\gamma,1}\big\|\Sigma^{\frac{\zeta}{2}}\xi\big\|.
\end{align*}
Here in $(1)$ we use the fact that $\Sigma^{1/2}$ is an elliptic operator on the subgroup generated by $X$ and $V$. Hence its positive powers control all
$L^2$–Sobolev norms along $X$ and $V$, and in particular
\begin{align*}
\norm{\xi}\leq C\big\|\Sigma^{\frac{\zeta}{2}}\xi\big\|.
\end{align*}
This completes the proof in the case $k=\RR$.

Suppose $k=\CC$.  Fix $\mathcal{P}_i$, $i=1,2$ as in Lemma \ref{le:16}. Similar to \eqref{for:79}, for any $t\geq0$ and any  $\psi\in H_\rho$ with $\Sigma^{\frac{t}{2}}\psi\in H_\rho$,  we have
\begin{align}\label{for:254}
 \norm{\mathcal{P}_1\mathcal{P}_2\psi}_{t}&\leq C_{t}\big\|\Sigma^{\frac{t}{2}}\psi\big\|.
\end{align}
From \eqref{for:249}, the fractional equation
\begin{align}\label{for:271}
    |U|^r\omega_{1}=\xi_1:=\xi-\mathcal{P}_1\xi
  \end{align}
  has a solution $\omega_1\in H_\rho$ with the estimate
  \begin{align*}
 \norm{\omega_{1}}\leq C_{r}\norm{\xi}\leq C_{r}\big\|\Sigma^{\frac{\zeta}{2}}\xi\big\|,
\end{align*}
using again ellipticity of $\Sigma^{1/2}$ as above.

It follows from \eqref{for:250} of Lemma \ref{le:23} that  the fractional equation
\begin{align}\label{for:272}
    |U|^{a}|\textrm{i} U|^{b}\omega_{2}=\mathcal{P}_2\mathcal{P}_1\xi
  \end{align}
  has a solution $\omega_{2}\in H_\rho$ with the estimate
  \begin{align*}
 \norm{\omega_{2}}\leq C_{r,\gamma}\big\|\mathcal{P}_2\mathcal{P}_1\xi\big\|_{\zeta}\overset{\text{(1)}}{=} C_{r,\gamma}\big\|\mathcal{P}_1\mathcal{P}_2\xi\big\|_{\zeta}\overset{\text{(2)}}{\leq} C_{r,\gamma,1}\big\|\Sigma^{\frac{\zeta}{2}}\xi\big\|.
\end{align*}
Here in $(1)$ we use \eqref{for:255}; in $(2)$ we use \eqref{for:254} with $t=\zeta$.

From \eqref{for:249} again, the fractional equation
\begin{align}\label{for:274}
    |\textrm{i}U|^r\omega_3=\mathcal{P}_1\xi-\mathcal{P}_2\mathcal{P}_1\xi
  \end{align}
  has a solution $\omega_3\in H_\rho$ with the estimate
  \begin{align*}
 \norm{\omega_3}\leq C_{r}\norm{\mathcal{P}_1\xi}\overset{\text{(1)}}{\leq}C_{r,1}\norm{\xi}\leq C_{r,1}\big\|\Sigma^{\frac{\zeta}{2}}\xi\big\|.
\end{align*}
Here in $(1)$ we use \eqref{for:255}.

Finally, we note that
\begin{align}\label{for:265}
\xi=( \xi-\mathcal{P}_1\xi)+(\mathcal{P}_2\mathcal{P}_1\xi)+(\mathcal{P}_1\xi-\mathcal{P}_2\mathcal{P}_1\xi),
\end{align}
so the triple $(\omega_1,\omega_2,\omega_3)$ provides a solution to the
fractional system
\[
  |U|^r\omega_1 + |U|^{a}|\mathrm{i} U|^{b}\omega_{2}
  + |\mathrm{i}U|^r\omega_3 = \xi
\]
with the required estimates. This completes the proof in the case $k=\CC$.
\subsubsection{Proof of Lemma \ref{le:16}}\label{sec:41}

The operator $\mathcal{P}$ is defined via a projection-valued measure. This operator was originally employed by R.~Howe to study the decay of matrix coefficients and was later used by the author to construct smoothing operators for applying the KAM method \cite{W5}. In the present paper, the operator is applied in a different
way: it plays a key role in our analysis of the fractional cohomological equation, yielding partial smooth estimates for the solution.

Suppose $Y\in \Lie(P)$. Let $S_1$ be the one-parameter subgroup generated by $Y$ and assume $S_1$ is isomorphic to $\RR$. By the spectral decomposition for $S_1$
(see Section~\ref{sec:13}), we have
\[
  |Y|^r = \int_{\widehat{\RR}} |\chi|^{r}\,d\sigma(\chi)\qquad \forall\,r>0,
\]
where $\sigma$ is the spectral measure of $Y$. For any $f\in L^\infty(\RR, d\sigma)$ taking values in $\RR$, we define an operator $\mathcal{P}_{Y,f}$ on $H_\rho$ by
\[
  \mathcal{P}_{Y,f}:=\int_{\widehat{\RR}} f(\chi)\, d\sigma(\chi).
\]
\begin{remark} In applications, $f$ is typically chosen to be a compactly supported
$C^\infty$ function. The operator $\mathcal{P}_{Y,f}$ then acts as a smoothing
operator along the $Y$-direction: it truncates the spectrum of $Y$ and
regularizes vectors in that direction.
\end{remark}
The following properties are standard (see Sections 8.4 and 8.5 of \cite{W5}):
\begin{enumerate}
 \item\label{for:215} $\|\mathcal{P}_{Y,f}\|\leq \|f\|_\infty$;
 \item\label{for:93} for any $\xi,\,\eta\in H_\rho$ we have
 \[
   \langle \mathcal{P}_{Y,f}\xi,\eta\rangle=\langle \xi,\mathcal{P}_{Y,f}\eta\rangle.
 \]
\end{enumerate}
The next result shows that: if a vector $\xi$ has Sobolev
regularity only in the directions of a subgroup $Q$ (complementary to $S$),
then by applying smoothing operators along $u_1$ and $u_2$ (spanning $S$) we
recover full Sobolev regularity in all directions of $P$.
\begin{lemma}\label{le:14} (Lemma 8.6 of \cite{W5})
Suppose $Q$ and $S$ are subgroups of $P$ such that
\[
  \Lie(P)=\Lie(S)\oplus\Lie(Q).
\]
Assume that $S$ is abelian and that $\Lie(S)$ is spanned by $\{u_1,u_2\}$.
Choose $f_1,f_2\in \mathcal{S}(\RR)$ taking real values. Then for any
$\xi\in W^{n,Q}(H_\rho)$, $n\in\NN$, the vector
\[
  \xi'=\mathcal{P}_{u_1,f_1}\,\mathcal{P}_{u_2,f_2}\,\xi
\]
belongs to $W^{n}(H_\rho)$ and satisfies the estimates
\[
  \|\xi'\|_{l}\leq C_{f_1,f_2,l}\,\|\xi\|_{Q,l},\qquad \forall\, 0\leq l\leq n.
\]
\end{lemma}
\begin{remark} The constant $C_{f_1,f_2,l}$ depends only on $f_1$, $f_2$ and $l$, and can be chosen
uniformly for all unitary representations $\rho$ of $P$; in particular, it does not
depend on the specific representation $(\rho,H_\rho)$.
\end{remark}
We now prove Lemma \ref{le:16}. Fix \( f\in \mathcal{S}(\mathbb{R}) \) satisfying  $0\leq f\leq 1$ and
\begin{align}\label{for:269}
 f(t) =
\begin{cases}
1, & \text{if } |t| \leq 1, \\
0, & \text{if } |t| \geq 2.
\end{cases}
\end{align}
Suppose $k=\RR$.  Let $\mathcal{P}:=\mathcal{P}_{U,f}$. Then property \eqref{for:24} follows immediately from \eqref{for:215}.

To see \eqref{for:77}, note that
\[
  I-\mathcal{P}=\int_{\widehat{\RR}} (1-f)(\chi)\, d\sigma(\chi).
\]
By the choice of $f$, we have
\[
  \big\||\chi|^{-q}(1-f)(\chi)\big\|_\infty \le 1+\|f\|_\infty
\]
for every $q>0$. Therefore, for any $\psi\in H_\rho$ the fractional equation
\[
  |U|^q\omega_q=\psi-\mathcal{P}\psi
\]
has the unique solution
\[
  \omega_q=\int_{\widehat{\RR}} |\chi|^{-q}(1-f)(\chi)\,\psi_\chi\, d\sigma(\chi)\in H_\rho
\]
and we obtain the estimate
\[
  \|\omega_q\|\le \big\||\chi|^{-q}(1-f)(\chi)\big\|_\infty\|\psi\|
  \le (1+\|f\|_\infty)\|\psi\|,
\]
which is \eqref{for:77} (with a constant depending only on $f$, and in particular independent of the specific representation $\rho$).

For \eqref{for:76}, let $Q$ be the subgroup with Lie algebra spanned by $X$ and $V$, and let $u_1=u_2=U$. Then the hypotheses of Lemma~\ref{le:14} apply, and we obtain, for every $n\in\NN$,
\[
  \|\mathcal{P}\psi\|_{2n}\le C_{\mathcal P,n}\,\|\psi\|_{Q,2n}
  \;\le\; C_{\mathcal P,n,1}\,\|\Sigma^n\psi\|,
\]
since $\Sigma^n$ is elliptic in the directions of $X$ and $V$ and controls the $Q$-Sobolev norms. This yields \eqref{for:76}. Thus Lemma~\ref{le:16} is proved in the case $k=\RR$.

Suppose $k=\CC$.  Let $\mathcal{P}_1=\mathcal{P}_{U,f}$ and $\mathcal{P}_2=\mathcal{P}_{\mathrm{i}U,f}$. Since $U$ and $\mathrm{i}U$ commute, the corresponding spectral projections commute as well, so
\[
  \mathcal{P}_1\mathcal{P}_2 = \mathcal{P}_2\mathcal{P}_1.
\]
Property \eqref{for:255} then follows from \eqref{for:215}.

 For \eqref{for:249}, the same argument as for \eqref{for:77} applies.

Finally, for \eqref{for:76} in the complex case, we apply Lemma~\ref{le:14} with $S$ the abelian subgroup generated by $U$ and $\mathrm{i}U$, and $Q$ the subgroup generated by $X$, $V$, $\mathrm{i}X$, and $\mathrm{i}V$. Then $\Sigma^n$ is elliptic in the $Q$-directions, and Lemma~\ref{le:14} gives
\[
  \|\mathcal{P}_1\mathcal{P}_2\psi\|_{2n}
  \;\le\; C_{\mathcal{P}_1,\mathcal{P}_2,n}\,\|\psi\|_{Q,2n}
  \;\le\; C_{\mathcal{P}_1,\mathcal{P}_2,n,1}\,\|\Sigma^n\psi\|.
\]
This establishes the claimed smoothing property and completes the proof of Lemma~\ref{le:16}. Hence, we complete the proof of Proposition~\ref{po:4}.

\subsection{Proof of Theorem \ref{th:6}: solvability and estimates}

We now combine the rank-one inputs obtained above. The proof is by induction
on the number of commuting factors. The commutation relations between the
operators associated with different factors allow one to apply the rank-one
solvability statements successively, while preserving the required Sobolev
control. In the partial-Sobolev case this produces a decomposition of the data
into $2^n$ terms; in the full-Sobolev case no such decomposition is needed.

Before we proceed to the proof of Theorem \ref{th:6}, we list two lemmas that will be used in the proof.
\begin{lemma}\label{ob:1} Suppose $\mathcal{A}$ is an essentially self-adjoint operator on a Hilbert space $\mathcal{L}$ and satisfies
\begin{align}\label{for:264}
 \langle \mathcal{A}\vartheta, \vartheta\rangle \geq\norm{\vartheta}^2\qquad \forall\, \vartheta\in \text{Dom}(\mathcal{A}).
\end{align} Then
\begin{enumerate}
  \item\label{for:116} if $\mathcal{A}^r\vartheta\in \mathcal{L}$ for some $r>0$ and $\vartheta\in \mathcal{L}$, then $\mathcal{A}^a\vartheta\in \mathcal{L}$ for any $0\leq a\leq r$;

  \smallskip
  \item\label{for:117} for any $r\leq 0$,  $\mathcal{A}^r$ is a bounded linear map on $\mathcal{L}$ with $\norm{\mathcal{A}^r}\leq 1$.
\end{enumerate}
\end{lemma}
\begin{proof} We recall spectral theory in Section \ref{sec:19}. The assumption implies that the spectrum of \(\mathcal{A}\) is contained in \([1,\infty)\). Then for any measurable function \(f\) defined on \([1,\infty)\), one can define the operator \(f(\mathcal{A})\).

\eqref{for:116}:  We note that for any \(0\leq a\leq r\), the function \(x\mapsto x^a\) is dominated by \(x\mapsto x^r\) on \([1,\infty)\).
Then  it follows from the spectral calculus that \(\mathcal{A}^a\vartheta\) belongs to \(\mathcal{L}\).

\eqref{for:117}: If \(r\leq 0\), then for every \(x\geq 1\) we have \(x^r\leq 1\). Consequently, the operator \(\mathcal{A}^r\) is bounded with \(\|\mathcal{A}^r\|\leq 1\).

\end{proof}

\begin{lemma}\label{ob:2}
For any $1\leq i\neq j\leq n$, any $c_i,c_j\in\RR$, and any $r\ge0$, we have
\[
 \mathcal{R}_{i}^{c_i} \mathcal{R}_{j}^{c_j}=\mathcal{R}_{j}^{c_j} \mathcal{R}_{i}^{c_i}
  \quad\text{and}\quad
  \mathcal{R}_{i}^{c_i}|Y_j|^r=|Y_j|^r\mathcal{R}_{i}^{c_i},
\]
where $\mathcal{R}_i$ denotes either $\Sigma_i$ or $\mathcal{D}_i$, and $Y_j$ denotes either $U_j$ or $\mathrm{i}U_j$ (according to whether
$k_j=\RR$ or $k_j=\CC$).
\end{lemma}

\begin{proof}
We recall the direct integral decomposition in Section \ref{sec:20}. Because $S_i$ and $S_j$ commute, we see that
$(\beta|_{S_iS_j},\,\mathcal{L})$ decomposes into a direct integral of irreducible representations of the form: $(\rho_1 \otimes \rho_2, \,\mathcal{H}_1 \otimes \mathcal{H}_2)$, where $(\rho_1, \mathcal{H}_1)$ is an irreducible representation of $S_i$ and $(\rho_2, \mathcal{H}_2)$ is an irreducible representation of $S_j$.  On $\mathcal{H}_1\otimes\mathcal{H}_2$, \(\mathcal{R}_i\) acts as an operator on \(\mathcal{H}_1\) and \(\mathcal{R}_j\) on \(\mathcal{H}_2\), hence
\[
  \mathcal{R}_{i}^{c_i} \mathcal{R}_{j}^{c_j}=\mathcal{R}_{j}^{c_j} \mathcal{R}_{i}^{c_i}\quad\text{on }\mathcal{H}_1\otimes\mathcal{H}_2.
\]
Since $\mathcal{L}$ decomposes into a direct integral of tensor products $\mathcal{H}_1\otimes\mathcal{H}_2$, it follows that

\[
  \mathcal{R}_{i}^{r_i} \mathcal{R}_{j}^{r_j}=\mathcal{R}_{j}^{r_j} \mathcal{R}_{i}^{r_i}\quad\text{on }\mathcal{L}.
\]
Likewise \(|U_j|\) or $\textrm{i}U_j$ (affiliated to \(S_j\)) commutes with \(\mathcal{R}_i\). Thus, we have
\begin{align*}
  \mathcal{R}_{i}^{c_i}|Y_j|^r=|Y_j|^r\mathcal{R}_{i}^{c_i},\qquad \forall\,c_i,\, r\in\RR.
\end{align*}
Thus, we complete the proof.
\end{proof}
Fix $1\leq i\leq n$. Suppose $k_i=\CC$.  Let $E_{i,1}$ (resp. $E_{i,2}$) be the spectral measure of the positive self-adjoint operator $|U_i|$ (resp. $|\textrm{i}U_i|$),
  so that
  \[
    |U_i| = \int_{[0,\infty)} \lambda\,dE_{i,1}(\lambda)\quad\text{and}\quad|\textrm{i}U_i| = \int_{[0,\infty)} \lambda\,dE_{i,2}(\lambda).
  \]
Recall $\mathcal{P}_{U,f}$ and $\mathcal{P}_{\mathrm{i}U_i,f}$ defined in Section \ref{sec:41}, where $f$ is as in \eqref{for:269}. Then
  \[
    \mathcal{P}_1:=\mathcal{P}_{U_i,f} = \int_{[0,\infty)} f(\lambda)\,dE_{i,1}(\lambda)\quad\text{and}\quad \mathcal{P}_2:=\mathcal{P}_{\mathrm{i}U_i,f} = \int_{[0,\infty)} f(\lambda)\,dE_{i,2}(\lambda).
  \]
\begin{lemma}\label{le:20} Let $r_1,\dots,r_n\ge0$ and define
\begin{align*}
  \mathcal A_i \;=\; \prod_{j\ne i}|Y_j|^{r_j}
\end{align*}
where $Y_j$ denotes either $U_j$ or $\mathrm{i}U_j$ according to whether
$k_j=\mathbb R$ or $k_j=\mathbb C$.

If $\psi\in\mathcal L$ and $\mathcal A_i\psi\in\mathcal L$, then
$\mathcal A_i\mathcal{P}_1\psi\in\mathcal L$ and
$\mathcal A_i(\mathcal{P}_1-\mathcal{P}_2\mathcal{P}_1)\psi\in\mathcal L$, and
\[
  \big\|\mathcal A_i\mathcal{P}_1\psi\big\|
    \;\le\; \|\mathcal A_i\psi\|,
  \qquad
  \big\|\mathcal A_i(\mathcal{P}_1-\mathcal{P}_2\mathcal{P}_1)\psi\big\|
    \;\le\; 2\|\mathcal A_i\psi\|.
\]

\end{lemma}
\begin{proof} Since
\[
  |U_i|\mathcal A_i = \mathcal A_i |U_i|\quad\text{and}\quad  |\textrm{i}U_i|\mathcal A_i = \mathcal A_i |\textrm{i}U_i|,
\]
$\mathcal A_i$ commutes with every bounded Borel function of $|U_i|$ and $|\textrm{i}U_i|$. In particular $f(|U_i|)=\mathcal P_1$ and $f(|\textrm{i}U_i|)=\mathcal P_2$ commute with $\mathcal A_i$.
Thus
\[
  \mathcal A_i\mathcal{P}_1 = \mathcal{P}_1\mathcal A_i,
  \qquad
  \mathcal A_i(\mathcal{P}_1-\mathcal{P}_2\mathcal{P}_1)
   = (\mathcal{P}_1-\mathcal{P}_2\mathcal{P}_1)\mathcal A_i.
\]
By the construction in Section~\ref{sec:41} we chose $f$ with
$|f|\le1$, so $\|\mathcal{P}_j\|\le1$ for $j=1,2$ (see \eqref{for:215} of Section~\ref{sec:41}). Therefore,
for any $\psi$ with $\mathcal A_i\psi\in\mathcal L$,
\[
  \big\|\mathcal A_i\mathcal{P}_1\psi\big\|
  =\big\|\mathcal{P}_1\mathcal A_i\psi\big\|
  \le\|\mathcal{P}_1\|\,\|\mathcal A_i\psi\|
  \le \|\mathcal A_i\psi\|
\]
and
\[
  \big\|\mathcal A_i(\mathcal{P}_1-\mathcal{P}_2\mathcal{P}_1)\psi\big\|
  =\big\|(\mathcal{P}_1-\mathcal{P}_2\mathcal{P}_1)\mathcal A_i\psi\big\|
  \le \big(\|\mathcal{P}_1\|+\|\mathcal{P}_2\mathcal{P}_1\|\big)\,\|\mathcal A_i\psi\|
  \le 2\,\|\mathcal A_i\psi\|.
\]
This proves the claim.
\end{proof}
Fix once and for all either $\mathcal R_j=\Sigma_j$ for all $j$, or
$\mathcal R_j=\mathcal D_j$ for all $j$.  We claim that, for each $1\le i\le n$ and every
\[
  \xi\in\mathcal L\quad\text{with}\quad
  \mathcal{R}_i^{\frac{\zeta_i}{2}}\cdots\mathcal{R}_1^{\frac{\zeta_1}{2}}\xi\in\mathcal L,
\]
the following holds:
\begin{enumerate}
  \item $\mathcal{R}=\Sigma$: there exist vectors $\omega_{\lambda}, \,\xi_\lambda\in\mathcal L$, $1\le\lambda\le 2^i$, such that
\[
   |\Lambda_{1,\lambda}|^{r_1}\cdots|\Lambda_{i,\lambda}|^{r_i}\,\omega_{\lambda} = \xi_\lambda,\quad \sum_{\lambda}\xi_\lambda=\xi
\]
  \item $\mathcal{R}=\mathcal{D}$: for any $1\le\lambda\le 2^i$, there exists $\omega_{\lambda}\in\mathcal L$, such that
\[
   |\Lambda_{1,\lambda}|^{r_1}\cdots|\Lambda_{i,\lambda}|^{r_i}\,\omega_{\lambda} = \xi.
\]
\end{enumerate}
In either case we have the uniform estimate
\[
  \|\omega_{\lambda}\|\le
  C_{\mathfrak r_i,\mathfrak p_i}\,
  \big\|\mathcal{R}_i^{\frac{\zeta_i}{2}}\cdots\mathcal{R}_1^{\frac{\zeta_1}{2}}\xi\big\|,
\]
where $\mathfrak{r}_i=(r_1,\cdots, r_i)$ and $\mathfrak{p}_i=(\gamma_1,\cdots, \gamma_i)$.

We prove this result by induction on $i$.  \emph{Base case $i=1$.} We consider the restricted representation $(\beta|_{S_1},\,\mathcal{L})$.
If $\mathcal R_1=\Sigma_1$, the statement follows from
Proposition~\ref{po:4}. If $\mathcal R_1=\mathcal D_1$, it follows from
Lemma~\ref{le:23}. In the complex case for $S_{1}$, we choose $a,b\ge0$ with
$a+b=r_{1}$ so that either $\Lambda_{1,\lambda}=U_{1}$ (take $a=r_{1},b=0$) or
$\Lambda_{1,\lambda}=\mathrm{i}U_{1}$ (take $a=0,b=r_{1}$) in Lemma~\ref{le:23}.

Suppose the result holds for some $i$. Now, suppose that $\xi\in \mathcal{L}$ satisfies
\begin{align*}
 \mathcal{R}_{i+1}^{\frac{\zeta_{i+1}}{2}}\mathcal{R}_i^{\frac{\zeta_i}{2}}\cdots\mathcal{R}_1^{\frac{\zeta_1}{2}}\xi\in \mathcal{L}.
\end{align*}
 Lemma \ref{ob:2} shows that
\begin{align}\label{for:162}
   &\mathcal{R}_i^{\frac{\zeta_i}{2}}\cdots\mathcal{R}_1^{\frac{\zeta_1}{2}}(\mathcal{R}_{i+1}^{\frac{\zeta_{i+1}}{2}}\xi)=\mathcal{R}_{i+1}^{\frac{\zeta_{i+1}}{2}}\mathcal{R}_i^{\frac{\zeta_i}{2}}\cdots
   \mathcal{R}_1^{\frac{\zeta_1}{2}}\xi\in \mathcal{L}.
\end{align}
Applying successively the bounded operators \(\mathcal{R}_j^{-\frac{\zeta_j}{2}}\) (\eqref{for:117} of Lemma \ref{ob:1}) and using Lemma \ref{ob:2}, we see that
\begin{equation}\label{for:159}
 \mathcal{R}_{i+1}^{\frac{\zeta_{i+1}}{2}}\xi\in\mathcal{L}.
\end{equation}
\emph{Step 1}: apply the induction hypothesis in the first $i$ factors.
\begin{enumerate}
  \item $\mathcal{R}=\Sigma$: From \eqref{for:162} and \eqref{for:159}, it follows from the inductive assumption that there is a decomposition
\begin{align*}
 \sum_{\lambda=1}^{2^i}\big(\mathcal{R}_{i+1}^{\frac{\zeta_{i+1}}{2}}\xi\big)_\lambda=\mathcal{R}_{i+1}^{\frac{\zeta_{i+1}}{2}}\xi\quad\text{with}\quad \big(\mathcal{R}_{i+1}^{\frac{\zeta_{i+1}}{2}}\xi\big)_\lambda\in\mathcal{L},
\end{align*}
 and each of the equations
\begin{align}\label{for:160}
   |\Lambda_{1,\lambda}|^{r_1}|\Lambda_{2,\lambda}|^{r_2}\cdots |\Lambda_{i,\lambda}|^{r_i}\psi_\lambda=\big(\mathcal{R}_{i+1}^{\frac{\zeta_{i+1}}{2}}\xi\big)_\lambda
\end{align}
  has solutions $\psi_\lambda\in \mathcal{L}$ with the estimate
  \begin{align}\label{for:161}
 \norm{\psi_\lambda}\leq C_{\mathfrak{r}_i,\mathfrak{p}_i} \big\|\mathcal{R}_i^{\frac{\zeta_i}{2}}\cdots\mathcal{R}_1^{\frac{\zeta_1}{2}}(\mathcal{R}_{i+1}^{\frac{\zeta_{i+1}}{2}}\xi)\big\|
 \overset{\text{(1)}}{=}
     C_{\mathfrak{r}_{i},\mathfrak{p}_{i}}\,
     \big\|\mathcal{R}_{i+1}^{\frac{\zeta_{i+1}}{2}}\mathcal{R}_i^{\frac{\zeta_i}{2}}\cdots\mathcal{R}_1^{\frac{\zeta_1}{2}}\xi\big\|.
\end{align}
Here in $(1)$ we use \eqref{for:162}.
  \item $\mathcal{R}=\mathcal{D}$: for any $1\le\lambda\le 2^i$, it follows from the inductive assumption that there exists
        $\psi_\lambda\in\mathcal L$ solving the equation
\begin{align}\label{for:279}
   |\Lambda_{1,\lambda}|^{r_1}|\Lambda_{2,\lambda}|^{r_2}\cdots |\Lambda_{i,\lambda}|^{r_i}\psi_\lambda=\mathcal{R}_{i+1}^{\frac{\zeta_{i+1}}{2}}\xi
\end{align}
 and satisfying the same estimate \eqref{for:161}.
\end{enumerate}
In both cases, it follows from \eqref{for:117} of Lemma \ref{ob:1} that $\mathcal{R}_{i+1}^{-\frac{\zeta_{i+1}}{2}}$ is a bounded linear map on $\mathcal{L}$ and
\begin{align}\label{for:149}
  \theta_\lambda=\mathcal{R}_{i+1}^{-\frac{\zeta_{i+1}}{2}}\psi_\lambda \in \mathcal{L}
\end{align}
\emph{Step 2}: incorporate the $(i+1)$-st factor.

\emph{Case I}:  $\mathcal{R}=\mathcal{D}$.  From \eqref{for:279} we have
\begin{align}\label{for:280}
  \mathcal{R}_{i+1}^{-\frac{\zeta_{i+1}}{2}}&\big( |\Lambda_{1,\lambda}|^{r_1}|\Lambda_{2,\lambda}|^{r_2}\cdots |\Lambda_{i,\lambda}|^{r_i}\psi_\lambda \big)=\mathcal{R}_{i+1}^{-\frac{\zeta_{i+1}}{2}}(\mathcal{R}_{i+1}^{\frac{\zeta_{i+1}}{2}}\xi)\notag\\
  \overset{\text{(1)}}{\Rightarrow} & |\Lambda_{1,\lambda}|^{r_1}|\Lambda_{2,\lambda}|^{r_2}\cdots |\Lambda_{i,\lambda}|^{r_i}(\mathcal{R}_{i+1}^{-\frac{\zeta_{i+1}}{2}}\psi_\lambda)=\xi\notag\\
  \Rightarrow & |\Lambda_{1,\lambda}|^{r_1}|\Lambda_{2,\lambda}|^{r_2}\cdots |\Lambda_{i,\lambda}|^{r_i}\theta_\lambda=\xi.
\end{align}
Here in $(1)$ we use Lemma \ref{ob:2}.

We now consider the restricted representation $(\beta|_{S_{i+1}},\,\mathcal{L})$. From \eqref{for:149} and \eqref{for:161}, it follows from Lemma \ref{le:23} that
the fractional equation
\begin{align*}
  |\Lambda_{i+1,\lambda}|^{r_{i+1}}\omega_{\lambda}=\theta_\lambda
\end{align*}
has a solution $\omega_{\lambda}\in \mathcal{L}$ with
\begin{align}\label{for:281}
 \|\omega_{\lambda}\|
   &\leq C_{r_{i+1},\gamma_{i+1}}
     \big\|\mathcal{R}_{i+1}^{\frac{\zeta_{i+1}}{2}}\theta_\lambda\big\|
     = C_{r_{i+1},\gamma_{i+1}}\|\psi_\lambda\|\notag\\
     &\overset{\text{(2)}}{\leq}
     C_{\mathfrak{r}_{i+1},\mathfrak{p}_{i+1}}\,
     \big\|\mathcal{R}_{i+1}^{\frac{\zeta_{i+1}}{2}}\mathcal{R}_i^{\frac{\zeta_i}{2}}\cdots\mathcal{R}_1^{\frac{\zeta_1}{2}}\xi\big\|.
\end{align}
Here in $(2)$ we use \eqref{for:161}.

Combining with \eqref{for:280} we get, for each $\lambda$,
\[
  |\Lambda_{1,\lambda}|^{r_1}\cdots|\Lambda_{i+1,\lambda}|^{r_{i+1}}\omega_\lambda
   = \xi,
\]
and the desired estimate in the $\mathcal D$-case.

  \emph{Case II}:  $\mathcal{R}=\Sigma$. Arguing as in \eqref{for:280}, from \eqref{for:160} and Lemma~\ref{ob:2} we have
  \begin{align}\label{for:88}
 |\Lambda_{1,\lambda}|^{r_1}|\Lambda_{2,\lambda}|^{r_2}\cdots |\Lambda_{i,\lambda}|^{r_i}\theta_\lambda=\mathcal{R}_{i+1}^{-\frac{\zeta_{i+1}}{2}}(\mathcal{R}_{i+1}^{\frac{\zeta_{i+1}}{2}}\xi)_\lambda
\end{align}
and
\begin{align*}
 \sum_{\lambda=1}^{2^i}\mathcal{R}_{i+1}^{-\frac{\zeta_{i+1}}{2}}(\mathcal{R}_{i+1}^{\frac{\zeta_{i+1}}{2}}\xi)_\lambda=\xi.
\end{align*}
\emph{Subcase $(a)$}:  $k_{i+1}=\RR$. From \eqref{for:149} and \eqref{for:161}, it follows from Proposition \ref{po:4} that, for each $\lambda$, the following holds:
the fractional equation
\begin{align*}
  |\Lambda_{i+1,\lambda}|^{r_{i+1}}\omega_{\lambda}=\theta_\lambda
\end{align*}
has a solution $\omega_{\lambda}\in \mathcal{L}$ satisfying the same estimate \eqref{for:281}.

\emph{Subcase $(b)$}: $k_{i+1}=\CC$. Then we have a decomposition (see right above Lemma \ref{le:20})
\begin{align*}
\theta_\lambda=(\theta_\lambda-\mathcal{P}_1\theta_\lambda)+(\mathcal{P}_2\mathcal{P}_1\theta_\lambda)+(\mathcal{P}_1\theta_\lambda-\mathcal{P}_2\mathcal{P}_1\theta_\lambda),
\end{align*}
where $\mathcal{P}_1:=\mathcal{P}_{U_{i+1},f}$ and $\mathcal{P}_2:=\mathcal{P}_{\mathrm{i}U_{i+1},f}$.

It follows from  Lemma \ref{le:20} that
\begin{gather*}
 |\Lambda_{1,\lambda}|^{r_1}\cdots |\Lambda_{i,\lambda}|^{r_i}
  (\theta_\lambda-\mathcal{P}_1\theta_\lambda)\in \mathcal{L}, \quad
  |\Lambda_{1,\lambda}|^{r_1}\cdots |\Lambda_{i,\lambda}|^{r_i}
  (\mathcal{P}_2\mathcal{P}_1\theta_\lambda)\in \mathcal{L},\\
  |\Lambda_{1,\lambda}|^{r_1}\cdots |\Lambda_{i,\lambda}|^{r_i}
  (\mathcal{P}_1\theta_\lambda-\mathcal{P}_2\mathcal{P}_1\theta_\lambda)\in \mathcal{L}.
\end{gather*}
 From \eqref{for:271} and \eqref{for:272} of Section \ref{sec:53} the fractional equation
\begin{align*}
    |U_{i+1}|^{r_{i+1}}\omega_{\lambda,1}=(\theta_\lambda-\mathcal{P}_1\theta_\lambda)+(\mathcal{P}_2\mathcal{P}_1\theta_\lambda)
  \end{align*}
  has a solution $\omega_{\lambda,1}\in \mathcal{L}$ with the estimate
  \begin{align*}
 \norm{\omega_{\lambda,1}}\leq C_{r_{i+1},\gamma_{i+1}}\big\|\Sigma_{i+1}^{\frac{\zeta_{i+1}}{2}}\theta_\lambda\big\|= C_{r_{i+1},\gamma_{i+1}}\|\psi_\lambda\|
    \overset{\text{(3)}}{\leq}C_{\mathfrak{r}_{i+1},\mathfrak{p}_{i+1}}\,
     \big\|\Sigma_{i+1}^{\frac{\zeta_{i+1}}{2}}\Sigma_i^{\frac{\zeta_i}{2}}\cdots\Sigma_1^{\frac{\zeta_1}{2}}\xi\big\|
\end{align*}
Here in $(3)$ we use \eqref{for:161}.

From \eqref{for:274} of Section \ref{sec:53} (for the restricted representation $(\beta|_{S_{i+1}},\,\mathcal{L})$), the fractional equation
\begin{align*}
    |\textrm{i}U_{i+1}|^{r_{i+1}}\omega_{\lambda,2}=\mathcal{P}_1\theta_\lambda-\mathcal{P}_2\mathcal{P}_1\theta_\lambda
  \end{align*}
  has a solution $\omega_{\lambda,2}\in \mathcal{L}$ with a similar estimate
  \begin{align*}
 \norm{\omega_{\lambda,2}}\leq C_{r_{i+1}}\big\|\Sigma_{i+1}^{\frac{\zeta_{i+1}}{2}}\theta_\lambda\big\|\leq C_{\mathfrak{r}_{i+1},\mathfrak{p}_{i+1}}\,
     \big\|\Sigma_{i+1}^{\frac{\zeta_{i+1}}{2}}\Sigma_i^{\frac{\zeta_i}{2}}\cdots\Sigma_1^{\frac{\zeta_1}{2}}\xi\big\|.
\end{align*}
Renumbering the family of solutions (real or complex case) so that
we have \(\omega_\lambda\), \(1\le\lambda\le 2^{i+1}\), each satisfying the
estimate above, we have
\[
  \sum_{\lambda=1}^{2^{i+1}}
  |\Lambda_{1,\lambda}|^{r_1}\cdots |\Lambda_{i+1,\lambda}|^{r_{i+1}}\omega_{\lambda}=\xi
\]
and for each $\lambda$,
\begin{align*}
 |\Lambda_{1,\lambda}|^{r_1}\cdots |\Lambda_{i+1,\lambda}|^{r_{i+1}}\omega_{\lambda}\in \mathcal{L},
  \qquad
  \|\omega_{\lambda}\|
  \le C_{\mathfrak{r}_{i+1},\mathfrak{p}_{i+1}}\,
     \big\|\Sigma_{i+1}^{\frac{\zeta_{i+1}}{2}}\Sigma_i^{\frac{\zeta_i}{2}}\cdots\Sigma_1^{\frac{\zeta_1}{2}}\xi\big\|.
\end{align*}
This completes the induction step from $i$ to $i+1$. In the case $\mathcal R=\Sigma$ this gives
\eqref{for:82}, and in the case $\mathcal R=\mathcal D$ it gives
\eqref{for:277} of Theorem~\ref{th:6}.

\subsection{Proof of Theorem \ref{th:6}: sharpness of the solvable range}

We finally prove the sharpness statement in Theorem~\ref{th:6}. The point is
that if one exponent $r_i$ exceeds the corresponding strong spectral gap
$\gamma_i$, then solvability already fails in the corresponding rank-one
factor. By localizing to that factor and smoothing in the remaining
directions, one constructs smooth data for which the full product equation has
no solution.

Assume that $r_1,\cdots, r_n\in \RR^+$ and
for some $1\le i\le n$, $r_i>\gamma_i$. Fix $c$ with $\gamma_i < c < r_i$.  Suppose $m\in\NN$ with $m>r_i$. We will construct $\xi\in W^m(\mathcal L)$ such that for any $\lambda\in\{1,\dots,2^n\}$,
the equation
\[
 |\Lambda_{1,\lambda}|^{r_1}|\Lambda_{2,\lambda}|^{r_2}\cdots |\Lambda_{n,\lambda}|^{r_n}\omega_\lambda = \xi
\]
has no solutions $\omega_\lambda\in\mathcal L$.


 We first construct $\xi$ so that: $(*)$ $\xi\in W^m(\mathcal L)$ and the single fractional equation $|\Lambda_{i,\lambda}|^{c}\,\omega = \xi$
has no solution $\omega\in\mathcal L$.

We have a direct integral decomposition:
\begin{align*}
  \beta=\int_Z \beta_z d\varsigma(z),\quad \mathcal{L}=\int_Z \mathcal{L}_z d\varsigma(z),
\end{align*}
for some measure space $(Z,\varsigma)$ (see Section \ref{sec:20}), where each $(\beta_z,\mathcal{L}_z)$ is an irreducible unitary
representation of $S$. By Remark \ref{re:4}
for $\varsigma$-a.e.\,$z$ we have
\[
 (\beta_z,\, \mathcal{L}_z)
 \cong
 \bigl(\pi_{l_1(z)}\otimes \cdots \otimes \pi_{l_n(z)},\,
       \mathcal{H}_{l_1(z)}\otimes \cdots \otimes \mathcal{H}_{l_n(z)}\bigr),
\]
where $(\pi_{l_j(z)},\mathcal{H}_{l_j(z)})$ is an irreducible
representation of $SL(2,k_j)$ with strong spectral gap $\tau_{z,j}$.


For $1\leq j\leq n$, let $E_{z,j,1}$ (resp. $E_{z,j,2}$) be the spectral measure of the positive self-adjoint operator $|U_j|$ (resp. $|\textrm{i}U_j|$ if $k_j=\CC$) in $\pi_{l_j(z)}$,
  so that
  \[
    |U_j| = \int_{[0,\infty)} \lambda\,dE_{z,j,1}(\lambda)\quad\text{and}\quad|\textrm{i}U_j| = \int_{[0,\infty)} \lambda\,dE_{z,j,2}(\lambda).
  \]
Set
\[
   \mathcal{P}_{z,j}=
   \begin{cases}
     \displaystyle \int_{[0,\infty)} f(\lambda)\,dE_{z,j,1}(\lambda), & k_j=\mathbb R,\\[8pt]
     \displaystyle \mathcal{P}_{z,j,1}\mathcal{P}_{z,j,2},\quad
       \mathcal{P}_{z,j,q}=\int_{[0,\infty)} f(\lambda)\,dE_{z,j,q}(\lambda),
       \,q=1,2, & k_j=\mathbb C,
   \end{cases}
\]
where $f$ is as in
\eqref{for:269} and set
\begin{align*}
 \mathcal{P}:= \int_Z \mathcal{P}_{z,1}\otimes \mathcal{P}_{z,2} \otimes\cdots \otimes \mathcal{P}_{z,i-1}\otimes I_i\otimes\mathcal{P}_{z,i+1}\cdots\otimes \mathcal{P}_{z,n} d\varsigma(z).
\end{align*}
Thus $\mathcal{P}$ acts as a smoothing operator on each factor
$\mathcal{H}_{l_j(z)}$, $j\neq i$, along $U_j$ and $\mathrm iU_j$ in the
complex case, while it acts as the identity on the $i$-th factor
$\mathcal{H}_{l_i(z)}$.

Let $Z_1=\{z\in Z: \tau_{z,i}<c\}$.  Since $\beta|_{S_i}$ has strong spectral gap $\gamma_i<c$, $\varsigma(Z_1)>0$.  Fix a subset $Z_2\subseteq Z_1$ with $0<\varsigma(Z_2)<\infty$. We construct a vector $\xi$ as follows:
\begin{align*}
\xi=\int_{Z_2} \xi_z\,d\varsigma(z),\qquad
\xi_z=v_{1,z}\otimes v_{2,z} \otimes\cdots \otimes v_{n,z},
\end{align*}
where
\begin{enumerate}
  \item For any $j\neq i$, we choose $v_{j,z}\in \mathcal{H}_{l_j(z)}$ with $\norm{v_{j,z}}_{\mathcal{H}_{l_j(z)},m}=1$.
  \item For any $i$, choose $v_{i,z}\in W^\infty(\mathcal{H}_{l_i(z)})$ with $\norm{v_{i,z}}_{\mathcal{H}_{l_i(z)}, m}=1$ satisfying
  the
  fractional equation $|\Lambda_{i,\lambda}|^{c}\omega=v_{i,z}$ has no solution $\omega\in \mathcal{H}_{l_i(z)}$  (see Lemma \ref{le:13}).

\item $\norm{\mathcal{P}_{z,j}v_{j,z}}_{\mathcal{H}_{l_j(z)},m}\leq C_m$ (see Lemma \ref{le:14}). We note that $C_m$ is independent of $z$.
\end{enumerate}
Then both $\xi$ and $\xi_1:=\mathcal{P}\xi$ satisfy $(*)$. Next, we show that the equation
\begin{align}\label{for:275}
|\Lambda_{1,\lambda}|^{r_1}|\Lambda_{2,\lambda}|^{r_2}\cdots |\Lambda_{n,\lambda}|^{r_n}\omega_\lambda = \xi
\end{align}
has no solutions $\omega_\lambda\in\mathcal L$. Otherwise, applying to $\mathcal{P}$ to \eqref{for:275}, we get
\begin{align*}
|\Lambda_{i,\lambda}|^{r_i} \omega_\lambda'= \xi_1, \qquad\text{where }\omega_\lambda'= (\Pi_{j\neq i}|\Lambda_{j,\lambda}|^{r_j})(\mathcal{P}\omega_\lambda ).
\end{align*}
By construction, $\omega_\lambda'\in \mathcal{L}$ since $\mathcal{P}$ is the smoothing operator along $\Lambda_{j,\lambda}$ for any $j\neq i$ and any $\lambda$. Next, we show that this implies that the equation
\[
 |\Lambda_{i,\lambda}|^{c} \omega_\lambda''= \xi_1
\]
has a solution $\omega_\lambda''\in \mathcal{L}$. This is exactly the trick we used at the end of  the proof of Lemma \ref{le:23}: solvability for a larger exponent implies solvability for a smaller exponent.
This contradicts the fact that  $\xi_1$ satisfies $(*)$. Then we complete the proof.

\section{Exponential order-$2$ mixing}\label{sec:34}

\subsection{Main results} We list  the notations that will appear in the following theorems:
\begin{enumerate}

  \item $\epsilon$, $(\pi, \mathcal{H})$, $\pi(a)$, $A$, $A^+$, $a$, $a^+$, $da$, $a$ of rational type and irrational type:  see Section \ref{sec:33}.

\item $W^{s,H}(\mathcal{H})$: see Section \ref{sec:9}.

  \item $H_{-,a}$, $H_{+,a}$, $H_{-0,a}$ and $H_{+0,a}$: see \eqref{for:123} of Section \ref{sec:33}.
  \item Maximal strongly orthogonal system $\mathcal{S}$,  $u(\mathcal{S})$, $S(\mathcal{S})$: see Sections \ref{for:220} and \ref{for:221}.
  \item $p_{\epsilon}(\mathcal{S})$, $\zeta_{\epsilon}(\mathcal{S})$, $\eta_\epsilon(\mathcal{S}, a)$: see Section \ref{for:222}.

\end{enumerate}

The next two results are in the semisimple setting. We recall that $\pi$ has a strong spectral gap (see \eqref{for:120} of Section \ref{sec:33}).

\begin{theorem}\label{th:7} Let $\mathcal{S}$ be a  strongly orthogonal system.  For any $a\in A^+$ and any
\begin{align*}
 \psi\in W^{p_\epsilon(\mathcal{S}),\, u(\mathcal{S})}(\mathcal{H})\quad\text{and}\quad \xi\in W^{\zeta_\epsilon(\mathcal{S}),\, S(\mathcal{S})}(\mathcal{H}),
\end{align*}
 we have
\begin{align*}
 \big|\langle \pi(a)\psi, \xi\rangle\big|\leq C_\epsilon \eta_\epsilon(\mathcal{S}, a)\norm{\psi}_{u(\mathcal{S}),p_{\epsilon}(\mathcal{S})}\,\big\|\xi\big\|_{S(\mathcal{S}),\zeta_{\epsilon}(\mathcal{S})}.
\end{align*}

\end{theorem}
\begin{remark}\label{re:7} To achieve an optimal decay rate, we let $\mathcal{S}$ be a maximal strongly orthogonal system.
\begin{itemize}

  \item \emph{About $C_\epsilon$}: $C_\epsilon$ depends on the spectral gap of $\pi$.

  \item If \(\pi|_{H_\theta}\) has no discrete-series summand for any \(\theta\in\mathcal S\)
(see Remark~\ref{re:12}), \emph{or} if we work with the tempered decay rate whenever a
discrete-series summand occurs, then \(p_{\epsilon}(\mathcal S)\) and
\(\zeta_{\epsilon}(\mathcal S)\) may be chosen \emph{independently of}
\(\eta_\epsilon(\mathcal S,\cdot)\), with dependence only on the real rank of \(G\).

  \item \emph{About partial norms}: We note that $S(\mathcal{S})\subseteq H_{-0,a}$ and $u(\mathcal{S})\subseteq H_{+,a}$. For any $a\in A$, there is $w$ in the Weyl group $W$ such that $w^{-1}aw\in A^+$. In that case, replace $S(\mathcal S)$ by $wS(\mathcal S)w^{-1}$ and $u(\mathcal S)$ by $wu(\mathcal S)w^{-1}$.
\end{itemize}

\end{remark}
\begin{corollary}\label{cor:9} Let $\mathcal{S}$ be a  strongly orthogonal system. For any  $0<s<1$, there exists $0<\gamma(s)<1$ (see \eqref{for:363}) such that for any $a\in A^+$ and any
\begin{align*}
 \psi\in W^{s,\, u(\mathcal{S})}(\mathcal{H})\quad\text{and}\quad \xi\in W^{s,\, S(\mathcal{S})}(\mathcal{H}),
\end{align*}
 we have
\begin{align*}
 \big|\langle \pi(a)\psi, \xi\rangle\big|\leq C_{s,\epsilon} \big(\eta_\epsilon(\mathcal{S}, a)\big)^{\gamma}\norm{\psi}_{u(\mathcal{S}),s}\,\big\|\xi\big\|_{S(\mathcal{S}),s}.
\end{align*}

\end{corollary}
\begin{remark}
Corollary \ref{cor:9} shows that partially hyperbolic algebraic actions have exponential mixing for partial $s$-H\"{o}lder vectors. We give the explicit dependence between $s$ and $\gamma$.
\end{remark}

\subsection{Proof of Theorem \ref{th:7}}\label{sec:21}We list  the notations that will appear in the following proofs:
\begin{enumerate}
  \item $H_\theta$, $H(\mathcal{S})$, $u_\theta$, $\theta(k)$, $\Psi_\theta$: see Sections \ref{for:221}. \, $\Phi^+$: see Section \ref{sec:36}.

  \item We can write $\mathcal{S}=\{\theta_1,\cdots, \theta_l\}$ of $\Phi^+$. For $1\le i\le l$,  set
  \begin{align*}
X_i=\Psi_{\theta_i}\begin{pmatrix}
  1 & 0 \\
  0 & -1
\end{pmatrix},\quad U_i=\Psi_{\theta_i}\begin{pmatrix}
  0 & 1 \\
  0 & 0
\end{pmatrix},\quad V_i=\Psi_{\theta_i}\begin{pmatrix}
  0 & 0 \\
  1 & 0
\end{pmatrix}.
  \end{align*}
\item For any $1\leq i\leq l$, set $k_i=\theta_i(k)$,  $\gamma_{\theta_i}=\gamma_i$,  $\zeta_{i,\epsilon}=\zeta_{\theta_i,\epsilon}$ (see Section \ref{for:222}).
  \item  For each $\lambda\in\{1,\dots,2^l\}$ and each $1\le i\le l$, set
\begin{align*}
 \Sigma_i&=\;
  \begin{cases}
    I - X_i^2 - V_i^2, &
      \text{if } k_i= \RR,\\[4pt]
    I - X_i^2 - (\mathrm{i}X_i)^2 - V_i^2 - (\mathrm{i}V_i)^2, &
      \text{if } k_i= \CC
  \end{cases}\qquad \text{and}\\
  \Lambda_{i,\lambda}&=
  \begin{cases}
    U_i, & \text{if } k_i=\RR,\\[4pt]
    U_i \text{ or } \mathrm{i}U_i, & \text{if } k_i=\CC,
  \end{cases}
\end{align*}
so that, as $\lambda$ ranges from $1$ to $2^l$, the $l$-tuples
$(\Lambda_{1,\lambda},\dots,\Lambda_{l,\lambda})$ exhaust all possible choices
of $U_i$ or $\mathrm{i}U_i$ in the complex factors.

\end{enumerate}
Then we have:
\begin{enumerate}

  \item We recall that for each $1\leq i\leq l$,  $\text{Lie}(H_{\theta_i})=\mathfrak{sl}(2,k_i)$.  Since $\mathcal{S}$ is strongly orthogonal, the subalgebras $\Lie(H_{\theta_i})$ commute pairwise, hence
      \begin{align*}
       \text{Lie}(H(\mathcal{S}))=\mathfrak{sl}(2,k_1)\times\cdots \times \mathfrak{sl}(2,k_l).
      \end{align*}
 \item \(\sum_{i=1}^{l}r_i=p_{\epsilon}(\mathcal{S})\) (recall Section \ref{for:221}), where $r_i=\gamma_i-\epsilon$.
   \item\label{for:195} $\xi\in W^{\zeta_{\epsilon}(\mathcal{S}),\, S(\mathcal{S})}(\mathcal{H})$ implies that $\Sigma_l^{\frac{\zeta_{l,\epsilon}}{2}}\cdots \Sigma_1^{\frac{\zeta_{1,\epsilon}}{2}}\xi\in\mathcal{H}$.
   \item\label{for:200} For any $1\leq i\leq l$,
$\pi|_{H_{\theta_i}}$ has strong spectral gap $\gamma_{i}$.

\end{enumerate}

We consider the restricted representation $\pi|_{H(\mathcal{S})}$. By the points above, we can apply
 Theorem \ref{th:6} to conclude that there is a decomposition $\xi=\sum_{\lambda=1}^{2^l}\xi_\lambda$ with $\xi_\lambda\in \mathcal{H}$, such that each of the equations
\begin{align*}
|\Lambda_{1,\lambda}|^{r_1}|\Lambda_{2,\lambda}|^{r_2}\cdots |\Lambda_{l,\lambda}|^{r_l}\omega_\lambda=\xi_\lambda, \quad 1\leq \lambda\leq 2^l,
\end{align*}
has a solution $\omega_\lambda\in \mathcal{H}$ with the estimate
\begin{align}\label{for:135}
\norm{\omega_\lambda} \leq C_{\gamma,\epsilon} \big\|\Sigma_l^{\frac{\zeta_{l,\epsilon}}{2}}\cdots \Sigma_1^{\frac{\zeta_{1,\epsilon}}{2}}\xi\big\|\leq C_{\gamma,\epsilon} \norm{\xi}_{S(\mathcal{S}),\zeta_{\epsilon}(\mathcal{S})},
\end{align}
where $\gamma=(\gamma_1,\cdots, \gamma_l)$. Then we have
\begin{align}\label{for:166}
 \big|\langle \pi(a)\psi, \,\xi\rangle\big|&=\Big|\sum_{\lambda=1}^{2^l}\langle \pi(a)\psi, \,\xi_\lambda\rangle\Big|\leq\sum_{\lambda=1}^{2^l}\Big|\big\langle \pi(a)\psi, \, |\Lambda_{1,\lambda}|^{r_1}|\Lambda_{2,\lambda}|^{r_2}\cdots |\Lambda_{l,\lambda}|^{r_l}\omega_\lambda\big\rangle\Big|\notag\\
 &\overset{\text{(i)}}{=}\sum_{\lambda=1}^{2^l}\Big|\big\langle |\Lambda_{1,\lambda}|^{r_1}|\Lambda_{2,\lambda}|^{r_2}\cdots |\Lambda_{l,\lambda}|^{r_l}(\pi(a)\psi), \, \omega_\lambda\big\rangle\Big|\notag\\
 &\overset{\text{(ii)}}{=}\Pi_{i=1}^l c_i^{r_i}\Big|\big\langle \pi(a)\big(|\widetilde{\Lambda_{1,\lambda}}|^{r_1}|\widetilde{\Lambda_{2,\lambda}}|^{r_2}\cdots |\widetilde{\Lambda_{l,\lambda}}|^{r_l}\psi\big), \, \omega_\lambda\big\rangle\Big|\notag\\
 &\overset{\text{(iii)}}{\leq} \Pi_{i=1}^l c_i^{r_i}\big\|  |\widetilde{\Lambda_{1,\lambda}}|^{r_1}|\widetilde{\Lambda_{2,\lambda}}|^{r_2}\cdots |\widetilde{\Lambda_{l,\lambda}}|^{r_l}\psi \big\|\cdot \norm{\omega_\lambda}\notag\\
 &\overset{\text{(vi)}}{\leq} C_{\gamma,\epsilon} \,\eta_\epsilon(\mathcal{S}, a)\|\psi\|_{u(\mathcal{S}),\,p_{\epsilon}(\mathcal{S})} \cdot \norm{\xi}_{S(\mathcal{S}),\zeta_{\epsilon}(\mathcal{S})}.
\end{align}
We explain steps:
\begin{enumerate}
  \item [(i)] We use \eqref{for:163} of Lemma \ref{le:17}.
  \item [(ii)] We use \eqref{for:169}  of Lemma \ref{le:17}, where
\begin{align*}
 c_i=\|\text{Ad}_{a^{-1}}\Lambda_{i,\lambda}\|=\theta_i(a)^{-1}\|\Lambda_{i,\lambda}\|,\quad \widetilde{\Lambda_{i,\lambda}}=\frac{\text{Ad}_{a^{-1}}\Lambda_{i,\lambda}}{\|\text{Ad}_{a^{-1}}\Lambda_{i,\lambda}\|}=\frac{\Lambda_{i,\lambda}}{\|\Lambda_{i,\lambda}\|}.
\end{align*}
  \item [(iii)] We use Cauchy-Schwarz inequality and the fact that $\pi$ is unitary.
  \item [(vi)] We use \eqref{for:135}, the fact that $u_{\theta_i}\subseteq u(\mathcal{S})$ (see Section \ref{for:221}) and $\sum_{i=1}^lr_i=p_{\epsilon}(\mathcal{S})$ (see \eqref{for:195}), and
\begin{align*}
 \Pi_{i=1}^lc_i^{r_i}\leq C\Pi_{i=1}^l \theta_i(a)^{-r_i} =C\Pi_{i=1}^l \theta_i(a)^{-(\gamma_i-\epsilon)}=\eta_\epsilon(\mathcal{S}, a)\quad\text{(see Section \ref{for:222})}.
\end{align*}
\end{enumerate}
 Then we complete the proof.

\subsubsection{Proof of Corollary \ref{cor:9}} Set
\begin{align}\label{for:363}
 s_0=\zeta_1(\mathcal{S})\quad\text{and}\quad \gamma(s)=\min\!\Big\{\frac{s}{4s_0},\,\frac{1}{2}\Big\}.
\end{align}
We apply the smoothing operator $\mathfrak{s}_b$ (see Section \ref{sec:9}) to $\psi$ and $\xi$ for $(\pi|_{H_{+,a}},\,\mathcal{H})$ and $(\pi|_{H_{-0,a}},\,\mathcal{H})$ respectively. Then
 \begin{align*}
  \mathfrak{s}_b \psi\in W^{\infty,\,H_{+,a}}(\mathcal H)\quad\text{and}\quad \mathfrak{s}_b\xi\in W^{s,\,H_{-0,a}}(\mathcal H).
 \end{align*}
 It follows from Theorem~\ref{th:7} that
\begin{align}\label{for:199}
 \big|\langle \pi(a)\mathfrak{s}_b\psi,\, \mathfrak{s}_b\xi\rangle\big|&\leq C_\epsilon \eta_\epsilon(\mathcal{S}, a)\norm{\mathfrak{s}_b\psi}_{u(\mathcal{S}),s_0}\,\big\|\mathfrak{s}_b\xi\big\|_{S(\mathcal{S}),s_0}\notag\\
 &\overset{\text{(1)}}{\leq} C_{\epsilon,1} \eta_\epsilon(\mathcal{S}, a)b^{2s_0}\norm{\psi}\,\big\|\xi\big\|.
\end{align}
Here in $(1)$ we use \eqref{for:197} of Section \ref{sec:9}. Then we have
\begin{align*}
\big|\langle \pi(a)\psi, \,\xi\rangle\big|&=\big|\langle \pi(a)(\psi-\mathfrak{s}_b\psi)+\pi(a)\mathfrak{s}_b\psi, \,(\xi-\mathfrak{s}_b\xi)+\mathfrak{s}_b\xi\rangle\big|\notag\\
&\leq \big|\langle \pi(a)(\psi-\mathfrak{s}_b\psi), \,\xi-\mathfrak{s}_b\xi\rangle\big|+\big|\langle \pi(a)(\psi-\mathfrak{s}_b\psi), \,\mathfrak{s}_b\xi\rangle\big|\notag\\
&+\big|\langle \pi(a)\mathfrak{s}_b\psi, \,(\xi-\mathfrak{s}_b\xi)\rangle\big|+\big|\langle \pi(a)\mathfrak{s}_b\psi, \,\mathfrak{s}_b\xi\rangle\big|\notag\\
&\overset{\text{(1)}}{\leq} \|\psi-\mathfrak{s}_b\psi\|\|\xi-\mathfrak{s}_b\xi\|+\|\psi-\mathfrak{s}_b\psi\| \|\mathfrak{s}_b\xi\|\notag\\
&+\|\mathfrak{s}_b\psi\| \|\xi-\mathfrak{s}_b\xi\|+\big|\langle \pi(a)\mathfrak{s}_b\psi, \,\mathfrak{s}_b\xi\rangle\big|\notag\\
&\overset{\text{(2)}}{\leq}C_sb^{-2s}\|\psi\|_{u(\mathcal{S}),s}\norm{\xi}_{S(\mathcal{S}),\,s}+C_sb^{-s}\|\psi\|_{u(\mathcal{S}),s}\norm{\xi}\notag\\
&+C_sb^{-s}\|\psi\|\norm{\xi}_{S(\mathcal{S}),\,s}+C_\epsilon \eta_\epsilon(\mathcal{S}, a) b^{2s_0}\|\psi\|\norm{\xi}.
\end{align*}
Here in $(1)$ we use Cauchy-Schwarz inequality and the fact that $\pi$ is unitary; in $(2)$ we use \eqref{for:197} and \eqref{for:198} of Section \ref{sec:9}  and \eqref{for:199}.

Let $b=\big(\eta_\epsilon(\mathcal{S}, a)\big)^{-\frac{1}{4s_0}}$. Then the above inequality implies that
\begin{align*}
 \big|\langle \pi(a)\psi,\, \xi\rangle\big|&\leq C_\epsilon \big(\eta_\epsilon(\mathcal{S}, a)\big)^\gamma \norm{\psi}_{u(\mathcal{S}),s}\,\big\|\xi\big\|_{S(\mathcal{S}),s}.
\end{align*}
Then we complete the proof.

\section{Higher order exponential mixing for rank-one actions}

\subsection{Main results} We list  the notation that will appear in the following theorems:
\begin{enumerate}

  \item $(\pi, \mathcal{H})$, where  $\mathcal H=L_0^2(\mathcal X,\varrho)$:  see \eqref{for:120} of Section \ref{sec:33}.

\item strong spectral gap: see \eqref{for:120} of Section \ref{sec:33}.

\item $W_{+,a}$, $W_{-,a}$, $W_{0,a}$: see Section \ref{sec:33}.

\item $H_{0,a}$, $H_{-,a}$, $H_{+,a}$, $H_{-0,a}$ and $H_{+0,a}$: see \eqref{for:123} of Section \ref{sec:33}.

  \item Maximal strongly orthogonal system $\mathcal{S}$,  $u(\mathcal{S})$, $S(\mathcal{S})$: see Sections \ref{for:220} and \ref{for:221}.
  \item $p_{\epsilon}(\mathcal{S})$, $\zeta_{\epsilon}(\mathcal{S})$, $\eta_\epsilon(\mathcal{S}, a)$, $a\in A$: see Section \ref{for:222}.

  \item Essentially semisimple,  $\tau$, $s_z$ and $w_z$: see  \eqref{for:173} of Section \ref{sec:33}.

\end{enumerate}
Let $\alpha$ be a rank one abelian partially hyperbolic algebraic action on $\mathcal{X}$. Let $\mathcal Z\subset G$ be a closed abelian subgroup of rank one and is essentially semisimple.
Then there exists a maximal split torus $A\le G$ and a one-parameter subgroup $(a^t)\subset A$
such that for every $z\in\mathcal Z$ the semisimple part satisfies $s_z=a^{\tau(z)}$ (see  \eqref{for:173} of Section \ref{sec:33}).

Let $\mathcal S$ be a maximal strongly orthogonal system for $A$. Recall that $(\pi, \mathcal{H})$ has a strong spectral gap.

\begin{theorem}\label{th:8}  \emph{(Semisimple setting)}
 For any $f_1,\cdots, f_n\in C_c^{\infty}(\mathcal{X})$ and any $z_1,\cdots,z_n\in\mathcal{Z}$, we have:
 \begin{align*}
 &\Big|\int_{\mathcal{X}} \Pi_{i=1}^{n} \pi(z_i)f_i \,d\varrho - \Pi_{i=1}^{n} \int_{\mathcal{X}} f_i \,d\varrho\Big|\\
 &\leq C_{n,\epsilon}\,\max_{1\le i\neq j\le n} \eta_\epsilon(\mathcal{S}, s_{z_i-z_j})\Pi_{i=1}^{n}\norm{f_{i}}_{C^{\zeta_{\epsilon}(\mathcal{S})}}.
\end{align*}
 Moreover, if   $\int_{\mathcal{X}} f_i \, d\varrho=0$, $1\leq i\leq n$, then
  \begin{enumerate}
     \item\label{for:207} Let $q\in \operatorname*{arg\,min}_{1\le i\le n} \tau(z_i)$.  Then
  \begin{align*}
   &\Big|\int_{\mathcal{X}} \Pi_{i=1}^{n} \pi(z_i)f_i \,d\varrho \Big|\\
   &\le  C_{\epsilon,n}\max_{1\le i\neq j\le n} \eta_\epsilon(\mathcal{S}, s_{z_i-z_j}) \big(\Pi_{i\neq q}\norm{f_i}_{H_{+,a}, C^{p_{\epsilon}(\mathcal{S})}}\big)\norm{f_q}_{H_{-0,a}, C^{\zeta_{\epsilon}(\mathcal{S})}}.
  \end{align*}
     \item \label{for:208} If $n=3$, then
      \begin{align*}
   \Big|\int_{\mathcal{X}} \Pi_{i=1}^{3}\pi(z_i)f_i \, d\varrho\Big|&\leq C_\epsilon \big(\min_{1\le i,j\le 3} \eta_\epsilon(\mathcal{S}, s_{z_i-z_j})\big)^{\frac{1}{2}}\Pi_{i=1}^3\|f_i\|_{C^{\zeta_{\epsilon}(\mathcal{S})}}.
  \end{align*}

\item\label{for:203} If $n\geq 4$ no uniform bound in terms of $\min_{1\le i,j\le n} \eta_\epsilon(\mathcal{S}, s_{z_i-z_j})$ can hold in general.

  \end{enumerate}

\end{theorem}
\begin{remark} If $a\in A^+$, then we can restrict  $H_{+,a}$ to $u(\mathcal{S})$ and $H_{-0,a}$ to $S(\mathcal{S})$  (see  Section \ref{for:221}). For general $a\in A$,  we recall that there exists $\textbf{w}$ in the Weyl group such that $a^+=\textbf{w}^{-1}a\textbf{w}\in A^+$ (see \eqref{for:120} of Section \ref{sec:33}). Then we can restrict $H_{+,a}$ to $\textbf{w}u(\mathcal{S})\textbf{w}^{-1}$ and $H_{-0,a}$ to $\textbf{w}S(\mathcal{S})\textbf{w}^{-1}$.
\end{remark}

\subsection{Proof sketches}

We first treat the mean-zero case. Reorder the parameters so that
$\tau(z_n)=\min_{1\le i\le n}\tau(z_i)$,
and factor out the earliest time. After conjugating into the positive chamber,
we rewrite the multilinear correlation in the form
\[
\mathfrak m\big(\pi(z_1)f_1,\dots,\pi(z_n)f_n\big)
   = \mathfrak m\big(\pi((a^+)^\kappa)\mathcal F,\widetilde f_n\big),
\]
where
\[
\kappa=\min_{1\le i\le n-1}\big(\tau(z_i)-\tau(z_n)\big),
\qquad
\mathcal F=\prod_{i=1}^{n-1}\pi((a^+)^{\tau(z_i-z_n)-\kappa})\widetilde f_i.
\]
We then apply the two-point mixing estimate (Theorem~\ref{th:7}) to the
pair \(\big(\pi((a^+)^\kappa)\mathcal F,\widetilde f_n\big)\).

The key technical point is to control the partial Sobolev norm
\(\|\mathcal F\|_{H_{+,a^+},t}\). This is done by combining a
Leibniz/Kato--Ponce estimate with the contraction of
\(\Ad_{(a^+)^{-m}}\) on the unstable directions, which yields a bound in terms of $\Pi_{i=1}^{n-1}\|f_i\|_{H_{+,a},C^t}$.
In the non-split case, conjugation by the compact factors produces only
polynomial losses, and these are absorbed into the exponential decay after
slightly shrinking \(\epsilon\).

For general \(f_i\), we expand each function into its mean-zero part
\(f_i-\int_{\mathcal X}f_i\,d\varrho\) plus its mean, and apply the
mean-zero estimate to each nonempty subset. The sharper three-point bound is
obtained by grouping according to the larger of the two adjacent gaps, while
the obstruction for \(n\ge4\) follows from an explicit counterexample.

\subsection{Proof of Theorem \ref{th:8}}  For any $z_i\in \mathcal{Z}$ and $f_i\in C_c^\infty(\mathcal{X})$, $1\leq i\leq n$, set
\begin{align*}
  \mathfrak{m}\big(\pi(z_1)f_1,\pi(z_2)f_2,\cdots,\pi(z_n)f_n \big)=\int_{\mathcal{X}} \Pi_{i=1}^{n} \pi(z_i)f_i \,d\varrho.
\end{align*}
\emph{Case I: $\int_{\mathcal{X}} f_i \, d\varrho=0$, $1\leq i\leq n$.} Without loss of generality, assume $\tau(z_{n})=\min_{1\leq j\leq n} \tau(z_j)$. Set
\begin{align*}
 \kappa=\min_{1\leq i\leq n-1} (\tau(z_i)-\tau(z_{n}))>0.
\end{align*}
We recall that there exists $\textbf{w}$ in the Weyl group such that $a^+=\textbf{w}^{-1}a\textbf{w}\in A^+$ (see \eqref{for:120} of Section \ref{sec:33}).
Then
\begin{align}\label{for:128}
 \text{Ad}_{\textbf{w}}(W_{+,a^+})\subseteq W_{+,a}\quad\text{and}\quad \text{Ad}_{\textbf{w}}(W_{-0,a^+})\subseteq W_{-0,a}.
\end{align}
Then for any $z\in \mathcal{Z}$, we can write
\begin{align*}
 z=s_zw_z=a^{\tau(z)}w_z=\textbf{w}(a^+)^{\tau(z)}\textbf{w}^{-1}w_z.
\end{align*}
Since $w_{z_i}$ commutes with $a^{\tau(z_i)}$, it preserves the $a$-(un)stable/neutral directions, hence
\begin{align}\label{for:129}
 \text{Ad}_{w_z}(W_{+,a})\subseteq W_{+,a}\quad\text{and}\quad \text{Ad}_{w_z}(W_{-0,a})\subseteq W_{-0,a}.
\end{align}
Using the $G$-invariance of $\varrho$ and the change of variables $x\mapsto \mathbf w^{-1}x$, we have
\begin{align}\label{for:2}
 &\mathfrak{m}\big(\pi(z_1)f_1,\pi(z_2)f_2,\cdots,\pi(z_n)f_n \big)=\mathfrak{m}(\pi((a^+)^\kappa)\mathcal{F}, \widetilde{f_{n}}),
\end{align}
where $\widetilde{f_i}=\pi(\textbf{w}^{-1}w_{z_i-z_n})f_i$, $1\leq i\leq n$ and
\begin{gather*}
\mathcal{F}=\Pi_{i=1}^{n-1} \pi((a^+)^{\tau(z_i-z_n)-\kappa})\widetilde{f_i}.
\end{gather*}
Then from \eqref{for:127} and \eqref{for:125} of Section \ref{sec:33}, \eqref{for:128} and \eqref{for:129} we have
\begin{align}\label{for:130}
 \norm{\widetilde{f_i}}_{\tilde{H}, C^t}\leq C_t(1+|\tau(z_i-z_n)|)^{t\dim G}\norm{f_i}_{H, C^t},\qquad \forall\,t\geq0,\,1\leq i\leq n
\end{align}
with $(\tilde H,H)\in\{(H_{+,a^+},H_{+,a}),\,(H_{-0,a^+},H_{-0,a}),\,(G,G)\}$.

Since $W_{+,a^+}$ is $\text{Ad}_{a^+}$-invariant with the minimal Lyapunov exponent of $\text{Ad}_{a^+}|_{W_{+,a}}\geq \chi$, where $\chi>0$ is the minimal absolute nonzero Lyapunov exponent of $a^+$ and $\tau(z_i-z_{n})-\kappa\ge0$
\begin{align*}
 \big\|\text{Ad}_{(a^+)^{-(\tau(z_i-z_{n})-\kappa))}}\big|_{W_{+,a^+}}\big\|\leq C_\epsilon e^{-(\chi-\epsilon)(\tau(z_i-z_{n})-\kappa)}, \qquad 1\leq i\leq n-1.
\end{align*}
This together with \eqref{for:127} of Section \ref{sec:33} gives: for any $1\leq i\leq n-1$
\begin{align*}
 &\big\|\text{Ad}_{\textbf{w}w_{z_i-z_n}^{-1}}\circ\text{Ad}_{(a^+)^{-(\tau(z_i-z_{n})-\kappa))}}\big|_{W_{+,a^+}}\big\|\\
 &\leq C_\epsilon e^{-(\chi-\epsilon)(\tau(z_i-z_{n})-\kappa)}(1+|\tau(z_i-z_n)|)^{\dim G}\overset{\text{(1)}}{\leq}C(1+\kappa)^{\dim G}.
\end{align*}
Here in $(1)$ we note that
\[
1+|\tau(z_i)-\tau(z_n)|\ \le\ \big(1+|\tau(z_i)-\tau(z_n)-\kappa|\big)\,(1+\kappa).
\]
It follows from \eqref{for:125} of Section \ref{sec:33} that for any $t\geq0$ and $1\leq i\leq n-1$
\begin{align*}
 \big\|\pi((a^+)^{\tau(z_i-z_n)-\kappa})&\widetilde{f_i}\big\|_{H_{+,a^+},\,C^{t}}=\big\|\pi((a^+)^{\tau(z_i-z_n)-\kappa})\pi(\textbf{w}^{-1}w_{z_i-z_n})f_i\big\|_{H_{+,a^+},\,C^{t}}\\
 &\leq C_t(1+\kappa)^{t\dim G}\big\|f_i\big\|_{H_{+,a},\,C^{t}}
\end{align*}
 Thus, by Leibniz
(and Kato-Ponce for fractional orders) and \eqref{for:189} of Section \ref{sec:33},  for any $t\geq0$ we have
\begin{align}\label{for:133}
\|\mathcal F\|_{H_{+,a^+},\,t}&\le C_{t,n}(1+\kappa)^{(n-1)t\dim G}\prod_{i=1}^{n-1}\|f_i\|_{H_{+,a},\,C^t}.
\end{align}
\begin{remark}
If $\mathcal Z$ is contained in a split Cartan subgroup (so $s_z\in A$ and $w_z=e$ for all $z\in\mathcal Z$), then the conjugation factors by $w_z$ disappear and the transfer estimate \eqref{for:130} has no polynomial growth:
\[
\|\widetilde f_i\|_{\tilde H,C^t}\;\le\; C_t\,\|f_i\|_{H,C^t}\qquad (t\ge0).
\]
Thus, by Leibniz
(and Kato-Ponce for fractional orders)
\begin{align}\label{for:126}
\|\mathcal F\|_{H_{+,a},t}&\leq\|\mathcal F\|_{H_{+,a},C^t}\le C_{t,n}\Pi_{i=1}^{n-1} \|\pi(a^{\tau(z_i-z_n-z)})f_i\|_{H_{+,a},\,C^{t}}\notag\\
&\leq  C_{t,n,1}\Pi_{i=1}^{n-1}\|f_i\|_{H_{+,a},C^t}
\end{align}
for any $t\ge0$.

In the general (non-split) case, the elements $w_z$ need not be trivial, and the factors $(1+|\tau(z_i)-\tau(z_n)|)^{t\dim G}$ in \eqref{for:130} are unavoidable; hence $\|\mathcal F\|_{H_{+,a^+},t}$ is not uniformly bounded by a constant independent of the time gaps. We compensate for this polynomial growth by absorbing it into the exponential decay $\eta_{\epsilon/2}(\mathcal S,a^+)^{\kappa}$ after shrinking $\epsilon$.
\end{remark}
It follows from Theorem \ref{th:7} and Remark \ref{re:7} that: for any $\psi,\,\xi\in C^\infty_c(\mathcal{X})$
\begin{align*}
 \big|\langle \pi((a^+)^\kappa)\psi, \xi\rangle\big|\leq C_\epsilon \eta_{\frac{\epsilon}{2}}(\mathcal{S}, a^+)^{\kappa}\norm{\psi}_{H_{+,a^+},\,p_{\epsilon}(\mathcal{S})}\,\big\|\xi\big\|_{H_{-0,a^+},\,\zeta_{\epsilon}(\mathcal{S})}.
\end{align*}
For simplicity, denote $p_{\epsilon}(\mathcal{S})$ by $r_1$ and $\zeta_{\epsilon}(\mathcal{S})$ by $r_2$. Then $r_2\geq r_1$ by \eqref{for:4} of Remark \ref{re:12}. Thus, we have
\begin{align*}
 \big|\mathfrak{m}(\pi((a^+)^\kappa)\mathcal{F}, \tilde{f_{n}})\big|&\overset{\text{(1)}}{\leq} C_\epsilon \eta_{\frac{\epsilon}{2}}(\mathcal{S}, a^+)^{\kappa} \norm{\mathcal{F}}_{H_{+,a^+}, C^{r_{1}}}\norm{\widetilde{f_{n}}}_{H_{-0,a^+},\,C^{r_{2}}}\notag\\
&\overset{\text{(2)}}{\leq}C_{n,\epsilon}\eta_{\epsilon}(\mathcal{S}, a^+)^{\kappa}\big(\Pi_{i=1}^{n-1}\norm{f_i}_{H_{+,a},\, C^{r_{1}}}\big)\norm{f_n}_{H_{-0,a}, C^{r_{2}}}\\
&\overset{\text{(3)}}{\leq}   C_{n,\epsilon,1}\,\max_{1\le i\neq j\le n} \,\eta_{\epsilon}(\mathcal{S}, s_{z_i-z_j}) \big(\Pi_{i=1}^{n-1}\norm{f_i}_{H_{+,a},\, C^{r_{1}}}\big)\norm{f_n}_{H_{-0,a}, C^{r_{2}}}.
\end{align*}
Here in $(1)$ we recall $\kappa>0$ and we shrink $\epsilon$ to $\frac{\epsilon}{2}$ so that  the polynomial factor $(1+z)^{(n-1)r_1\dim G}$ arising in the bound of
\(\|\mathcal{F}\|_{H_{+,a^+}, C^{r_{1}}}\) is absorbed by
$\eta_{\epsilon/2}(\mathcal S,a^+)^\kappa$; in $(2)$ we use \eqref{for:130} and \eqref{for:133}; in $(3)$ we note that $\kappa\ge\gamma=\min_{1\le i\le n} \tau(z_i)$ and
\begin{align*}
\eta_{\epsilon}(\mathcal{S}, s_{z_i-z_j})=\eta_{\epsilon}(\mathcal{S}, a^{\tau(z_i-z_j)})=\eta_{\epsilon}(\mathcal{S}, a^+)^{|\tau(z_i-z_j)|}\quad \text{and}\quad
\eta_{\epsilon}(\mathcal{S}, a)<1.
\end{align*}
This together with \eqref{for:2} proved \eqref{for:207}.

\smallskip
\emph{Case II (general means):} Let $\vec{f_i}=f_i-c_i$, where $c_i=\int_{\mathcal{X}} f_i \, d\varrho$, $1\leq i\leq n$.
Then
\begin{align*}
\Pi_{i=1}^n \pi(z_i)f_i=\sum_{J\subseteq\{1,\dots,n\}}
\big(\Pi_{i\in J}\pi(z_i)\vec{f_i} \big)
\big(\Pi_{i\notin J} c_i\big),\quad \int_{\mathcal{X}} \vec{f_i} \, d\varrho=0.
\end{align*}
The $J=\varnothing$ term equals $\Pi_i c_i=\Pi_i\int_{\mathcal X}f_i\,d\varrho$.
For every nonempty $J$,  by Case $I$ we have
\begin{align*}
&\Big|\int_{\mathcal X}\big(\Pi_{i\in J}\pi(z_i)\vec f_i\big)
\big(\Pi_{i\notin J} c_i\big)\,d\varrho\Big|\\
&\le\;
C_{n,\epsilon}\,\max_{ i\neq j,\,i,j\in J} \,\eta_{\epsilon}(\mathcal{S}, s_{z_i-z_j})\Pi_{i\in J}\|\vec f_i\|_{C^{r_2}}\big(\Pi_{i\notin J} \norm{f_i}_{C^0}\big)\\
 &\le  C_{n,\epsilon}\,\max_{1\le i\neq j\le n} \,\eta_{\epsilon}(\mathcal{S}, s_{z_i-z_j})\Pi_{i=1}^n\|f_i\|_{C^{r_2}}.
\end{align*}
Summing over $J\neq\varnothing$
yields the result.

\smallskip
\emph{Case III: $\int_{\mathcal{X}} f_i \, d\varrho=0$, $1\leq i\leq 3$.} Without loss of generality, assume that $\tau(z_1)\geq \tau(z_2)\geq \tau(z_3)$. Then
\begin{align}\label{for:3}
 \tau(z_1)-\tau(z_3)=\max_{1\leq i\neq j\leq 3}|\tau(z_i)-\tau(z_j)|.
\end{align}
$\textbf{(i)}$: If $\tau(z_{2})-\tau(z_{3})>\frac{1}{2}(\tau(z_1)-\tau(z_3))$, set
\begin{align*}
\mathcal{F}= \big(\pi((a^+)^{\tau(z_1)-\tau(z_2)})\tilde{f_1}\big) \tilde{f_2},\qquad \text{where }\tilde{f_i}=\pi(\textbf{w}^{-1}w_{z_i-z_3})f_i.
\end{align*}
Then
\begin{align}\label{for:343}
 \mathfrak{m}\big(\pi(z_1)f_1,\,\pi(z_2)f_2,\, \pi(z_3)f_{3}\big)=\mathfrak{m}\big(\pi((a^+)^{\tau(z_2)-\tau(z_3)})\mathcal{F},\,\tilde{f_{3}}\big).
\end{align}
It follows from Theorem \ref{th:7} that
\begin{align}\label{for:344}
 \big|\mathfrak{m}\big(\pi((a^+)&^{\tau(z_2)-\tau(z_3)})\mathcal{F},\,\tilde{f_{3}}\big)\big|\leq C_\epsilon \eta_{\frac{\epsilon}{2}}(\mathcal{S}, a^+)^{\tau(z_2)-\tau(z_3)}\norm{\mathcal{F}}_{H_{+,a^+},\,r_1}\,\big\|\tilde{f_{3}}\big\|_{H_{-0,a^+},\,r_2}\notag\\
 &\overset{\text{(1)}}{\leq} C_{\epsilon}\eta_{\frac{\epsilon}{2}}(\mathcal{S}, a^+)^{\frac{1}{2}\max_{1\leq i,j\leq 3}\abs{\tau(z_i)-\tau(z_j)}}\norm{\mathcal{F}}_{H_{+,a^+},\,C^{r_1}}\,\big\|\tilde{f_{3}}\big\|_{H_{-0,a^+},\,C^{r_2}}\notag\\
 &\leq C_{\epsilon,1}\eta_{\frac{\epsilon}{2}}(\mathcal{S}, a^+)^{\frac{1}{2}\max_{1\leq i,j\leq 3}\abs{\tau(z_i)-\tau(z_j)}}\norm{\mathcal{F}}_{H_{+,a^+},\,C^{r_1}}\,\big\|f_{3}\big\|_{C^{r_2}}.
\end{align}
Here in $(1)$ we use \eqref{for:189} of Section \ref{sec:33} and  we recall the assumption $\tau(z_{2})-\tau(z_{3})>\frac{1}{2}(\tau(z_1)-\tau(z_3))$.

Since $W_{+,a^+}$ is $\text{Ad}_{a^+}$-invariant with the minimal Lyapunov exponent of $\text{Ad}_{a^+}|_{W_{+,a}}>0$ and and $\tau(z_1)-\tau(z_{2})\geq0$,
it follows from
\eqref{for:125} of Section \ref{sec:33} that
\begin{align*}
 \|\pi((a^+)^{\tau(z_1)-\tau(z_2})\widetilde{f_1}\|_{H_{+,a^+},\,C^{r_1}}\leq C\|\widetilde{f_1}\|_{H_{+,a^+},\,C^{r_1}}.
\end{align*}
Thus, by Leibniz
(and Kato-Ponce for fractional orders), we have
\begin{align*}
 \norm{\mathcal{F}}_{H_{+,a^+},\,C^{r_1}}&\leq C\|\pi((a^+)^{\tau(z_1)-\tau(z_2})\widetilde{f_1}\|_{H_{+,a^+},\,C^{r_1}}\|\widetilde{f_2}\|_{H_{+,a^+},\,C^{r_1}}\\
 &\leq C_1\|\widetilde{f_1}\|_{H_{+,a^+},\,C^{r_1}}\|\widetilde{f_2}\|_{H_{+,a^+},\,C^{r_1}} \\
 &\leq C(1+|\tau(z_1-z_3)|)^{r_1\dim G}(1+|\tau(z_2-z_3)|)^{r_1\dim G}\norm{f_1}_{C^{r_1}}\norm{f_2}_{C^{r_1}}\\
 &\overset{\text{(1)}}{\leq} C_1(1+|\tau(z_1-z_3)|)^{(r_1+r_2)\dim G}\norm{f_1}_{C^{r_2}}\norm{f_2}_{C^{r_2}}.
\end{align*}
Here in $(1)$ we recall \eqref{for:3} and $r_2\geq r_1$.

This together with \eqref{for:344} give
\begin{align*}
 &\big|\mathfrak{m}\big(\pi((a^+)^{\tau(z_2)-\tau(z_3)})\mathcal{F},\,\tilde{f_{3}}\big)\big|\\
 &\leq C_{\epsilon}\eta_{\frac{\epsilon}{2}}(\mathcal{S}, a^+)^{\frac{1}{2}\max_{1\leq i,j\leq 3}\abs{\tau(z_i)-\tau(z_j)}}(1+|\tau(z_1-z_3)|)^{(r_1+r_2)\dim G}\Pi_{i=1}^3\norm{f_i}_{C^{r_2}}\\
 &\overset{\text{(1)}}{\leq} C_{\epsilon}\eta_{\epsilon}(\mathcal{S}, a^+)^{\frac{1}{2}\max_{1\leq i,j\leq 3}\abs{\tau(z_i)-\tau(z_j)}}\Pi_{i=1}^3\norm{f_i}_{C^{r_2}}.
\end{align*}
Here in $(1)$ We shrink $\epsilon$ to $\frac{\epsilon}{2}$ so that  the polynomial factor arising in the bound of
$\norm{\mathcal{F}}_{H_{+,a^+},\,C^{r_1}}$ is absorbed by
$\eta_{\frac{\epsilon}{2}}(\mathcal{S}, a^+)^{\frac{1}{2}\max_{1\leq i,j\leq 3}\abs{\tau(z_i)-\tau(z_j)}}$.

\smallskip
$\textbf{(ii)}$: If $\tau(z_{2})-\tau(z_{3})\leq\frac{1}{2}(\tau(z_1)-\tau(z_3))$. Then
\begin{align}\label{for:345}
 \tau(z_1)-\tau(z_2)&=\tau(z_3)-\tau(z_2)+(\tau(z_1)-\tau(z_3))\notag\\
 &>-\frac{1}{2}(\tau(z_1)-\tau(z_3))+(\tau(z_1)-\tau(z_3))=\frac{1}{2}(\tau(z_1)-\tau(z_3))>0.
\end{align}
Set
\begin{align*}
\mathcal{G}=\widetilde{f_2} \big(\pi((a^+)^{\tau(z_2)-\tau(z_3)})\widetilde{f_3}\big),\qquad\text{where }\widetilde{f_i}=\pi(\textbf{w}^{-1}w_{z_i-z_2})f_i.
\end{align*}
Then
\begin{align}\label{for:347}
&\mathfrak{m}\big(\pi(z_1)f_1,\,\pi(z_2)f_2,\, \pi(z_3)f_{3}\big)=\mathfrak{m}\big(\pi((a^+)^{\tau(z_1)-\tau(z_2)})\widetilde{f_1},\mathcal{G}\big).
\end{align}
It follows from Theorem \ref{th:7} that
\begin{align}\label{for:346}
 \big|\mathfrak{m}\big(&\pi((a^+)^{\tau(z_1)-\tau(z_2)})\widetilde{f_1},\mathcal{G}\big)\big|\overset{\text{(1)}}{\leq} C_\epsilon \eta_{\frac{\epsilon}{2}}(\mathcal{S}, a^+)^{\tau(z_1)-\tau(z_2)}\norm{\widetilde{f_1}}_{H_{+,a^+},\,r_1}\,\big\|\mathcal{G}\big\|_{H_{-0,a^+},\,r_2}\notag\\
 &\overset{\text{(2)}}{\leq} C_\epsilon \eta_{\frac{\epsilon}{2}}(\mathcal{S}, a^+)^{\tau(z_1)-\tau(z_2)}(1+|\tau(z_1-z_2)|)^{r_1\dim G}\norm{f_1}_{C^{r_1}}\,\big\|\mathcal{G}\big\|_{H_{-0,a^+},\,C^{r_2}}.
\end{align}
Here in $(1)$ we shrink $\epsilon$ to $\frac{\epsilon}{2}$ to absorb the polynomial growth factor  from the bound of $\mathcal{G}$; in $(2)$ we use \eqref{for:189} of Section \ref{sec:33}.

Since $W_{-0,a^+}$ is $\text{Ad}_{a^+}$-invariant with the maximal Lyapunov exponent of $\text{Ad}_{a^+}|_{W_{-0,a^+}}\leq0$ and $\tau(z_3)-\tau(z_{2})\leq 0$,
it follows from
\eqref{for:125} of Section \ref{sec:33} that
\begin{align*}
 \|\pi((a^+)^{\tau(z_2)-\tau(z_3)})\widetilde{f_3}\|_{H_{-0,a},C^{r_2}}\leq C\|\widetilde{f_3}\|_{H_{-0,a},C^{r_2}}.
\end{align*}
Thus, by Leibniz
(and Kato-Ponce for fractional orders), we have
\begin{align*}
\|\mathcal G\|_{H_{-0,a},C^{r_2}}&\leq C  \|\widetilde{f_2}\|_{H_{-0,a},C^{r_2}}\|\pi((a^+)^{\tau(z_2)-\tau(z_3)})\widetilde{f_3}\|_{H_{-0,a},C^{r_2}}\leq C_1\|\widetilde{f_2}\|_{C^{r_2}}\|\widetilde{f_3}\|_{C^{r_2}}\\
&\le C_2(1+|\tau(z_2)-\tau(z_3)|)^{r_2\dim G}\|f_2\|_{C^{r_2}}\|f_3\|_{C^{r_2}}\notag\\
&\overset{\text{(1)}}{\leq} C_3(1+|\tau(z_1)-\tau(z_3)|)^{r_2\dim G}\|f_2\|_{C^{r_2}}\|f_3\|_{C^{r_2}}\notag\\
&\overset{\text{(2)}}{\leq} C_4(1+|\tau(z_1)-\tau(z_2)|)^{r_2\dim G}\|f_2\|_{C^{r_2}}\|f_3\|_{C^{r_2}}.
\end{align*}
Here in $(1)$ we recall the assumption $\tau(z_{2})-\tau(z_{3})\leq\frac{1}{2}(\tau(z_1)-\tau(z_3))$; in $(2)$ we use \eqref{for:345}.

This together with \eqref{for:346} gives
\begin{align*}
 &\big|\mathfrak{m}\big(\pi((a^+)^{\tau(z_1)-\tau(z_2)})\widetilde{f_1},\mathcal{G}\big)\big|\\
 &\leq  C_\epsilon \eta_{\frac{\epsilon}{2}}(\mathcal{S}, a^+)^{\tau(z_1)-\tau(z_2)}(1+|\tau(z_1-z_2)|)^{(r_1+r_2)\dim G}\norm{f_1}_{C^{r_1}}\|f_2\|_{C^{r_2}}\|f_3\|_{C^{r_2}}\\
&\overset{\text{(1)}}{\leq} C_\epsilon \eta_{\varepsilon}(\mathcal{S}, a^+)^{\tau(z_1)-\tau(z_2)}\norm{f_1}_{C^{r_1}}\|f_2\|_{C^{r_2}}\|f_3\|_{C^{r_2}}\\
&\overset{\text{(2)}}{\leq} C_{\epsilon,1}\eta_{\frac{\epsilon}{2}}(\mathcal{S}, a^+)^{\frac{1}{2}\max_{1\leq i,j\leq 3}\abs{\tau(z_i)-\tau(z_j)}}\Pi_{i=1}^3\norm{f_i}_{C^{r_2}}.
\end{align*}
Here in $(1)$ the polynomial factor is absorbed by $\eta_{\frac{\epsilon}{2}}(\mathcal{S}, a^+)^{\tau(z_1)-\tau(z_2)}$; in $(2)$ we use \eqref{for:345} and recall $r_2\geq r_1$.
Hence, we proved \eqref{for:208}.

\smallskip
\emph{Case IV: } Suppose $m\in\NN$. For any $n\geq2$ we can choose non-zero $f_1,\,f_2\in C_c^\infty(\mathcal{X})$ with $\int_{\mathcal{X}} f_1 \, d\varrho=0$, $i=1,2$ and  $c:=\int_{\mathcal{X}} f_2^n \, d\varrho\neq0$.
Then we have
\begin{align*}
 &\Big|\int_{\mathcal{X}} \big(\pi(m)f_1\big)^2 \big(\pi(2m)f_2\big)^n \,d\varrho \Big|=\Big|\int_{\mathcal{X}} \big(\pi(m)f_1\big)^2 \big(\pi(2m)f_2^n\big) \,d\varrho \Big|\notag\\
 &=\Big|\int_{\mathcal{X}} \big(\pi(m)f_1\big)^2 \big(\pi(2m)(f_2^n-c)\big) \,d\varrho +c\int_{\mathcal{X}} f_1^2 \, d\varrho\Big|\notag\\
 &\overset{\text{(1)}}{\geq} |c|\int_{\mathcal{X}} f_1^2 \, d\varrho\,-C_\epsilon \big( \eta_\epsilon(\mathcal{S}, s_{m})\big)^{\frac{1}{2}}\|f_1\|_{C^{r_{2}}}^2\|f_2^n-c\|_{C^{r_{2}}}.
\end{align*}
Here in $(1)$ we use \emph{Case III}.

On the other hand, if this quantity were uniformly bounded in terms of $\min_{1\leq i,j\le n} \eta_\epsilon(\mathcal{S}, s_{z_i-z_j})$, i.e.,  if there were $0<\gamma<1$ with
\[
 \Big|\!\int_{\mathcal{X}} \big(\pi(m)f_1\big)^2 \big(\pi(2m)f_2\big)^n \,d\varrho \Big|
 \le C_{f_1,f_2}\,\big(\eta_\epsilon(\mathcal{S}, s_{m})\big)^\gamma\qquad\forall \,m\in\mathbb N,
\]
then letting $m\to\infty$, we get a contraction.  The right-hand side tends to $0$ while $|c|\int f_1^2\,d\varrho>0$. This proves \eqref{for:203}.

\section{Higher order exponential mixing: semisimple setting}
\subsection{Main results} We list  the notation that will appear in the following theorems:
\begin{enumerate}

  \item Strong spectral gap, $(\pi, L^2(\mathcal X,\varrho))$:  see \eqref{for:120} of Section \ref{sec:33}.

  \item Maximal strongly orthogonal system $\mathcal{S}$: see Section \ref{for:220}.
  \item $p_{\epsilon}(\mathcal{S})$, $\zeta_{\epsilon}(\mathcal{S})$, $\eta_\epsilon(\mathcal{S}, a)$, $a\in A$: see Section \ref{for:222}.

  \item Essentially semisimple, $d(s_{z_1},s_{z_2})$, $s_z$: see  \eqref{for:173} of Section \ref{sec:33}.

\end{enumerate}

 Let $\mathcal Z\subset G$ be a closed abelian subgroup of rank $\ell$, $\ell\geq2$ and $\mathcal Z$ is essentially semisimple.
Then there exist a maximal split torus $A\le G$
such that for every $z\in\mathcal Z$ the semisimple part satisfies $s_z=\mathfrak{a}^{\tau(z)}\in A$ (see  \eqref{for:173} of Section \ref{sec:33}).

Let $\mathcal S$ be a maximal strongly orthogonal system for $A$ and recall that $(\pi, L_0^2(\mathcal X,\varrho))$ has a strong spectral gap.

\begin{theorem}\label{th:3} For any $f_1,\cdots, f_n\in C_c^{\infty}(\mathcal{X})$ and any $z_1,\cdots,z_n\in\mathcal{Z}$, we have:
 \begin{align}\label{for:219}
 &\Big|\int_{\mathcal{X}} \Pi_{i=1}^{n} \pi(z_i)f_i \,d\varrho - \Pi_{i=1}^{n} \int_{\mathcal{X}} f_i \,d\varrho\Big|\notag\\
 &\leq C_{n,\epsilon}\,\max_{1\le i\neq j\le n}\eta_\epsilon (\mathcal{S}, s_{z_i-z_j})^{\frac{1}{(n-1)|\mathcal{S}|}}\,
\Pi_{i=1}^n\norm{f_i}_{C^{\zeta_{\epsilon}(\mathcal{S})}}.
\end{align}
Moreover, if   $\int_{\mathcal{X}} f_i \, d\varrho=0$, $1\leq i\leq n$, then
\begin{enumerate}
  \item\label{for:247} If $n=3$ and $\int_{\mathcal{X}} f_i \,d\varrho=0$, $1\leq i\leq3$,  then
\begin{align*}
 &\Big|\int_{\mathcal{X}} \Pi_{i=1}^{3} \pi(z_i)f_i \,d\varrho \Big|\le C_{\epsilon}\,\eta_\epsilon (\mathcal{S}, s_{z_{i_1}-z_{i_2}})^{\frac{1}{2|\mathcal{S}|}}\,
\Pi_{i=1}^3\norm{f_i}_{C^{\zeta_{\epsilon}(\mathcal{S})}}.
\end{align*}
where $1\leq i_1\neq i_2\leq 3$ satisfying
\begin{align*}
   d(s_{z_{i_1}},s_{z_{i_2}})=\max_{1\leq i,j\leq3}d(s_{z_{i}},s_{z_{j}}).
  \end{align*}
  \item\label{for:150} In particular, if $\mathcal Z\subseteq A$, then
\begin{align*}
 &\Big|\int_{\mathcal{X}} \Pi_{i=1}^{3} \pi(z_i)f_i \,d\varrho \Big|\le C_{\epsilon}\,\min_{1\le i,j\le 3}\eta_\epsilon (\mathcal{S}, s_{z_{i}-z_{j}})^{\frac{1}{2|\mathcal{S}|}}\,
\Pi_{i=1}^3\norm{f_i}_{C^{\zeta_{\epsilon}(\mathcal{S})}}.
\end{align*}
  \item\label{for:191} If $n\geq 4$, even assuming $\mathcal Z\subseteq A$,,  no uniform bound in terms of $\min_{1\le i,j\le n} \eta_\epsilon(\mathcal{S}, s_{z_i-z_j})$ can hold in general.
\end{enumerate}

\end{theorem}
\begin{remark}  When $n=3$ in the split case $\mathcal Z\subseteq A$, we use the same two-block decomposition as in the proof of Theorem~\ref{th:8}. Here $w_z=e$ for all $z$, so the transfer estimate \eqref{for:130} has no polynomial loss (cf.\ \eqref{for:130}–\eqref{for:133}). Consequently, the two-block argument together with the gap split yields the exponent $1/(2|\mathcal S|)$ with the minimal value of $\eta_\epsilon(\mathcal S,s_{z_i-z_j})$ (equivalently, the maximal gap); see \eqref{for:150}.
\end{remark}

\subsection{Proof of Theorem \ref{th:3}}

We recall  the following notation that will be used in the proof.
\begin{enumerate}
\item Maximal strongly orthogonal system $\mathcal{S}$,  $u(\mathcal{S})$, $S(\mathcal{S})$: see Sections \ref{for:220} and \ref{for:221}.

  \item $d(s_{z_1},s_{z_2})$, $\mathfrak{a}$, $\tau$, $s_z$ and $w_z$: see  \eqref{for:173} of Section \ref{sec:33}.

\item Let $\mathcal S=\{\beta_1,\dots,\beta_l\}\subset\Phi^+$, $l=|\mathcal S|$ be a maximal strongly orthogonal system for $A$. Set $\gamma_i:=\gamma_{\beta_i}$ (see Section \ref{for:222}).

\end{enumerate}

\subsubsection{Proof strategy} The basic idea is to bootstrap \emph{order-$2$ mixing} (rank-one estimates for two-point
correlations with partial Sobolev norms; see Theorem~\ref{th:7}) to \emph{higher-order}
mixing via a two-block decomposition and an induction on the number of factors.  The
partial norms (along the stable/unstable/neutral directions) allow us to place the
``expanding derivative cost" on one block and the ``contracting derivative cost" on
the other, so that the order-$2$ estimate yields an exponential gain that is uniform
with respect to the remaining factors (and their translates).

\subsubsection{Proof of \eqref{for:219} of Theorem \ref{th:3}}\label{sec:27}

 For any $z_i\in \mathcal{Z}$ and $f_i\in C_c^\infty(\mathcal{X})$, $1\leq i\leq n$, set
\begin{align*}
  \mathfrak{m}\big(\pi(z_1)f_1,\pi(z_2)f_2,\cdots,\pi(z_n)f_n \big)=\int_{\mathcal{X}} \Pi_{i=1}^{n} \pi(z_i)f_i \,d\varrho.
\end{align*}
We prove by induction on $n$. The case of $n=2$ follows from Theorem \ref{th:8} directly.

Suppose the argument holds for $n=k$; that is, for
\[
z_1,\dots, z_k\in \mathcal{Z}\quad\text{and}\quad f_1,\dots, f_k\in C_c^{\infty}(\mathcal{X}),
\]
we have
\begin{align}\label{for:235}
 &\Big|\int_{\mathcal{X}} \Pi_{i=1}^{k} \pi(z_i)f_i \,d\varrho - \Pi_{i=1}^{k} \int_{\mathcal{X}} f_i \,d\varrho\Big|\notag\\
&\le C_{k,\epsilon}\,\max_{1\le i\neq j\le k}\eta_\epsilon (\mathcal{S}, s_{z_i-z_j})^{\frac{1}{(k-1)|\mathcal{S}|}}\,
\Pi_{i=1}^k\norm{f_i}_{C^{\zeta_{\epsilon}(\mathcal{S})}}.
\end{align}
\textbf{Inductive Step:} Suppose
\begin{align*}
z_1,\cdots,z_{k+1}\in \mathcal{Z}\quad\text{ and }\quad f_1,\cdots, f_{k+1}\in C_c^{\infty}(\mathcal{X}).
\end{align*}
There exists $t_i\in \RR^m$, $1\leq i\leq k+1$ such that
\begin{align*}
 s_{z_i}=\mathfrak{a}^{t_i},\qquad 1\leq i\leq k+1.
\end{align*}
Without loss of generality, assume
\begin{align}\label{for:136}
 \norm{t_1-t_{k+1}}=\max_{1\leq i,j\leq k+1} \norm{t_i-t_{j}}.
\end{align}
There exists $1\leq i_0\leq l$ such that
\begin{align*}
 \eta_{\frac{\epsilon}{2}} (\mathcal{S}, \mathfrak{a}^{t_1-t_{k+1}})^{\frac{1}{l}}\geq \min\{\beta_{i_0}(\mathfrak{a}^{t_{1}-t_{k+1}})^{-(\gamma_{i_0}-\frac{\epsilon}{2})},\,\beta_{i_0}(\mathfrak{a}^{t_{k+1}-t_{1}})^{-(\gamma_{i_0}-\frac{\epsilon}{2})}\}.
\end{align*}
Without loss of generality, assume that
\begin{align}\label{for:225}
\eta_{\frac{\epsilon}{2}} (\mathcal{S}, \mathfrak{a}^{t_1-t_{k+1}})^{\frac{1}{l}}\geq \beta_{i_0}(\mathfrak{a}^{t_{k+1}-t_{1}})^{-(\gamma_{i_0}-\frac{\epsilon}{2})}.
\end{align}
\emph{Note}. The above inequality implies that $\beta_{i_0}(\mathfrak{a}^{t_{k+1}-t_{1}})>1$.   Set
\begin{enumerate}
  \item [$(\mathcal{T}_1)$]\namedlabel{for:230}{$\mathcal{T}_1$} $D_1=\{1\leq i\leq k: \beta_{i_0}(\mathfrak{a}^{t_{k+1}-t_{i}})>1\}$. Then $1\in D_1$ and $|D_1|\leq k$.

  \smallskip
  \item [$(\mathcal{T}_2)$]\namedlabel{for:231}{$\mathcal{T}_2$} $D_2=\{1\leq i\leq k: \beta_{i_0}(\mathfrak{a}^{t_{k+1}-t_{i}})\leq 1\}.$
  \smallskip
  \item [$(\mathcal{T}_3)$]\namedlabel{for:226}{$\mathcal{T}_3$} $c=\max_{i\in D_1}\{\beta_{i_0}(\mathfrak{a}^{t_{k+1}-t_{i}})\}$. Clearly, $c\geq \beta_{i_0}(\mathfrak{a}^{t_{k+1}-t_{1}})>1$.
\end{enumerate}
Then $D_1$ and $D_2$ form a partition of the set \(\{1,\dots, k\}\). The $k+1$ constants $c^{\frac{j}{k}}$, $0\leq j\leq k$ divide the interval $[1,c]$ into $k$ subintervals $[  c^{\frac{j}{k}}, c^{\frac{j+1}{k}}]$, $0\leq j\leq k-1$.
All values $\beta_{i_0}(\mathfrak{a}^{t_{k+1}-t_{i}})$, $i\in D_1$,  except those equal to $c$, lie in $(1,c)$.  Since there are at most $|D_1|-1\leq k-1$ such values, there exists some $0\leq j\leq k-1$, such that none of these numbers fall inside
 the interval $(c^{\frac{j}{k}}, c^{\frac{j+1}{k}})$. Consequently, for all $i\in D_1$, we have either
\begin{align*}
 \beta_{i_0}(\mathfrak{a}^{t_{k+1}-t_{i}})-c^{\frac{j}{k}}\leq 0\quad\text{or}\quad \beta_{i_0}(\mathfrak{a}^{t_{k+1}-t_{i}})-c^{\frac{j+1}{k}}\geq 0.
\end{align*}
Set
\begin{enumerate}
  \item [$(\mathcal{T}_4)$] $D_{1,1}=\{i\in D_1: \beta_{i_0}(\mathfrak{a}^{t_{1+k}-t_{i}})\leq c^{\frac{j}{k}}\}$. We note that $D_{1,1}$ may be an empty set.

  \smallskip
  \item[$(\mathcal{T}_5)$]  $D_{1,2}=\{i\in D_1: \beta_{i_0}(\mathfrak{a}^{t_{1+k}-t_{i}})\geq c^{\frac{j+1}{k}}\}$. Then $D_{1,2}\neq\emptyset$.

  \smallskip

  \item [$(\mathcal{T}_6)$] \namedlabel{for:246}{$\mathcal{T}_6$} $D_{1,1}$ and $D_{1,2}$ form a partition of $D_1$ and $1\in D_{1,2}$. $D_{1,1}$, $D_{1,2}$ and $D_2$ form a partition of the set $\{1,\dots, k\}$.

  \smallskip
  \item [$(\mathcal{T}_7)$]  \namedlabel{for:223}{$\mathcal{T}_7$} If $D_{1,1}=\emptyset$, let $\mathfrak{q}_1=k+1$. If $D_{1,1}\neq\emptyset$, choose index  $\mathfrak{q}_1 \in D_{1,1}$ such that
  \begin{align*}
   \beta_{i_0}(\mathfrak{a}^{t_{k+1}-t_{ \mathfrak{q}_1}})=\max_{i\in D_{1,1}} \beta_{i_0}(\mathfrak{a}^{t_{1+k}-t_{i}})\leq c^{\frac{j}{k}}.
  \end{align*}

  \item [$(\mathcal{T}_8)$] \namedlabel{for:224}{$\mathcal{T}_8$}  Choose index $\mathfrak{q}_2 \in D_{1,2}$ such that
  \begin{align*}
   \beta_{i_0}(\mathfrak{a}^{t_{k+1}-t_{ \mathfrak{q}_2}})=\min_{i\in D_{1,2}} \beta_{i_0}(\mathfrak{a}^{t_{1+k}-t_{i}})\geq c^{\frac{j+1}{k}}.
  \end{align*}

\end{enumerate}
Let $\mathcal{S}_1=\{\beta_{i_0}\}$. Then:
\begin{enumerate}
  \item [$(\mathcal{T}_9)$] \namedlabel{for:227}{$\mathcal{T}_9$} $\mathcal{S}_1$ is a  strongly orthogonal system of $A$ (see Section  \ref{for:220}).

  \smallskip
  \item [$(\mathcal{T}_{10})$]\namedlabel{for:228}{$\mathcal{T}_{10}$}   $\text{Lie}(u(\mathcal{S}_1))=U_{\beta_{i_0}}$ and $\text{Lie}(S(\mathcal{S}_1))$ is spanned by $U_{\beta_{i_0}^{-1}}$ and $X_{\beta_{i_0}}$ (see Section \ref{for:221}).

      \smallskip
  \item [$(\mathcal{T}_{11})$] \namedlabel{for:244}{$\mathcal{T}_{11}$}  For any $a\in A$ (see Section \ref{for:222})
  \begin{align*}
   \zeta_{\epsilon}(\mathcal{S}_1)= \gamma_{i_0}+2+\epsilon\quad\text{and}\quad \eta_\epsilon(\mathcal{S}_1, a)=\beta_{i_0}(a^+)^{-(\gamma_{i_0}-\epsilon)}. \end{align*}

\end{enumerate}
Then we have
\begin{align*}
 \int_{\mathcal{X}} \Pi_{i=1}^{k+1} \pi(z_i)f_i \,d\varrho=\int_{\mathcal{X}} \Pi_{i=1}^{k+1} \pi(z_i-z_{\mathfrak{q}_2})f_i \,d\varrho=\mathfrak{m}\big(\pi(\mathfrak{a}^{t_{ \mathfrak{q}_1}-t_{ \mathfrak{q}_2}})\mathcal{F}_1,\,\mathcal{F}_2\big),
\end{align*}
where
\begin{gather}\label{for:143}
 \mathcal{F}_1=\Pi_{i\in D_{1,1}\cup D_2\cup\{k+1\} } \pi(\mathfrak{a}^{t_{i}-t_{ \mathfrak{q}_1}})\tilde{f}_i,\quad
 \mathcal{F}_2= \Pi_{i\in D_{1,2}}\, \pi(\mathfrak{a}^{t_i-t_{ \mathfrak{q}_2}})\tilde{f}_i\\
 \tilde{f}_i=f_i\circ (w_{z_iz_{\mathfrak{q}_2}^{-1}}).\notag
\end{gather}
From \eqref{for:127} and \eqref{for:125} of Section \ref{sec:33},  we have: for any $t\geq0$, $1\leq i\leq k+1$
\begin{align}\label{for:229}
 \norm{\widetilde{f_i}}_{C^t}&\leq C_t(1+\|t_i-t_{\mathfrak{q}_2}\|)^{t\dim G}\norm{f_i}_{C^t}\overset{\text{(1)}}{\leq} C_{t,1}(\|t_1-t_{1+k}\|+1)^{t\dim G}\norm{f_i}_{C^{t}}.
\end{align}
Here in $(1)$ we recall \eqref{for:136}. Then we have
\begin{align}\label{for:145}
 \Big|\int_{\mathcal{X}} \Pi_{i=1}^{k+1} \pi(z_i)f_i \,d\varrho - \Pi_{i=1}^{k+1} \int_{\mathcal{X}} f_i \,d\varrho\Big|\leq I_1+I_2,
\end{align}
where
\begin{align*}
 I_1&=\Big|\mathfrak{m}\big(\pi(\mathfrak{a}^{t_{ \mathfrak{q}_1}-t_{ \mathfrak{q}_2}})\mathcal{F}_1,\,\mathcal{F}_2\big)-\mathfrak{m}(\mathcal{F}_1)\mathfrak{m}(\mathcal{F}_2)\Big|,\\
 I_2&=\Big|\mathfrak{m}(\mathcal{F}_1)\mathfrak{m}(\mathcal{F}_2)-\Pi_{i=1}^{k+1} \int_{\mathcal{X}} f_i \,d\varrho\Big|.
\end{align*}
We now estimate $I_1$ and $I_2$ separately.

\smallskip
\textbf{\emph{Estimate for $I_1$}}: From \eqref{for:227}, it follows from  Theorem \ref{th:7} for $\mathcal{S}=\mathcal{S}_1$ that
\begin{align}\label{for:232}
 I_1&\leq C_\epsilon \eta_{\frac{\epsilon}{2}}(\mathcal{S}_1, \mathfrak{a}^{t_{ \mathfrak{q}_1}-t_{ \mathfrak{q}_2}})\norm{\mathcal{F}_1}_{u(\mathcal{S}_1),\,C^{\zeta_\epsilon(\mathcal{S}_1)}}\,\big\|\mathcal{F}_2\big\|_{S(\mathcal{S}_1),\,C^{\zeta_\epsilon(\mathcal{S}_1)}}.
 \end{align}
 We point out that we shrink $\epsilon$ to $\frac{\epsilon}{2}$ to absorb polynomial factors when estimating Sobolev norms of  $\mathcal{F}_1$ and $\mathcal{F}_2$.
Next, we estimate $\eta_{\frac{\epsilon}{2}}(\mathcal{S}_1, \mathfrak{a}^{t_{ \mathfrak{q}_1}-t_{ \mathfrak{q}_2}})$,  $\mathcal{F}_1$ and $\mathcal{F}_2$ respectively.

\emph{Estimate for $\mathcal{F}_1$}: we have
\begin{align}\label{for:233}
 \norm{\mathcal{F}_1}&_{u(\mathcal{S}_1),C^{\zeta_\epsilon(\mathcal{S}_1)}}\leq C_k\,\Pi_{i\in D_{1,1}\cup D_2\cup \{k+1\}}\big\| \pi(\mathfrak{a}^{t_i-t_{ \mathfrak{q}_1}})\tilde{f}_i\big\|_{u(\mathcal{S}_1),C^{\zeta_\epsilon(\mathcal{S}_1)}}\notag\\
 &\overset{(x)}{\leq} C_k\,\Pi_{i\in D_{1,1}\cup D_2\cup \{k+1\}}\big\| \tilde{f}_i\big\|_{u(\mathcal{S}_1),C^{\zeta_\epsilon(\mathcal{S}_1)}}\notag\\
 &\overset{(y)}{\leq} C_{k,1}\, \Pi_{i\in D_{1,1}\cup D_2\cup \{k+1\}}\big\{(\|t_1-t_{1+k}\|+1)^{\zeta_\epsilon(\mathcal{S}_1)\dim G }\big\| f_i\big\|_{C^{\zeta_\epsilon(\mathcal{S}_1)}}\big\}
\end{align}
We explain steps:
\begin{itemize}
  \item In $(x)$ if $D_{1,1}=\emptyset$, then for any $i\in D_2\cup \{k+1\}$, we have
  \begin{align*}
   \beta_{i_0}(\mathfrak{a}^{t_{i}-t_{ \mathfrak{q}_1}})&=\beta_{i_0}(\mathfrak{a}^{-(t_{k+1}-t_{i})})\overset{\text{(1)}}{\geq}1.
 \end{align*}
  Here in $(1)$ we use  \eqref{for:231}. Next, we suppose $D_{1,1}\neq\emptyset$.

  For any $i\in D_{1,1}$ we have
  \begin{align*}
   \beta_{i_0}(\mathfrak{a}^{t_{i}-t_{ \mathfrak{q}_1}})&=\beta_{i_0}(\mathfrak{a}^{-(t_{k+1}-t_{i})})\beta_{i_0}(\mathfrak{a}^{t_{k+1}-t_{ \mathfrak{q}_1}})\overset{(*)}{=}\beta_{i_0}(\mathfrak{a}^{-(t_{k+1}-t_{i})})\big( \max_{r\in D_{1,1}} \beta_{i_0}(\mathfrak{a}^{t_{1+k}-t_{r}})\big)\\
   &\geq \beta_{i_0}(\mathfrak{a}^{-(t_{k+1}-t_{i})})\beta_{i_0}(\mathfrak{a}^{t_{k+1}-t_{ i}})=1.
 \end{align*}
  Here in $(*)$ we use \eqref{for:223}.

  For any $i\in D_2\cup \{k+1\}$, we have
  \begin{align*}
   \beta_{i_0}(\mathfrak{a}^{t_{i}-t_{ \mathfrak{q}_1}})&=\beta_{i_0}(\mathfrak{a}^{-(t_{k+1}-t_{i})})\beta_{i_0}(\mathfrak{a}^{t_{k+1}-t_{ \mathfrak{q}_1}})\overset{\text{(1)}}{\geq} \beta_{i_0}(\mathfrak{a}^{t_{k+1}-t_{\mathfrak{q}_1}})\overset{\text{(2)}}{\geq} 1.
 \end{align*}
  Here in $(1)$ we use  \eqref{for:231}; in $(2)$ we recall that $\mathfrak{q}_1 \in D_{1,1}\subseteq D_1$ and then use \eqref{for:230}.

  This shows that for any $i\in D_{1,1}\cup D_2\cup \{k+1\}$ we have
 \begin{align*}
  \beta_{i_0}(\mathfrak{a}^{t_{i}-t_{ \mathfrak{q}_1}})\geq 1.
 \end{align*}
  Then for $U_{\beta_{i_0}}$, which spans $\text{Lie}(u(\mathcal{S}_1))$ (see \eqref{for:228}), we have
  \begin{align*}
 \|\text{Ad}_{\mathfrak{a}^{-(t_{i}-t_{ \mathfrak{q}_1})}}U_{\beta_{i_0}}\|=\beta_{i_0}(\mathfrak{a}^{-(t_{i}-t_{ \mathfrak{q}_1})})\leq 1.
\end{align*}
  Then $(x)$ follows from \eqref{for:125} of Section \ref{sec:33}.
  \smallskip
  \item In $(y)$ we use \eqref{for:229}.
\end{itemize}
\emph{Estimate for $\mathcal{F}_2$}: we have
\begin{align}\label{for:234}
 \norm{\mathcal{F}_2}&_{S(\mathcal{S}_1),\,C^{\zeta_\epsilon(\mathcal{S}_1)}}\leq C_k\,\Pi_{i\in D_{1,2}}\big\| \pi(\mathfrak{a}^{t_i-t_{ \mathfrak{q}_2}})\tilde{f}_i\big\|_{S(\mathcal{S}_1),\,C^{\zeta_\epsilon(\mathcal{S}_1)}}\notag\\
 &\overset{(y_1)}{\leq} C_k\,\Pi_{i\in D_{1,2}}\big\| \tilde{f}_i\big\|_{S(\mathcal{S}_1),\,C^{\zeta_\epsilon(\mathcal{S}_1)}}\notag\\
 &\overset{(y)}{\leq} C_{k,1}\Pi_{i\in D_{1,2}}\big\{(\|t_1-t_{1+k}\|+1)^{\zeta_\epsilon(\mathcal{S}_1)\dim (G)}\big\|f_i\big\|_{S(\mathcal{S}_1),\,C^{\zeta_\epsilon(\mathcal{S}_1)}}\big\}.
\end{align}
We explain steps.
\begin{itemize}

  \item In $(y_1)$ for any $i\in D_{1,2}$ we have
 \begin{align*}
   \beta_{i_0}(\mathfrak{a}^{t_{i}-t_{ \mathfrak{q}_2}})&=\beta_{i_0}(\mathfrak{a}^{t_{k+1}-t_{\mathfrak{q}_2}})\beta_{i_0}(\mathfrak{a}^{t_{ i}-t_{k+1}})\overset{(*)}{=}\min_{r\in D_{1,2}} \beta_{i_0}(\mathfrak{a}^{t_{1+k}-t_{r}})\big(\beta_{i_0}(\mathfrak{a}^{t_{k+1}-t_{ i}})\big)^{-1}\\
   &\leq \beta_{i_0}(\mathfrak{a}^{t_{1+k}-t_{i}})\big(\beta_{i_0}(\mathfrak{a}^{t_{k+1}-t_{ i}})\big)^{-1}=1.
 \end{align*}
 Here in $(*)$ we use \eqref{for:224}.

Then for  $U_{\beta_{i_0}^{-1}}$ and $X_{\beta_{i_0}}$, which spans $\text{Lie}(S(\mathcal{S}_1))$ (see \eqref{for:228}), we have
  \begin{align*}
 \|\text{Ad}_{\mathfrak{a}^{-(t_{i}-t_{ \mathfrak{q}_2})}}U_{\beta_{i_0}^{-1}}\|&=\beta_{i_0}(\mathfrak{a}^{t_{i}-t_{ \mathfrak{q}_2}})\leq 1\quad\text{and}\quad
  \|\text{Ad}_{\mathfrak{a}^{-(t_{i}-t_{ \mathfrak{q}_2})}}X_{\beta_{i_0}}\|=1.
\end{align*}
  Then $(y_1)$ follows from \eqref{for:125} of Section \ref{sec:33}.
  \smallskip
\item In $(y)$ we use \eqref{for:229}.
\end{itemize}
\emph{Estimate for $\eta_{\frac{\epsilon}{2}}(\mathcal{S}_1, \mathfrak{a}^{t_{ \mathfrak{q}_1}-t_{ \mathfrak{q}_2}})$}: if $D_{1,1}=\emptyset$, then
\begin{align}\label{for:385}
   \beta_{i_0}(\mathfrak{a}^{t_{ \mathfrak{q}_1}-t_{ \mathfrak{q}_2}})=\beta_{i_0}(\mathfrak{a}^{t_{ k+1}-t_{ \mathfrak{q}_2}})\overset{\text{(*)}}{\geq }c^{\frac{j+1}{k}}\geq c^{\frac{1}{k}}>1.
 \end{align}
 Here in $(*)$ we use \eqref{for:224}. If $D_{1,1}\neq\emptyset$, then
 \begin{align}\label{for:245}
   \beta_{i_0}(\mathfrak{a}^{t_{ \mathfrak{q}_1}-t_{ \mathfrak{q}_2}})=\beta_{i_0}(\mathfrak{a}^{t_{ k+1}-t_{ \mathfrak{q}_2}})\beta_{i_0}(\mathfrak{a}^{-(t_{k+1}-t_{ \mathfrak{q}_1})})\overset{\text{(*)}}{\geq }c^{\frac{j+1}{k}}c^{-\frac{j}{k}}=c^{\frac{1}{k}}\overset{\text{(**)}}{> } 1.
 \end{align}
 Here in $(*)$ we use \eqref{for:223} and \eqref{for:224}; in $(**)$ we use \eqref{for:226}.

Then  \eqref{for:385} and \eqref{for:245} imply that
 \begin{align}\label{for:141}
  \eta_{\frac{\epsilon}{2}}(\mathcal{S}_1, \mathfrak{a}^{t_{ \mathfrak{q}_1}-t_{ \mathfrak{q}_2}})&\overset{\text{(*)}}{=}\beta_{i_0}(\mathfrak{a}^{t_{ \mathfrak{q}_1}-t_{ \mathfrak{q}_2}})^{-(\gamma_{i_0}-\frac{\epsilon}{2})}\leq (c^{\frac{1}{k}})^{-(\gamma_{i_0}-\frac{\epsilon}{2})}\notag\\
  &\overset{\text{(**)}}{\leq}  \big(\beta_{i_0}(\mathfrak{a}^{t_{1+k}-t_{1}})^{\frac{1}{k}}\big)^{-(\gamma_{i_0}-\frac{\epsilon}{2})}
  \overset{(***)}{\leq}  \eta_{\frac{\epsilon}{2}} (\mathcal{S}, \mathfrak{a}^{t_1-t_{k+1}})^{\frac{1}{lk}}.
 \end{align}
Here in $(*)$ we recall \eqref{for:244} and use the fact that $\beta_{i_0}(\mathfrak{a}^{t_{ \mathfrak{q}_1}-t_{ \mathfrak{q}_2}})>1$ (see \eqref{for:245}); in
$(**)$ we use \eqref{for:226};  in $(***)$ we use \eqref{for:225}.

It follows from \eqref{for:232}, \eqref{for:233}, \eqref{for:234} and \eqref{for:141} that
\begin{align}
 I_1&\leq  C_{k,\epsilon} \eta_{\frac{\epsilon}{2}} (\mathcal{S}, \mathfrak{a}^{t_1-t_{k+1}})^{\frac{1}{lk}}(\|t_1-t_{1+k}\|+1)^{\zeta_\epsilon(\mathcal{S}_1)\dim (G)k}\Pi_{1\leq i\leq k+1}\big\| f_i\big\|_{C^{\zeta_\epsilon(\mathcal{S}_1)}}\notag\\
 &\leq C_{k,\epsilon,1} \,\eta_{\epsilon} (\mathcal{S}, \mathfrak{a}^{t_1-t_{k+1}})^{\frac{1}{lk}}\,\Pi_{1\leq i\leq k+1}\big\| f_i\big\|_{C^{\zeta_\epsilon(\mathcal{S}_1)}}\label{for:146}\\
  &\leq C_{k,\epsilon,1} \,\max_{1\leq i\neq j\leq k+1}\eta_{\epsilon} (\mathcal{S}, s_{z_i-z_j})^{\frac{1}{lk}}\,\Pi_{1\leq i\leq k+1}\big\| f_i\big\|_{C^{\zeta_\epsilon(\mathcal{S}_1)}}.\label{for:238}
\end{align}
\textbf{\emph{Estimate for $I_2$}}: Recalling \eqref{for:143}, we have
\begin{align*}
 I_2&\leq \big|\mathfrak{m}(\mathcal{F}_1)-\Pi_{i\in D_{1,1}\cup D_2\cup \{k+1\}}\mathfrak{m}( f_i)\big|\, |\mathfrak{m}(\mathcal{F}_2)|\\
 &+\big|\Pi_{i\in D_{1,1}\cup D_2\cup \{k+1\}}\mathfrak{m}( f_i)\big|\, \big|\mathfrak{m}(\mathcal{F}_2)- \Pi_{i\in D_{1,2}}\mathfrak{m}( f_i)\big|\\
 &\leq \big|\mathfrak{m}(\mathcal{F}_1)-\Pi_{i\in D_{1,1}\cup D_2\cup \{k+1\}}\mathfrak{m}( f_i)\big| \,\Pi_{i\in D_{1,2}} \norm{f_i}_{C^0}\\
 &+ \big|\mathfrak{m}(\mathcal{F}_2)- \Pi_{i\in D_{1,2}}\mathfrak{m}( f_i)\big|\, \Pi_{i\in D_{1,1}\cup D_2\cup \{k+1\}}\norm{f_i}_{C^0}.
\end{align*}
By inductive assumption \eqref{for:235}, we have
\begin{align*}
 &\big|\mathfrak{m}(\mathcal{F}_1)-\Pi_{i\in D_{1,1}\cup D_2\cup \{k+1\}}\mathfrak{m}( f_i)\big|\notag\\
 &\leq  C_{d,\epsilon}\,\max_{i\neq j\in D_{1,1}\cup D_2\cup \{k+1\}}\eta_\epsilon (\mathcal{S}, s_{z_i-z_j})^{\frac{1}{(d-1)|\mathcal{S}|}}\,
\Pi_{i\in D_{1,1}\cup D_2\cup \{k+1\}}\norm{f_i}_{C^{\zeta_\epsilon(\mathcal{S}_1)}}\notag\\
&\leq C_{k, \epsilon}\,\max_{1\leq i\neq j\leq k+1}\eta_{\epsilon} (\mathcal{S}, s_{z_i-z_j})^{\frac{1}{lk}}\,
\Pi_{i\in D_{1,1}\cup D_2\cup \{k+1\}}\norm{f_i}_{C^{\zeta_\epsilon(\mathcal{S}_1)}}
\end{align*}
where $d=|D_{1,1}\cup D_2\cup \{k+1\}|$, and
\begin{align*}
 &\Big|\mathfrak{m}(\mathcal{F}_2)-\Pi_{i\in D_{1,2}}\mathfrak{m}( f_i)\Big|\notag\\
 &\leq  C_{|D_{1,2}|,\epsilon}\,\max_{i\neq j\in D_{1,2}}\eta_\epsilon (\mathcal{S}, s_{z_i-z_j})^{\frac{1}{(|D_{1,2}|-1)|\mathcal{S}|}}
 \Pi_{i\in D_{1,2}}\norm{f_i}_{C^{\zeta_\epsilon(\mathcal{S}_1)}}\notag\\
&\leq C_{k,\epsilon}\,\max_{1\leq i\neq j\leq k+1}\eta_{\epsilon} (\mathcal{S}, s_{z_i-z_j})^{\frac{1}{lk}}\,
\Pi_{i\in D_{1,2}}\norm{f_i}_{C^{\zeta_\epsilon(\mathcal{S}_1)}}.
\end{align*}
Consequently, since $D_{1,1}$, $D_{1,2}$, $D_2$ and $\{k+1\}$ form a partition of the set $\{1,\dots, k+1\}$. (see \eqref{for:246}), we have
\begin{align}\label{for:241}
 I_2&\leq C_{k,\epsilon}\,\max_{1\leq i\neq j\leq k+1}\eta_{\epsilon} (\mathcal{S}, s_{z_i-z_j})^{\frac{1}{lk}}\,
\Pi_{1\leq i\leq k+1}\norm{f_i}_{C^{\zeta_\epsilon(\mathcal{S}_1)}}.
\end{align}
It follows from \eqref{for:145}, \eqref{for:238} and \eqref{for:241} that
\begin{align*}
 &\Big|\int_{\mathcal{X}} \Pi_{i=1}^{k+1} \pi(z_i)f_i \,d\varrho - \Pi_{i=1}^{k+1} \int_{\mathcal{X}} f_i \,d\varrho\Big|\\
 &\leq C_{k,\epsilon}\,\max_{1\le i\neq j\le k+1}\eta_\epsilon (\mathcal{S}, s_{z_i-z_j})^{\frac{1}{k|\mathcal{S}|}}\,
\Pi_{i=1}^k\norm{f_i}_{C^{\zeta_\epsilon(\mathcal{S}_1)}}.
\end{align*}
This completes the inductive step. Hence, by induction the claim \eqref{for:219} is proved.

\subsubsection{Proof of \eqref{for:247} and \eqref{for:150} of Theorem \ref{th:3}} \emph{Proof of \eqref{for:247}}:
Proceed as in Section~\ref{sec:27}. We recall that $D_{1,1}$, $D_{1,2}$, $D_2$ and $\{k+1\}$ form a partition of the set $\{1,\dots, k+1\}$. In particular, when $k=2$ (i.e., $n=3$), either $|D_{1,1}\cup D_2\cup\{k+1\}|=1$ or $|D_{1,2}|=1$, which implies that $\mathfrak{m}(\mathcal{F}_1)\mathfrak{m}(\mathcal{F}_2)=0$. Then the result follows from \eqref{for:146} for $k=2$.

\emph{Proof of \eqref{for:150}}: Assume $\mathcal{Z}\subseteq A$, so there exist $t_i\in\mathbb{R}^m$ with $z_i=\mathfrak{a}^{t_i}$ for $1\le i\le3$. Without loss of generality,
\begin{align*}
\eta_\epsilon (\mathcal{S}, \mathfrak{a}^{t_1-t_3})=\min_{1\le i,j\le 3}\eta_\epsilon (\mathcal{S}, \mathfrak{a}^{t_i-t_j}).
\end{align*}
There exists $1\leq i_0\leq l$ such that
\begin{align*}
 \eta_{\epsilon} (\mathcal{S}, \mathfrak{a}^{t_1-t_{3}})^{\frac{1}{l}}\geq \min\{\beta_{i_0}(\mathfrak{a}^{t_{1}-t_{3}})^{-(\gamma_{i_0}-\epsilon)},\,\beta_{i_0}(\mathfrak{a}^{t_{3}-t_{1}})^{-(\gamma_{i_0}-\epsilon)}\}.
\end{align*}
Without loss of generality, assume that
\begin{align}\label{for:147}
\eta_{\epsilon} (\mathcal{S}, \mathfrak{a}^{t_1-t_{3}})^{\frac{1}{l}}\geq \beta_{i_0}(\mathfrak{a}^{t_{3}-t_{1}})^{-(\gamma_{i_0}-\epsilon)}.
\end{align}
\emph{Note}. The above inequality implies that $\beta_{i_0}(\mathfrak{a}^{t_{3}-t_{1}})>1$.
\begin{enumerate}
  \item If $\beta_{i_0}(\mathfrak{a}^{t_{2}-t_{1}})<1$, set $\mathcal{F}_1:=(\pi(\mathfrak{a}^{t_{2}-t_{1}})f_2)\,f_1$ and
  \begin{align*}
 \Big|\int_{\mathcal{X}} \Pi_{i=1}^{3} \pi(z_i)f_i \,d\varrho \Big|=J_1:=\mathfrak{m}\big(\pi(\mathfrak{a}^{t_{3}-t_{ 1}})f_3,\,\mathcal{F}_1\big).
\end{align*}
  \item If $1\leq\beta_{i_0}(\mathfrak{a}^{t_{2}-t_{1}})<\beta_{i_0}(\mathfrak{a}^{t_{3}-t_{1}})^{\frac{1}{2}} $, let $\mathcal{F}_2=f_2(\pi(\mathfrak{a}^{t_{1}-t_{ 2}})f_1)$, then
      \begin{align*}
 \Big|\int_{\mathcal{X}} \Pi_{i=1}^{3} \pi(z_i)f_i \,d\varrho \Big|=J_2:=\mathfrak{m}\big(\pi(\mathfrak{a}^{t_{3}-t_{ 2}})f_3,\,\mathcal{F}_2\big).
\end{align*}
We note that in this case
\begin{align}\label{for:151}
    \beta_{i_0}(\mathfrak{a}^{t_3-t_{2}})&=\beta_{i_0}(\mathfrak{a}^{t_3-t_{1}})\beta_{i_0}(\mathfrak{a}^{t_1-t_{2}})
    >\beta_{i_0}(\mathfrak{a}^{t_3-t_{1}})  (\beta_{i_0}(\mathfrak{a}^{t_3-t_{1}}))^{-\frac{1}{2}}\notag\\
    &=(\beta_{i_0}(\mathfrak{a}^{t_3-t_{1}}))^{\frac{1}{2}}>1.
  \end{align}

  \item If $\beta_{i_0}(\mathfrak{a}^{t_{3}-t_{1}})^{\frac{1}{2}}\leq\beta_{i_0}(\mathfrak{a}^{t_{2}-t_{1}})\leq \beta_{i_0}(\mathfrak{a}^{t_{3}-t_{2}})$, let $\mathcal{F}_3=f_2(\pi(\mathfrak{a}^{t_{3}-t_{ 1}})f_3)$, then \begin{align*}
 \Big|\int_{\mathcal{X}} \Pi_{i=1}^{3} \pi(z_i)f_i \,d\varrho \Big|=J_3:=\mathfrak{m}\big(\pi(\mathfrak{a}^{t_{2}-t_{ 1}})\mathcal{F}_3,\,f_1\big).
\end{align*}
  \item If $\beta_{i_0}(\mathfrak{a}^{t_{3}-t_{1}})<\beta_{i_0}(\mathfrak{a}^{t_{2}-t_{1}})$, let $\mathcal{F}_4=(\pi(\mathfrak{a}^{t_{2}-t_{ 3}})f_2)f_3$, then
  \begin{align*}
 \Big|\int_{\mathcal{X}} \Pi_{i=1}^{3} \pi(z_i)f_i \,d\varrho \Big|=J_4:=\mathfrak{m}\big(\pi(\mathfrak{a}^{t_{3}-t_{ 1}})\mathcal{F}_4,\,f_1\big).
\end{align*}
We note that in this case
\begin{align}\label{for:152}
  \beta_{i_0}(\mathfrak{a}^{-(t_{2}-t_{ 3})})=\beta_{i_0}(\mathfrak{a}^{-(t_{2}-t_{ 1})})\beta_{i_0}(\mathfrak{a}^{-(t_{1}-t_{3})})<\beta_{i_0}(\mathfrak{a}^{-(t_{3}-t_{ 1})})\beta_{i_0}(\mathfrak{a}^{-(t_{1}-t_{3})})=1.
 \end{align}
\end{enumerate}
Let $\mathcal{S}_1=\{\beta_{i_0}\}$. Then  for any $a\in A$ (see Section \ref{for:222} and \eqref{for:244})
\begin{align}\label{for:148}
\eta_\epsilon(\mathcal{S}_1, a)=\big(\max\{\beta_{i_0}(a), \beta_{i_0}(a^{-1})\}\big)^{-(\gamma_{i_0}-\epsilon)}.
\end{align}
Similar to \eqref{for:232}, it follows from  Theorem \ref{th:7} for $\mathcal{S}=\mathcal{S}_1$ that:
\begin{sect} \emph{Estimate for $J_1$}:
\begin{align*}
 J_1&\leq C_\epsilon \eta_{\epsilon}(\mathcal{S}_1, \mathfrak{a}^{t_{3}-t_{ 1}})\norm{f_3}_{u(\mathcal{S}_1),\,C^{\zeta_\epsilon(\mathcal{S}_1)}}\,\big\|\mathcal{F}_1\big\|_{S(\mathcal{S}_1),\,C^{\zeta_\epsilon(\mathcal{S}_1)}} \\
 &\overset{\text{(*)}}{\leq}C_{\epsilon,1} \eta_{\epsilon}(\mathcal{S}_1, \mathfrak{a}^{t_{3}-t_{ 1}})\norm{f_3}_{C^{\zeta_\epsilon(\mathcal{S}_1)}}\,\big\|f_1\big\|_{C^{\zeta_\epsilon(\mathcal{S}_1)}}
 \big\|f_2\big\|_{C^{\zeta_\epsilon(\mathcal{S}_1)}}.
\end{align*}
Here in $(*)$ we use \eqref{for:147}  and the fact that for  $U_{\beta_{i_0}^{-1}}$ and $X_{\beta_{i_0}}$, which spans $\text{Lie}(S(\mathcal{S}_1))$ (see \eqref{for:228}), we have
  \begin{align*}
 \|\text{Ad}_{\mathfrak{a}^{-(t_{2}-t_{1})}}U_{\beta_{i_0}^{-1}}\|&=\beta_{i_0}(\mathfrak{a}^{t_{2}-t_{ 1}})< 1\quad\text{and}\quad\|\text{Ad}_{\mathfrak{a}^{-(t_{i}-t_{ j})}}X_{\beta_{i_0}}\|=1.
\end{align*}
  Then the estimate for $\mathcal{F}_1$ follows from \eqref{for:125} of Section \ref{sec:33}.

\end{sect}

\begin{sect} \emph{Estimate for $J_2$}:
\begin{align*}
  J_2&\leq C_\epsilon \eta_{\epsilon}(\mathcal{S}_1, \mathfrak{a}^{t_{3}-t_{ 2}})\norm{f_3}_{u(\mathcal{S}_1),\,C^{\zeta_\epsilon(\mathcal{S}_1)}}\,\big\|\mathcal{F}_2\big\|_{S(\mathcal{S}_1),\,C^{\zeta_\epsilon(\mathcal{S}_1)}}\\
 &\overset{\text{(**)}}{\leq}C_{\epsilon,1} \eta_{\epsilon} (\mathcal{S}, \mathfrak{a}^{t_3-t_{1}})^{\frac{1}{2}}\norm{f_3}_{C^{\zeta_\epsilon(\mathcal{S}_1)}}\,\big\|f_1\big\|_{C^{\zeta_\epsilon(\mathcal{S}_1)}}
 \big\|f_2\big\|_{C^{\zeta_\epsilon(\mathcal{S}_1)}}.
\end{align*}
Here in $(**)$ we note that
\begin{align*}
   \eta_{\epsilon} (\mathcal{S}, \mathfrak{a}^{t_3-t_{2}})\overset{\text{($x$)}}{\leq}\beta_{i_0}(\mathfrak{a}^{t_3-t_{2}})^{-(\gamma_{i_0}-\epsilon)}
   \overset{\text{($x_1$)}}{\leq}\beta_{i_0}(\mathfrak{a}^{t_3-t_{1}})^{-\frac{1}{2}(\gamma_{i_0}-\epsilon)}= \eta_{\epsilon} (\mathcal{S}, \mathfrak{a}^{t_3-t_{1}})^{\frac{1}{2}}.
  \end{align*}
Here in $(x)$ we use \eqref{for:151} and \eqref{for:148}; in $(x_1)$ we use \eqref{for:151}.

  We also have
  \begin{align*}
 \|\text{Ad}_{\mathfrak{a}^{-(t_{1}-t_{2})}}U_{\beta_{i_0}^{-1}}\|&=\beta_{i_0}(\mathfrak{a}^{t_{1}-t_{ 2}})\leq1\quad\text{and}\quad\|\text{Ad}_{\mathfrak{a}^{-(t_{i}-t_{ j})}}X_{\beta_{i_0}}\|=1.
\end{align*}
  Then the estimate for $\mathcal{F}_2$ follows from \eqref{for:125} of Section \ref{sec:33}.
\end{sect}
\begin{sect} \emph{Estimate for $J_3$}:
\begin{align*}
J_3&\leq C_\epsilon \eta_{\epsilon}(\mathcal{S}_1, \mathfrak{a}^{t_{2}-t_{ 1}})\norm{\mathcal{F}_3}_{u(\mathcal{S}_1),\,C^{\zeta_\epsilon(\mathcal{S}_1)}}\,\big\|f_1\big\|_{S(\mathcal{S}_1),\,C^{\zeta_\epsilon(\mathcal{S}_1)}}\\
 &\overset{\text{($\diamondsuit$)}}{\leq}C_{\epsilon,1} \eta_{\epsilon} (\mathcal{S}, \mathfrak{a}^{t_3-t_{1}})^{\frac{1}{2}}\norm{f_2}_{C^{\zeta_\epsilon(\mathcal{S}_1)}}\,\big\|f_3\big\|_{C^{\zeta_\epsilon(\mathcal{S}_1)}}
 \big\|f_1\big\|_{C^{\zeta_\epsilon(\mathcal{S}_1)}}.
\end{align*}
Here in $(\diamondsuit)$ by \eqref{for:148} we have
  \begin{align*}
   \eta_{\epsilon} (\mathcal{S}, \mathfrak{a}^{t_2-t_{1}})=\beta_{i_0}(\mathfrak{a}^{t_2-t_{1}})^{-(\gamma_{i_0}-\epsilon)}
   \leq\beta_{i_0}(\mathfrak{a}^{t_3-t_{2}})^{-\frac{1}{2}(\gamma_{i_0}-\epsilon)}= \eta_{\epsilon} (\mathcal{S}, \mathfrak{a}^{t_3-t_{1}})^{\frac{1}{2}}.
  \end{align*}
   Then for $U_{\beta_{i_0}}$, which spans $\text{Lie}(u(\mathcal{S}_1))$ (see \eqref{for:228}), we have
  \begin{align*}
 \|\text{Ad}_{\mathfrak{a}^{-(t_{3}-t_{1})}}U_{\beta_{i_0}}\|=\beta_{i_0}(\mathfrak{a}^{-(t_{3}-t_{ 1})})< 1.
\end{align*}
   Then the estimate for $\mathcal{F}_3$ follows from \eqref{for:125} of Section \ref{sec:33}.
\end{sect}
\begin{sect} \emph{Estimate for $J_4$}:
\begin{align*}
J_4&\leq C_\epsilon \eta_{\epsilon}(\mathcal{S}_1, \mathfrak{a}^{t_{3}-t_{ 1}})\norm{\mathcal{F}_4}_{u(\mathcal{S}_1),\,C^{\zeta_\epsilon(\mathcal{S}_1)}}\,\big\|f_1\big\|_{S(\mathcal{S}_1),\,C^{\zeta_\epsilon(\mathcal{S}_1)}}\\
&\overset{\text{($\heartsuit$)}}{\leq}C_{\epsilon,1} \eta_{\epsilon} (\mathcal{S}, \mathfrak{a}^{t_3-t_{1}})^{\frac{1}{2}}\norm{f_2}_{C^{\zeta_\epsilon(\mathcal{S}_1)}}\,\big\|f_3\big\|_{C^{\zeta_\epsilon(\mathcal{S}_1)}}
 \big\|f_1\big\|_{C^{\zeta_\epsilon(\mathcal{S}_1)}}.
 \end{align*}
 Here in $(\heartsuit)$ by using \eqref{for:152}, we have
 \begin{align*}
 \|\text{Ad}_{\mathfrak{a}^{-(t_{2}-t_{3})}}U_{\beta_{i_0}}\|=\beta_{i_0}(\mathfrak{a}^{-(t_{2}-t_{ 3})})< 1.
\end{align*}
   Then the estimate for $\mathcal{F}_4$ follows from \eqref{for:125} of Section \ref{sec:33}.
\end{sect}
Finally, by  \eqref{for:147} we have
\begin{align*}
 \eta_{\epsilon} (\mathcal{S}, \mathfrak{a}^{t_3-t_{1}})^{\frac{1}{2}}\leq \eta_{\epsilon} (\mathcal{S}, \mathfrak{a}^{t_1-t_{3}})^{\frac{1}{2l}}.
\end{align*}
Combining the four cases we get the result.
\subsubsection{Proof of \eqref{for:191} of Theorem \ref{th:3}}
The argument proving \eqref{for:203} in Theorem \ref{th:8} applies verbatim to the present case, so we omit the details.


\begin{thebibliography}{99}

\bibitem{BEG} M. Bj{\"o}rklund, M. Einsiedler, A. Gorodnik, Quantitative multiple mixing. J. Eur. Math. Soc. 22 (2020), no. 5, pp. 1475-1529

\bibitem{Borel}A. Borel and N. Wallach, Continuous cohomology, discrete subgroups, and representations of reductive
groups. Annals of Mathematics Studies, 94. Princeton University Press, Princeton, N.J.; University of
Tokyo Press, Tokyo, 1980.

\bibitem{Clark} P. L. Clark, Geometry of numbers with applications to number theory, (preprint),
http://alpha.math.uga.edu/~pete/geometryofnumbers.pdf.

\bibitem{Cowling} M. Cowling, Sur les coefficients des représentations unitaires des groupes de Lie simples, Analyse harmonique
sur les groupes de Lie (Sém., Nancy–Strasbourg 1976–1978), II, Springer, Berlin (1979), 132–178.

\bibitem{Cowling1} Cowling, Michael, Uffe Haagerup, and Roger Howe. ``Almost $L^2$ Matrix Coefficients."  J. Reine Angew. Math., vol. 387, 1988, pp. 97–110.


\bibitem{Da1} S. G. Dani, Kolmogorov automorphisms on homogeneous spaces. Amer. J. Math. 98 (1976), no. 1, 119-163.

\bibitem{Da2} S. G. Dani, Spectrum of an affine transformation. Duke Math. J. 44 (1977), no. 1, 129-155.







\bibitem{F} L. Flaminio and G. Forni, On the cohomological equation for nilflows, J. Mod. Dyn., 1 (2007), 37-60.

\bibitem{Forni}L. Flaminio, G. Forni. Invariant distributions and time averages for horocycle flows. Duke Math
J. 119 No. 3 (2003) 465-526.

\bibitem{flaminio2014effective}
Livio Flaminio and Giovanni Forni.
\newblock On effective equidistribution for higher step nilflows.
\newblock {\em arXiv preprint arXiv:1407.3640}, 2014.

\bibitem{flaminio2016effective}
Livio Flaminio, Giovanni Forni, and James Tanis.
\newblock Effective equidistribution of twisted horocycle flows and horocycle
  maps.
\newblock {\em Geometric and Functional Analysis}, 26(5):1359--1448, 2016.

\bibitem{Forni1} G. Forni. Ruelle resonances from cohomological equations https://arxiv.org/abs/2007.03116




\bibitem{lang2012sl2} S. Lang, $SL(2,\RR)$, Addison-Wesley, Reading, MA, 1975.

\bibitem{Li} J.-S. Li, The minimal decay of matrix coefficients for classical groups. Harmonic analysis in China, 146–169,
Math. Appl., 327, Kluwer Acad. Publ., Dordrecht, 1995.

\bibitem{Li1} 23. J.-S. Li and C.-B. Zhu, On the decay of matrix coefficients for exceptional groups. Math. Ann. 305 (1996),
no. 2, 249–270.

\bibitem{Hamilton} R. S. Hamilton, The inverse function theorem of Nash and Moser, Bull. Amer. Math. Soc.
(N.S.), 7 (1982), 65-222.

\bibitem{Howe} R. Howe, On a notion of rank for unitary representations of the classical groups. Harmonic analysis and
group representations, 223–331, Liguori, Naples, 1982.

\bibitem{HM} R. Howe, and C. Moore, Asymptotic properties of unitary representations, J. Funct. Anal. 32 (1979), 72–96.

\bibitem{tan} R. E. Howe and E. C. Tan, Non-Abelian Harmonic Analysis, Springer-Verlag, 1992.


\bibitem{L} D. A. Lind. Dynamical properties of quasihyperbolic toral automorphisms. \emph{Ergod. Th.  Dynam. Sys}. 2(1)
(1982), 49-68.

\bibitem{margulis1991discrete} G. A. Margulis, Discrete subgroups of semisimple Lie groups, Berlin Heidelberg New York,
Springer-Verlag, 1991.

\bibitem{mautner1950unitary} F. I. Mautner, Unitary representations of locally compact groups, II, Ann. of Math. (2) 52
(1950), 528-556.


\bibitem{Moore} C. Moore, Exponential decay of correlation coefficients for geodesic flows. Group representations, ergodic
theory, operator algebras, and mathematical physics (Berkeley, Calif., 1984), 163–181, Math. Sci. Res. Inst.
Publ., 6, Springer, New York, 1987.

\bibitem{mo} S. Mozes, Mixing of all orders of Lie groups actions. Invent. Math. 107 (1992), no. 2, 235-241; erratum:
Invent. Math. 119 (1995), no. 2, 399.






\bibitem{Kleinbock} D. Kleinbock and G. A. Margulis, Bounded orbits of
nonquasiunipotent flows on homogeneous spaces, Sinai's Moscow
Seminar on Dynamical Systems, 141.172, {\em AMS Transl}. Ser.
2, {\bf 171}, AMS  Providence, RI, 1996.


\bibitem{Knapp} Knapp, Anthony W. Representation Theory of Semisimple Groups: An Overview Based on Examples. Princeton University Press, 1986.







\bibitem{GR1} A. Gorodnik and R. Spatzier. Exponential mixing of nilmanifold automorphisms. JAMA 123, 355–396 (2014). https://doi.org/10.1007/s11854-014-0024-7.



  \bibitem{GR}  A. Gorodnik and R. Spatzier. Mixing properties of commuting nilmanifold automorphisms. \emph{Acta Math}.
215(1) (2015), 127-159.

\bibitem{Spatzier1}A. Katok, R. Spatzier, First cohomology of Anosov actions of higher rank abelian groups and applications to rigidity, Publications Math{\'e}matiques de l'Institut des Hautes {\'E}tudes Scientifiques, Volume 79 (1994), pp. 131-156.

\bibitem{Konstantoulas} I. Konstantoulas, Effective decay of multiple correlations in semidirect product actions. J. Mod. Dyn. 10
(2016), 81–111.

\bibitem{vw} K. Vinhage and Z. J. Wang, Local Rigidity of Higher Rank Homogeneous Abelian
Actions: a Complete Solution via the Geometric Method, Geom Dedicata (2018). https://doi.org/10.1007/s10711-018-0379-5.

\bibitem{BM} P. E. Blanksby and H. L. Montgomery, Algebraic integers near the unit circle, Acta Arith.
18 (1971), 355-369.


\bibitem{Oh1} H. Oh, Tempered subgroups and representations with minimal decay of matrix coefficients. Bull. Soc. Math.
France 126 (1998), no. 3, 355–380.



\bibitem{oh}H. Oh, Uniform pointwise bounds for matrix coefficients of unitary representations and applications to
Kazhdan constants. Duke Math. J. 113 (2002), no. 1, 133-192.



\bibitem{Robinson} D. W. Robinson, Elliptic Operators and Lie Groups, Oxford Mathematical Monographs, 1991.





\bibitem{Star} A. Starkov, Multiple mixing of homogeneous flows. Dokl. Akad. Nauk 333 (1993), no. 4, 442-445; translation
in Russian Acad. Sci. Dokl. Math. 48 (1994), no. 3, 573-578.

\bibitem{TV} B. Timoth\'ee, and P. P. Varj\'u. 2024. Exponential Multiple Mixing for Commuting Automorphisms of a Nilmanifold.  \emph{Ergodic Theory and Dynamical Systems} 44: 1729-1740

\bibitem{BV} B. Speh and T. N. Venkataramana, On the Restriction of Representations of $\mathrm{SL}(2,\mathbb{C})$ to $\mathrm{SL}(2,\mathbb{R})$, in Lie Groups and Lie Algebras: E. B. Dynkin's Seminar, Progress in Mathematics, vol. 169, Birkhäuser Boston, 1995, pp. 279–288.


\bibitem{Wallach} N. Wallach: Real Reductive groups I, Pure and Applied Math. 132, Academic Press,
Boston, MA (1988).

\bibitem{W1} Z. J. Wang, Cohomological equation and cocycle rigidity of parabolic actions in some higher-rank Lie groups, accepted by Geom. and Funct. Analysis, Volume 25, Issue 6, (2015), 1956-2020

\bibitem{Zhenqi2}Z. J. Wang, Uniform pointwise bounds for Matrix coefficients of unitary representations on
semidirect products, J. functional analysis, Volume 267, Issue 1, 2014, 15-79.



\bibitem{W5} Z. J. Wang, Local rigidity of weak or no hyperbolicity algebraic actions, J of the AMS, 2025 Volume 38, pp 1107-1191.

\bibitem{W6} Z. J. Wang, Local rigidity of partially hyperbolicity algebraic actions, submitted.

\bibitem{W2} Z. J. Wang, Multiple fractional cohomological equations and quantitative mixing on nilmanifolds, submitted.


\bibitem{Zimmer} R. J. Zimmer, {\em Ergodic theory and semisimple groups},
Birkh\"{a}user, Boston,  1984

\bibitem{howe-moore}R. Howe and C. C. Moore, Asymptotic properties of unitary representations, J. Func. Anal. 32
(1979), Kluwer Acad., 72-96.



\bibitem{Warner} G. Warner, Harmonic analysis on semi-simple Lie groups. I. Die Grundlehren der mathematischen Wissenschaften,
Band 188. Springer-Verlag, New York-Heidelberg, 1972.


\bibitem{Warner1} G. Warner, Harmonic analysis on semi-simple Lie groups. II. Die Grundlehren der mathematischen Wissenschaften,
Band 189. Springer-Verlag, New York-Heidelberg, 1972.



\end{thebibliography}
\end{document}